\RequirePackage{rotating}
\documentclass[11pt,reqno]{amsart}
\usepackage{amsfonts,amsmath,amssymb,amsbsy,amstext,amsthm,accents,enumerate,float,fullpage,verbatim, mathtools, url, microtype,eucal,bm,mathbbol, parskip, cases, rotating, ulem}
\usepackage[dvipsnames]{xcolor}
\usepackage{colortbl, diagbox}
\usepackage[shortlabels]{enumitem}
\usepackage[colorlinks=true,hyperindex, linkcolor=blue, pagebackref=false, citecolor=cyan]{hyperref}
\usepackage[alphabetic,lite,backrefs]{amsrefs}

\usepackage{lscape}
\usepackage{cleveref}
\usepackage{cancel} 

\usepackage{graphicx,pstricks}
\usepackage{graphics}
\usepackage{subcaption}
\usepackage{import}
\graphicspath{ {images/} }
\usepackage{subfiles}
\usepackage{float, pdfpages}

\usepackage[all,cmtip]{xy}
\usepackage{tikz}
\usetikzlibrary{arrows,decorations.markings, cd}
\tikzstyle arrowstyle=[scale=1]
\tikzstyle directed=[postaction={decorate,
decoration={markings,mark=at position .65 with {\arrow[arrowstyle]{stealth}}}}]
\usepackage{mathdots}
\usepackage{bbold} 
\usepackage{tikz-cd}   
\usepackage[shortlabels]{enumitem}
\usepackage{ytableau}
\usepackage[ruled,vlined]{algorithm2e}
\usepackage[all]{nowidow}
\usepackage{genyoungtabtikz}

\newtheorem{lemma}{Lemma}[section]
\newtheorem{theorem}[lemma]{Theorem}

\newtheorem{prop}[lemma]{Proposition}
\newtheorem{cor}[lemma]{Corollary}
\newtheorem{conj}[lemma]{Conjecture}

\newtheorem{question}[lemma]{Question}
\newtheorem{problem}[lemma]{Problem}

\theoremstyle{definition}
\newtheorem{definition}[lemma]{Definition}
\newtheorem{example}[lemma]{Example}

\newtheorem{remark}[lemma]{Remark}

\newcommand{\one}{\mathbb{1}}

\newcommand{\Hilb}{\operatorname{Hilb}}

\newcommand{\GG}{\mathcal{G}}

\newcommand{\DowlingGeometry}[2]{{Q_{#1}(#2)}}

\newcommand{\OM}{\mathcal{M}}

\newcommand{\Ext}{\operatorname{Ext}}

\newcommand{\im}{\operatorname{im}}

\newcommand{\Br}{\operatorname{Br}}

\newcommand{\End}{\operatorname{End}}

\newcommand{\Sym}{\operatorname{Sym}} 

\newcommand{\supp}{\operatorname{supp}}

\newcommand{\mons}{\operatorname{Mons}}
\newcommand{\braid}{\operatorname{Br}}
\newcommand{\BEZ}{\operatorname{BEZ}}

\newcommand{\multichoose}[2]{\ensuremath{\left(\kern-.3em\left(\genfrac{}{}{0pt}{}{#1}{#2}\right)\kern-.3em\right)}}
\newcommand{\qmultichoose}[2]{\ensuremath{\left[\kern-.3em\left[\genfrac{}{}{0pt}{}{#1}{#2}\right]\kern-.3em\right]_q}}
\newcommand{\qbinom}[2]{\ensuremath{\left[ \genfrac{}{}{0pt}{}{#1}{#2} \right]_q}}

\newcommand{\sgn}[2]{{\mathrm{sgn}_{{#1},{#2}}}}

\newcommand{\Whit}{\text{\bf Whit}}

\newcommand{\kk}{\mathbb k}
\renewcommand{\aa}{\mathbf a}

\newcommand{\mm}{\mathbf m}

\newcommand{\xx}{\mathbf x}
\newcommand{\yy}{\mathbf y}
\newcommand{\zz}{\mathbf z}

\renewcommand{\k}{\mathbb{k}}

\newcommand{\cA}{\mathcal{A}}

\newcommand{\cC}{\mathcal{C}}

\newcommand{\cF}{\mathcal{F}}
\newcommand{\cG}{\mathcal{G}}

\newcommand{\cL}{\mathcal{L}}
\newcommand{\cM}{\mathcal{M}}

\newcommand{\cS}{\mathcal{S}}

\newcommand{\cU}{\mathcal{U}} 		

\newcommand{\cX}{\mathcal{X}}

\newcommand{\CC}{\mathbb{C}}

\newcommand{\FF}{\mathbb{F}}

\newcommand{\NN}{\mathbb{N}}
\newcommand{\PP}{\mathbb{P}}

\newcommand{\RR}{\mathbb{R}}

\newcommand{\ZZ}{\mathbb{Z}}

\newcommand{\m}{\mathfrak{m}}

\renewcommand{\r}{\mathfrak{r}}

\newcommand{\Hom}{\text{Hom}}

\newcommand{\Aut}{\text{Aut}}
\newcommand{\Tor}{\text{Tor}}

\def\CC{{\mathbb C}}
\def\ZZ{{\mathbb Z}}
\def\NN{{\mathbb N}}
\def\RR{{\mathbb R}}

\def\PP{{\mathbb P}}

\def\FF{{\mathbb F}}

\def\schurfunctor{{\mathbb{S}}}

\def\specht{\mathcal{S}}

\newcommand{\OS}{\mathrm{OS}}
\newcommand{\VG}{\mathrm{VG}}
\newcommand{\ch}{\mathrm{ch}}
\newcommand{\Conf}{\mathrm{Conf}}
\newcommand{\Lie}{\mathrm{Lie}}
\newcommand{\symm}{\mathfrak{S}}
\newcommand{\cnkrep}{\mathfrak{c}}
\newcommand{\Snkrep}{\mathcal{S}}
\newcommand{\init}{\mathrm{in}}
\newcommand{\eq}{\mathrm{eq}}

\newcommand{\regrep}{\rho_{\mathrm{reg}}}
\newcommand{\halfregrep}{\rho_{\frac{1}{2}\mathrm{reg}}}

\usepackage{array}
\newcolumntype{H}{>{\setbox0=\hbox\bgroup}c<{\egroup}@{}}
\begin{document}

\title[Koszulity, supersolvability and Stirling Representations]{Koszulity, supersolvability and Stirling Representations}

\author{Ayah Almousa}
\address{University of South Carolina, Columbia SC 29208}
\email{aalmousa@mailbox.sc.edu}
\author{Victor Reiner}
\author{Sheila Sundaram}
\address{University of Minnesota - Twin Cities, Minneapolis MN 55455}
\email{reiner@umn.edu, shsund@umn.edu}

\date{\today}

\begin{abstract}
Supersolvable hyperplane arrangements and matroids are known to give rise to certain Koszul algebras, namely their
Orlik-Solomon algebras and graded Varchenko-Gel'fand algebras.
We explore how this interacts with group actions, particularly for the braid arrangement and the action
of the symmetric group, where the Hilbert functions of
the algebras and their Koszul duals are given by Stirling numbers of the first and second kinds, respectively.
The corresponding symmetric group representations exhibit
branching rules that interpret Stirling number recurrences,
which are shown to apply to all supersolvable arrangements.  They also enjoy representation stability properties that follow from Koszul duality.
\end{abstract}

\maketitle

\setcounter{tocdepth}{1}
\tableofcontents

\section{Introduction}
\label{sec: intro}
This paper was motivated by a connection between {\it Stirling numbers} and
{\it Koszul algebras}.
The (signless) {\it Stirling numbers of the first kind $c(n,k)$} and {\it Stirling numbers of the second kind $S(n,k)$} are centuries-old answers to certain counting problems:  $c(n,k)$ is the number of permutations $\{1,2,\ldots,n\}$ with $k$ cycles, while $S(n,k)$ is the number of set partitions of $\{1,2,\ldots,n\}$ with $k$ blocks. On the other hand, Koszul algebras $A$ and their Koszul dual algebras $A^!$ originated in work of Priddy \cite{Priddy} and Fr\"oberg \cite{Froberg-original} in the 1970s (see also B\u{a}rc\u{a}nescu and Manolache \cite{BarcanescuManolache79 , BarcanescuManolache81}), playing an important role in topology, and in homological and commutative algebra.  

The connection stems from a particular Koszul dual pair of graded $\kk$-algebras 
$A=\bigoplus_{d=0}^\infty A_d$
and
$A^!=\bigoplus_{d=0}^\infty A^!_d$,
described later, carrying actions of the symmetric group $\symm_n$.  Their {\it Hilbert series}
\begin{align}
    \Hilb(A,t):=\sum_{d=0}^\infty \dim_\kk A_d t^d 
             &= (1+t)(1+2t)\cdots (1+(n-1)t) \\
             &= \sum_{k=0}^{n-1}c(n,n-k) t^k,\\
    \Hilb(A^!,t):=\sum_{d=0}^\infty \dim_\kk A^!_d \, t^d 
             &= \frac{1}{(1-t)(1-2t)\cdots (1-(n-1)t)}\\
             &= \sum_{k=0}^\infty S((n-1)+k,n-1) \, t^k
\end{align}
re-interpret the Stirling numbers $c(n,k), S(n,k)$.

In fact, there are two different well-studied algebras that can play the role of the algebra $A$ above:  the {\it Orlik-Solomon algebra} $\OS(\Br_n)$, or the
{\it graded Varchenko-Gel'fand algebra} $\VG(\Br_n)$, associated to the matroid and
oriented matroid $\Br_n$ for the {\it braid arrangement on $n$ strands}, also
known as the type $A$ {\it reflection hyperplane arrangement}, or the {\it graphic arrangement} associated to the complete graph on $n$ vertices.
A great deal is known about the $\symm_n$-representations on the graded components $A_d$ for either one of these algebras $A=\OS(M),\VG(\OM)$, due to their importance in the topology of configuration spaces and in combinatorics. Their Koszul duals $A^!$ have seen less study from a combinatorial representation theory viewpoint, and were our original main interest.

A natural framework here turns out to be the combinatorial notion of {\it supersolvability}.  Well-known results show that the algebras $A=\OS(M), \VG(\OM)$
for supersolvable matroids $M$ and oriented matroids $\OM$ have {\it quadratic Gr\"obner basis} presentations, which then implies their Koszulity.

Sections~\ref{sec: koszul-algebras},~\ref{sec: grobner}, and~\ref{sec: supersolvable} give background for this story. 
Section~\ref{sec: koszul-algebras} is mainly a
review of basic theory of Koszul algebras carrying group actions,
although it contains one new observation on {\it branching rules} (Proposition~\ref{prop: general-branching}).
Section~\ref{sec: grobner} recalls notions from {\it noncommutative Gr\"obner bases}, along with special features of commutative or anti-commutative rings, connecting quadratic Gr\"obner bases with Koszulity.  
Section~\ref{sec: supersolvable} reviews {\it matroids}, {\it oriented matroids} and the notion of supersolvability.

Section~\ref{sec:OS-and-VG-rings} starts with a review of 
the well-studied anti-commutative 
Orlik-Solomon algebras $\OS(M)$ and their not quite
as well-studied commutative counterparts, the graded Varchenko-Gel'fand rings $\VG(\OM)$.  After 
recalling why both $A=\OS(M), \VG(\OM)$ are Koszul algebras whenever
$M,\OM$ are supersolvable, the first main result, Theorem~\ref{thm: shriek-presentations}, gives
an explicit (noncommutative) quadratic Gr\"obner basis presentation for their Koszul duals $A^!$.  In the case
of $A=\OS(M)$, the presentation for $A^!$ is consistent with Kohno's presentation \cites{kohno1983holonomy, Kohno} of the {\it holonomy Lie algebra} for the cohomology of the complement of a complex hyperplane arrangement;  in the case
of $A=\VG(\OM)$, the presentation for $A^!$ appears to be new.  An application of the presentation, Corollary~\ref{cor: equiv-hilb-factorizations},
gives a Koszul dual analogue of the fact that multiplication by the
sum of the variables $\sum_i x_i$ endows $A=\OS(M)$ with an (equivariant) exact chain complex structure: 
in the supersolvable case, right-multiplication by the sum of the dual variables $\sum_i y_i$ within $A^!=\OS(M)^!$ gives an (equivariant) injective self-map of degree one.

Section~\ref{sec: example-section} pauses to illustrate the foregoing theory on simple examples of supersolvable matroids, such as Boolean matroids and rank two matroids, including discussion of equivariant structure.

Section~\ref{sec: supersolvable-branching-rules} proves the next main result, Theorem~\ref{thm: supersolvable-branching-exact-sequences}, giving
branching rules for
$A=\OS(M),\VG(\OM)$ and their Koszul duals $A^!$,
in the form of short exact sequences that apply
whenever $M, \OM$ are supersolvable.  For braid matroids $\Br_n$, these short exact sequences re-interpret the two classical Stirling number recurrences:
\begin{equation}
    \label{Stirling-recurrences-in-intro}
    \begin{aligned}
c(n,k) &= (n-1)\cdot c(n-1,k) + c(n-1,k-1), \\
S(n,k) &= k\cdot S(n-1,k) + S(n-1,k-1).
\end{aligned}
\end{equation} 

Sections~\ref{sec: deviations}, \ref{sec: topology}, and \ref{sec: rep-stability} review more general theory of Koszul algebras $A$, particularly when $A$ is either anti-commutative (like $\OS(M)$) or commutative (like $\VG(\OM)$).  Section~\ref{sec: deviations}
recalls why the Koszul dual $A^!$ is the universal enveloping algebra for its Lie (super-)algebra of primitive elements, also known as its {\it homotopy Lie algebra}, and why the latter coincides in this setting with its own {\it linear strand}, the {\it holonomy Lie algebra}.
The Poincar\'e-Birkhoff-Witt Theorem for universal enveloping algebras
then leads to equivariant versions of results such as the {\it lower central series formula} in the anti-commutative case, and the theory of {\it acyclic closures} and {\it deviations} in the commutative case.
Section~\ref{sec: topology} briefly reviews 
the topological interpretations of Koszul duality,
and the interpretation of $\OS(M), \VG(\OM)$ in terms of the cohomology of complements of
subspace arrangements.
Section~\ref{sec: rep-stability} reviews Church and Farb's notion of
representation stability for $\symm_n$-representations \cite{ChurchFarb}.  It then proves two results on its interaction with Koszul duality (Corollaries~\ref{cor: shrieks-have-rep-stability-cor}, ~\ref{cor: primitives-inherit-rep-stability}) showing that after fixing $d$, representation stability for the $d^{th}$ graded components $\{A_d(n)\}_{n\geq 1}$ in a family of Koszul algebras implies the analogous representation stability for their Koszul duals $\{ A^!_d(n)\}_{n\geq 1}$,  along with a similar statement for their holonomy Lie algebras.

Finally, Section~\ref{sec: stirling-reps} returns
to the motivating example of the braid arrangement matroids $\Br_n$, examining the consequences of all the previous results for $\OS(\Br_n), \VG(\Br_n)$, including the aforementioned
branching rules re-interpreting the Stirling number recurrences, Corollary~\ref{cor: braid-matroid-branching}.  One surprise here is Theorem~\ref{thm: OS-perm-reps}, on the prevalence of {\it permutation} representations of $\symm_n$ among the homogeneous components $A^!_i$ of the Koszul dual $A^!$ when $A=\OS(\Br_n)$.

Section~\ref{sec: remarks-and-questions} collects some further remarks and questions. \Cref{appendix: tables} includes tables of data for the characters of the Stirling representations of the first and second kind for $\OS(\Br_n)$ and $\VG(\Br_n)$ and the primitives of their corresponding holonomy Lie algebras. In addition, the code at \cite{Almousa_StirlingRepresentations_2024} can also be used to generate more data.

\subsection*{Summary of main results} 
For the ease of the reader, we summarize below the main results and their applications to the type $A$ braid arrangement $\Br_n$.
\begin{itemize}
    \item \Cref{thm: shriek-presentations} provides an explicit noncommutative Gr\"{o}bner basis for Koszul duals of Orlik-Solomon and Varchenko-Gel'fand rings of supersolvable matroids.
    \begin{itemize}
        \item The discussion following \Cref{rem: Stirling-number-coincidence} explains the bijection between standard monomials for $\OS(\Br_n)$ and $\VG(\Br_n)$ and restricted growth functions.
    \end{itemize}
    \item \Cref{cor: equiv-hilb-factorizations} shows that for a supersolvable matroid $M$ or oriented matroid $\OM$, right-multiplication by the sum of the dual variables $\sum_i y_i$ in $A^! = \OS(M)^!$ gives a degree one injective self-map, and the sum of the squares of the dual variables $\sum_i y_i^2$ in $A^! = \VG(\OM)^!$ gives a degree two injective self-map. These maps are equivariant with respect to any group $G$ of automorphisms of $M, \cM$.
    \begin{itemize}
        \item We conjecture the existence of $\symm_n$-equivariant degree one injective self-maps for $\VG(\Br_n)$ in \Cref{conj: braid-matroid-equivariant-injectivity}. 
    \end{itemize}
    \item \Cref{prop: general-branching} shows that the graded pieces of an equivariant Koszul algebra $A$ satisfy branching rules of a certain form if and only if the corresponding graded pieces for $A^!$ do. \Cref{thm: supersolvable-branching-exact-sequences} gives short exact sequences for $\OS(M)$, $\VG(\OM)$ and their Koszul duals that lift such branching rules whenever one has supersolvable matroids.
    \begin{itemize}
        \item \Cref{cor: braid-matroid-branching} gives these branching rules for the Stirling representations, which lift the classical Stirling number recurrences \eqref{Stirling-recurrences-in-intro}.
    \end{itemize}
    \item \Cref{thm:holonomy-lie-algebra-presentation} gives a presentation for the holonomy Lie algebra of $\VG(\OM)$ for an arbitrary oriented matroid $\OM$. In the supersolvable case, this presentation is consistent with the Gr\"{o}bner basis for $\VG(\OM)^!$ from \Cref{thm: shriek-presentations}.
    \item \Cref{cor: shrieks-have-rep-stability-cor} shows that if a family of Koszul algebras $A(n)$ with actions by $\symm_n$ are representation stable, then so are their Koszul duals.
    \begin{itemize}
        \item \Cref{braid-matroid-shrieks-have-rep-stability-cor} applies this to show representation stability for $\OS(\braid_n)^!$ and $\VG(\braid_n)^!$. \Cref{conj: onset-stability-shrieks} conjectures that the bounds for the onset of stability given in \Cref{braid-matroid-shrieks-have-rep-stability-cor} are tight.
    \end{itemize}
    \item \Cref{cor: primitives-inherit-rep-stability} shows that  families of representation stable commutative or anti-commutative Koszul algebras $A(n)$ pass this
    representation stability to their holonomy Lie algebras $\cL(n)$.
    \begin{itemize}
        \item \Cref{cor: braid-matroid-primitives-are-rep-stable} states that this holds for the holonomy Lie algebras of $\OS(\braid_n)$ and $\VG(\braid_n)$. In \Cref{conj: onset-of-stability-primitives}, we conjecture that the onset of stability is at $2i$ for high enough $i$.
    \end{itemize}
    \item \Cref{thm: OS-perm-reps} summarizes several cases where $[\OS(\braid_n)_i^!]$ are permutation representations.
\end{itemize}

\section{Koszul algebras}
\label{sec: koszul-algebras}

We review here the definitions and properties of Koszul algebras.  Useful surveys and references are
Berglund \cite{Berglund}, Faber et al \cite{Faber}*{\S 2}, 
Fr\"oberg \cite{Froberg-survey},
Mazorchuk, Ovsienko and Stroppel \cite{MazorchukOvsienkoStroppel}, 
McCullough and Peeva \cite{mcculloughPeevaSurvey}*{\S ~8}, Polishchuk and Positselski \cite{PolishchukPositselski}, and Priddy \cite{Priddy}.

\subsection{Standard graded algebras and Koszul algebras}

Fix a field $\kk$ throughout this discussion.

\begin{definition}
({\it Standard graded $\kk$-algebras})
For $V$ a $\kk$-vector space with $\kk$-basis $x_1,\ldots,x_n$, let 
$$
T^i(V):=V^{\otimes i}:=\underbrace{V\otimes \cdots \otimes V}_{i \text{ tensor factors}},
$$
and define the {\it tensor algebra}
$
T_\kk(V)=\bigoplus_{i=0}^\infty T^i(V),
$
with concatenation product. We identify it with 
$$
T_\kk(V) \cong \kk\langle x_1,\ldots,x_n\rangle,
$$
the free associative $\kk$-algebra on $n$ letters.  It is a {\it graded} $\kk$-algebra, in which $T^i(V)$ is the $i^{th}$ homogeneous component, and is generated as an algebra in
degree $1$ by $V$, the span of $x_1,\ldots,x_n$.

A {\it standard graded (associative) $\kk$-algebra} is a graded quotient ring $A$ of $T_\kk(V)$, that is,
\begin{equation}
\label{standard-graded-algebra-presentation}
A=T_\kk(V)/I
\end{equation}
for some two-sided ideal $I \subset T_\kk(V)$ which is {\it homogeneous}:  
$I = \bigoplus_{i=0}^\infty I_i$
where $I_i:=I \cap T^i (V)$.  We will generally assume that the images of $x_1,\ldots,x_n$ within $A$ (which we still denote $x_1,\ldots,x_n$, abusing notation) are minimal generators for $A$ as a
$\kk$-algebra, or equivalently,
that $I=I_2 \oplus I_3 \oplus \cdots$. 
\end{definition}

\begin{definition}({\it Koszul algebras})
Given a standard graded $\kk$-algebra $A$, let $A_+:=\bigoplus_{i=1}^\infty A_i$,
and regard the field $\kk=A/A_+$ as the {\it trivial (graded, left-)$A$-module}, generated in degree $0$.

Call $A$ a {\it Koszul algebra} if the surjection $A \twoheadrightarrow \kk=A/A_+$ can be extended as the first step in a graded resolution of $\kk$ by free left $A$-modules, which is {\it linear} in the sense that it has this form:
\begin{equation}
\label{linear-resolution-of-k}
\begin{array}{ccrcccccccl}
0 \longleftarrow &\kk& {\longleftarrow}
&F_0& \overset{d_1}{\longleftarrow} &F_1& \overset{d_2}{\longleftarrow} &F_2& \overset{d_3}{
\longleftarrow} &F_3& \leftarrow \cdots\\
& & &\Vert& &\Vert& &\Vert& &\Vert& \\
 & & &A& &A(-1)^{\beta_1}& &A(-2)^{\beta_2}& &A(-3)^{\beta_3}&\\
\end{array}
\end{equation}
Here $F_i=A(-i)^{\beta_i}$ is a graded free left $A$-module of rank $\beta_i$, all of whose $A$-basis elements have been shifted to degree $i$, 
 that is $A(-i)_j := A_{j-i}$. 
Linearity of the above resolution is equivalent to saying that the matrices for the differentials $d_i: A_i \longrightarrow A_{i-1}$ in the resolution have only linear (degree one) entries, that is, all matrix entries lie in $A_1$.
\end{definition}

Koszulity of $A$ has strong consequences for its algebra presentation, and for the form of the resolution \eqref{linear-resolution-of-k}, related to the notion of
{\it quadratic algebras} and their {\it quadratic duals}.

\begin{definition}
({\it Quadratic algebras} and {\it quadratic duals}) Say that the standard graded $\kk$-algebra $A$ presented as
in \eqref{standard-graded-algebra-presentation} is a {\it quadratic algebra} if $I$ is generated as a two-sided ideal by 
$$
I_2=I \cap T^2(V)= I \cap (V \otimes V).
$$

For $A$ any quadratic algebra, presented as in 
\eqref{standard-graded-algebra-presentation}, one defines its {\it quadratic dual algebra} $A^!$ as follows.  Let $V^*$ have $\kk$-dual basis
$y_1,\ldots,y_n$ to the ordered $\kk$-basis $x_1,\ldots,x_n$ for $V$, so that the bilinear pairing
$V^* \times V \rightarrow \kk$ has $(y_i,x_j)=\delta_{ij}$.  Then $T^2(V^*)$ and $T^2(V)$ have dual $\kk$-bases
$$
\{y_i \otimes y_j\}_{1\le i,j \le n} 
\text{ and }
\{x_i \otimes x_j\}_{1\le i,j \le n}
$$
with respect to the bilinear pairing
$
T^2(V^*) \times T^2(V) \rightarrow \kk
$
defined by 
\begin{equation}
\label{bilinear-pairing-on-2-tensors}
(y \otimes y', x \otimes x')
:=(y,x) \cdot (y',x').
\end{equation}
Define $A^!$ as this
quadratic algebra quotient of the free associative algebra $T_\kk(V^*)=\kk\langle y_1,\ldots,y_n \rangle$:
$$
A^!:=T_\kk(V^*)/J
$$
where $J$ is the two-sided ideal generated by
$$
J_2=I_2^\perp = \{ p \in T^2(V^*): (p,q)=0 \text{ for all }q \in I_2 \}.
$$
\end{definition}

Note that this really is a duality, in the sense that 
$(A^!)^! \cong A$.

\begin{example}
\label{standard-example-exterior-symmetric}
A commutative polynomial ring $\kk[x_1,\ldots,x_n]$ is a quadratic $\kk$-algebra:
$$
A=\Sym(V)=\kk[x_1,\ldots,x_n] \cong
\kk\langle x_1,\ldots,x_n\rangle
/I
$$
where $I=(x_i x_j-x_j x_i)_{1 \leq i<j \leq n}.$
Its quadratic dual $A^!$ is 
the anti-commutative exterior algebra
$$
A^!=\wedge(V^*)=\wedge(y_1,\ldots,y_n)=
\kk\langle y_1,\ldots,y_n\rangle/J
$$
where 
$J=(y_i y_j + y_jy_i)_{1 \leq i<j \leq n}+(y_i^2)_{1 \leq i \leq n}.$
\end{example}

\subsection{Priddy's resolution and its consequences}

It is  not hard to show that Koszul algebras $A$ are always quadratic\footnote{Quadraticity is equivalent to having a partial linear resolution 
$
0 \leftarrow \kk \leftarrow A \leftarrow F_1\leftarrow F_2
$
up to homological degree $2$.}.
What is more remarkable is a result of Priddy \cite{Priddy},
using $A^!$ to construct a simple, explicit linear $A$-resolution of $\kk$ whenever $A$ is Koszul.  Before describing it, let us point out certain maps on $A$ and on the {\it graded $\kk$-dual} $(A^!)^*$.  The latter is defined to be the following graded $\kk$-vector subspace of
the usual dual $\Hom_\kk(A^!,\kk)$:
$$
(A^!)^*:=\bigoplus_{i=0}^\infty (A^!_i)^*.
$$
\begin{itemize}
\item For $x$ in $A_1$, the map on $A$ which right-multiplies by $x$, that is
$a \mapsto ax$, gives a left $A$-module map $A \longrightarrow A$, raising degree by one.
\item For $y$ in $A^!_1$, the map
precomposing $\varphi$ in $(A^!)^*$ with left-multiplication\footnote{This corrects a typo in the definition from \cite{mcculloughPeevaSurvey}, and agrees with \cite[\S2.3, pp.~25-27]{PolishchukPositselski}, \cite[Prop.~44]{MazorchukOvsienkoStroppel}.} by $y$, that is
$\varphi \mapsto \varphi.y$ where $(\varphi.y)(b):=\varphi(yb)$, gives a $\kk$-linear map $(A^!)^* \rightarrow (A^!)^*$, lowering degree by one. 
\item 
Combining these, any $x \otimes y$ in 
$A_1 \otimes A^!_1=V \otimes V^*$ gives rise to a (left $A$-module) map $A \otimes (A^!)^* \longrightarrow A \otimes (A^!)^*$ that sends $a  \otimes \varphi \longmapsto (x \otimes y).(a \otimes \varphi):=ax \otimes \varphi.y$.
\end{itemize}

\begin{theorem}(The Priddy resolution)
\label{thm:Priddy-resolution}
When $A$ is Koszul, 
the element $
c:=\sum_{j=1}^n x_j \otimes y_j 
$ 
in $A_1 \otimes A^!_1$
acting on $A \otimes_\kk (A^!)^*$ as
a left $A$-module map gives a linear resolution of $\kk$ as in
\eqref{linear-resolution-of-k}, 
$$
0 \longleftarrow \kk {\longleftarrow}
A \otimes_\kk (A^!_0)^* 
 \overset{d_1}{\longleftarrow} 
A \otimes_\kk (A^!_1)^* 
 \overset{d_2}{\longleftarrow} 
A \otimes_\kk (A^!_2)^*  
 \overset{d_3}{\longleftarrow} \cdots
$$
Its differential $d_i: A \otimes_\kk (A^!_i)^*  \overset{d_i}{\longrightarrow} A \otimes_\kk (A^!_{i-1})^*$ is given explicitly as follows:
\begin{equation}
\label{eq:Priddy-complex-differential}
a \otimes \varphi  \longmapsto  c.\left( a \otimes \varphi \right)
=\sum_{j=1}^n a x_j  \otimes \varphi.y_j.
\end{equation}
\end{theorem}

\begin{example}
\label{standard-example-exterior-symmetric-2}
    Continuing Example~\ref{standard-example-exterior-symmetric}, one can check that the Priddy resolution for $\kk$
    over $A=\kk[x_1,\ldots,x_n]=\Sym(V)$
    becomes the usual {\it Koszul resolution}
$$
0 \leftarrow \kk 
\leftarrow \Sym(V) \otimes_\kk \wedge^0(V)
\leftarrow \Sym(V) \otimes_\kk \wedge^1(V)
\leftarrow \Sym(V) \otimes_\kk \wedge^2(V)
\leftarrow \cdots
\leftarrow \Sym(V) \otimes_\kk \wedge^n(V)
\leftarrow 0,
$$
using that fact that $(A^!_i)^*=(\wedge^i(V^*))^* \cong \wedge^i(V)$.
\end{example}

We note some important consequences of Priddy's resolution.  Taking graded $\kk$-duals swaps the roles of $A$ and $A^!$ in the resolution.  Consequently, $A$ is Koszul if and only if $A^!$ is Koszul.  In this case, one calls $A^!$ the {\it Koszul dual algebra} of $A$. 
Priddy's resolution also has an important consequence for the {\it Hilbert series} of $A, A^!$:
$$
\begin{aligned}
\Hilb(A,t)&:=\sum_{i=0}^\infty \dim_\kk A_i t^i,\\
\Hilb(A^!,t)&:=\sum_{i=0}^\infty \dim_\kk A^!_i t^i
=\sum_{i=0}^\infty \dim_\kk (A^!_i)^* t^i
=\Hilb((A^!)^*,t).
\end{aligned}
$$

\begin{cor}
\label{cor:Koszul-dual-Hilbert-series}
    Whenever $A, A^!$ are Koszul, one has
    $
    \Hilb(A,t) \cdot \Hilb(A^!,-t)=1.
    $
\end{cor}
\begin{proof}
For each degree $d \geq 1$, taking the coefficient of $t^d$ on both sides in the corollary gives the identity
$$
\sum_{i=0}^d  (-1)^i \dim_\kk A_{d-i} \cdot \dim_\kk (A^!_i)^*  = 0
$$
asserting vanishing of Euler characteristic for the (exact) $d^{th}$ 
graded component in Priddy's resolution 
\begin{equation}
\label{exact-strand-of-Priddy-resolution}
0 \rightarrow A_d \otimes_\kk (A^!_0)^* 
\rightarrow A_{d-1} \otimes_\kk (A^!_1)^* \rightarrow \cdots
\rightarrow A_1 \otimes_\kk (A^!_{d-1})^*
\rightarrow A_0 \otimes_\kk (A^!_d)^*
\rightarrow 0. 
\end{equation}
\end{proof}

\begin{example}
\label{standard-example-exterior-symmetric-3}
For the pair of Koszul dual algebras
\begin{align*}
A&=\Sym(V)=\kk[x_1,\ldots,x_n],\\
A^!&=\wedge(V^*)=\wedge(y_1,\ldots,y_n),
\end{align*}
one has these Hilbert series
\begin{align*}
\Hilb(A,t)&=\frac{1}{(1-t)^n}
\quad \text{ with }
\dim_\kk A_i=\multichoose{n}{i}:=\binom{n+i-1}{i},\\
\Hilb(A^!,t)&=(1+t)^n
\quad \text{ with }
\dim_\kk A^!_i=\binom{n}{i}.
\end{align*}
\end{example}

\begin{example}
({\it Noncommutative monomial Koszul algebras})
When a 
two-sided ideal
$I$
inside 
$T(V)= \kk\langle x_1,\ldots,x_n\rangle
$
is generated by a subset of noncommutative monomials, it is called a 
{\it monomial ideal}.  It is called a {\it quadratic monomial ideal} if the generating monomials are quadratic,
that is, they form a subset of the $n^2$ monomials
 $\{x_i x_j: (i,j) \subseteq [n] \times [n]\}$.
Starting with any quadratic monomial ideal $I$,
one can associate two complementary binary relations $D, D^c \subseteq [n] \times [n]$:
\begin{align*}
D&:=\{(i,j) \in [n] \times [n]: x_i x_j \not\in I\},\\
D^c&:=\{(i,j) \in [n] \times [n]: x_i x_j \in I\}.
\end{align*}
In this setting, denote the ideal $I$ by $I_D$, and denote the 
quotient algebra 
$
A_D:=T(V)/I_D.
$
One can view 
$D$ as a choice of a {\it directed graph} on vertex set $[n]$ having no repeated directed arcs $i \rightarrow j$, but allowing (single) copies of loops $i \rightarrow i$
and (single) pairs of antiparallel arcs $i \rightarrow j$ and $j \rightarrow i$.  Then the $d^{th}$ homogeneous component $(A_D)_d$ of $A_D$ has a $\kk$-basis indexed by the monomials $x_{i_1} x_{i_2} \cdots x_{i_d}$ whose subscripts $(i_1,i_2,\ldots,i_d)$ correspond to walks with $d-1$ steps along arcs $i_j \rightarrow i_{j+1}$ in
the digraph $D$. Hence
$
\Hilb(A_D,t)=1+\sum_{d=1}^{\infty} a_{D}(d) t^d,
$
where $a_D(d)$ is the number of such walks.
 
It turns out that these (noncommutative) {\it quadratic monomial $\kk$-algebras} are always Koszul.  A linear resolution of $\kk$ over $A_D$ 
is a special case of a resolution constructed by Fr\"oberg in \cite{Froberg-original}, and was also described recursively by 
Bruns, Herzog and Vetter \cite{BrunsHerzogVetter}*{\S  3}; we review the latter construction here.  Note that the quadratic dual $A_D^!$ has the form 
$$
A_D^!=T(V^*)/J_{D^c} \quad \text{ where } \quad
J_{D^c}=(y_i y_j: (i,j) \in D)=((I_D)_2^\perp).
$$ 
Letting $A:=A_D$, the linear $A$-free resolution 
$
0 \leftarrow \kk \leftarrow F_0 \leftarrow F_1 \leftarrow \cdots
$
described recursively in \cite{BrunsHerzogVetter} has $F_d$ being a free left $A$-module whose $A$-basis elements 
$
\{ e_{(i_1,\ldots,i_d)} \}
$
are indexed by all walks 
$(i_1,i_2,\ldots,i_d)$
taking $d-1$ steps along arcs
$i_j \rightarrow i_{j+1}$ in the complement $D^c$.
Unraveling their recursion,  the resolution
has these $A$-linear differentials:
\begin{equation}
\label{eq:BHV-differential}
e_{(i_1,i_2,\ldots,i_d)} \longmapsto 
x_{i_1} \, e_{(i_2,\ldots,i_{d})}. 
\end{equation}
Note that one has an isomorphism of free $A$-modules 
\begin{equation}
\label{eq:BHV-to-Priddy-dictionary}
\begin{array}{rcl}
F_d & \longrightarrow& A \otimes_\kk (A^!_d)^*\\
a \, e_{(i_1,\ldots,i_d)} 
 &\longmapsto &a \otimes [y_{i_1} \cdots y_{i_d}]^*\\
\end{array}
\end{equation}
where 
$[y_{i_1} \cdots y_{i_d}]^* \in (A^!_d)^*$ is the $\kk$-linear functional
$A^!_d \rightarrow \kk$ sending $y_{i_1} \cdots y_{i_d}$ to $1$ and sending all other degree $d$ monomials to $0$.  One can check that the definitions preceding Theorem~\ref{thm:Priddy-resolution} imply
$$
 [y_{i_1} y_{i_2}\cdots y_{i_d}]^*.y_j
=
\begin{cases}
[y_{i_2} \cdots y_{i_{d}}]^*
&\text{ if } j=i_1\\
0&\text{ otherwise.}
\end{cases}
$$
One therefore concludes that the differential in Priddy's resolution is the $A$-linear map sending
$$
1 \otimes [y_{i_1} y_{i_2} \cdots y_{i_d}]^*
\longmapsto
\sum_{j=1}^n x_j \otimes
\left( [y_{i_1} y_{i_2} \cdots y_{i_d}]^* . y_j\right)
=x_{i_1} \otimes [y_{i_2} \cdots y_{i_{d}}]^*.
$$
This agrees with the differential described by \eqref{eq:BHV-differential} after passing through the isomorphism \eqref{eq:BHV-to-Priddy-dictionary}.

Note that since $A_D$ is Koszul, and $A_D^! \cong A_{D^c}$, one has
$
\Hilb(A_D,t) \cdot \Hilb(A_{D^c},-t)=1,
$
an identity which appeared earlier in work of Brenti \cite{Brenti}*{\S  7.5}.
\end{example}

Our goal is to study Koszul algebras $A$ together with symmetries coming from
a finite group $G$ of graded ring automorphisms.  We will regard each graded component $A_i$ and $A^!_i$ as representations of $G$, or equivalently, as
$\kk G$-modules.  In order to work over arbitrary fields $\kk$ where $\kk G$ might not be semisimple, we introduce the {\it Grothendieck ring} $R_\kk(G)$.

\begin{definition} ({\it Grothendieck ring}) 
\label{Grothendieck-ring-defn}
As a $\ZZ$-module, the {\it Grothendieck group}
of $\kk G$-modules $R_\kk(G)$ is a quotient
 of the free $\ZZ$-module whose basis is the set of
 isomorphism classes $[V]$ of finite-dimensional $\kk G$-modules $V$, and where one mods out by the $\ZZ$-span of these relations: 
\begin{equation}
\label{Grothendieck-group-relations}
    \{\,\, [V]-([U]+[W]) : \text{ for all } \kk G\text{-module short exact sequences }0\rightarrow U \rightarrow V \rightarrow W \rightarrow 0 \,\,\} 
\end{equation}
In particular, in $R_\kk(G)$ one has $[U \oplus W]=[U]+[W]$.  Multiplication in $R_\kk(G)$ is induced by the rule $[V]\cdot [W]:=[V \otimes_\kk W]$, which one can check is consistent with the relations in \eqref{Grothendieck-group-relations}.
\end{definition}
We collect here a few standard facts about $R_\kk(G)$, omitting the proofs.

\begin{prop}
\label{easy-facts-on-Grothendieck-ring-prop}
For any finite group $G$, one has the following.
\begin{itemize}
\item[(i)]
The relations in $R_\kk(G)$ imply $\sum_{i=0}^\ell (-1)^i [V_i]=0$ for longer exact sequences of $\kk G$-modules
$$
0 \leftarrow V_0 \leftarrow V_1 \leftarrow \cdots \leftarrow V_\ell \leftarrow 0.
$$
\item[(ii)]
More generally, a finite $\kk G$-module complex
$0\leftarrow C_0 \overset{\partial}{\leftarrow} \cdots \overset{\partial}{\leftarrow} C_\ell \leftarrow 0$ with homology $\{H_*\}$ gives an Euler-Poincar\'e-Hopf-Lefschetz relation 
$
\sum_{i=0}^\ell (-1)^i [C_i]=\sum_{i=0}^\ell (-1)^i [H_i]
$
in $R_\kk(G)$.
\item[(iii)] Short exact sequences $0\rightarrow U \rightarrow V \rightarrow W \rightarrow 0$ of $\kk G$-modules lead to {\it dual/contragredient}
exact sequences $0\rightarrow W^* \rightarrow V^* \rightarrow U^* \rightarrow 0$, and also
$(U\otimes V)^* \cong U^* \otimes V^*$.  Hence the involution $[U] \mapsto [U^*]$ induces a well-defined involutive ring automorphism  $(-)^*: R_\kk(G) \rightarrow R_\kk(G)$.
\item[(iv)] For subgroups $H$ of $G$, the map $[U] \mapsto [U\downarrow^G_H]$, where $U\downarrow^G_H$ is the {\it restriction} of the $\kk G$-module $U$ to
a $\kk H$-module, induces a well-defined ring map 
$(-)\downarrow: R_\kk(G) \rightarrow R_\kk(H)$.
\item[(v)]
Since $(U^*)\downarrow^G_H \cong \left( U\downarrow^G_H \right)^*$ as $\kk H$-modules,
the maps in (iii),(iv) commute. 
\end{itemize}
\end{prop}

\begin{remark}\label{rem:explain-G-action-A-to-Ashriek}
We explain here why a group $G$ acting on a Koszul algebra $A$ also acts on the Koszul dual $A^!$.
When a standard graded $\kk$-algebra $A=T(V)/I$ carries the action of a group $G$
of graded $\kk$-algebra automorphisms, the fact
that $G$ preserves $A_1=V$, and $A_1$ generates $A$, implies that one can regard $G$ as a subgroup of $GL(V)$, possibly replacing $G$ by $G/K$ if $K$ is the kernel of its action on $V$.
Then $G$ also acts contragrediently on $V^*$,
via $\varphi \mapsto \varphi \circ g^{-1}$.
This gives the natural $\kk$-bilinear pairing $V^* \otimes V \rightarrow \kk$ defined by $\varphi \otimes v \mapsto \varphi(v)$ a certain 
$G$-invariance:  
$$
g(\varphi \otimes v) = (\varphi \circ g^{-1}) \otimes g(v) \mapsto \varphi (g^{-1}(g(v)))=\varphi(v).
$$
The dual pairing \eqref{bilinear-pairing-on-2-tensors} between $T^2(V^*)$
and $T^2(V)$ then inherits this same $G$-invariance.

Consequently, when $A=T(V)/I$ is a quadratic algebra with the action of a group $G$ preserving
the subspace $I_2 \subset T^2(V)$ that generates the ideal $I$, then $G$ also preserves the subspace $J_2=I_2^\perp$ that generates the ideal $J$ defining the quadratic dual $A^!=T(V^*)/J$.  Thus $G$ also acts on $A^!$.
\end{remark}

The following proposition should not be surprising.

\begin{prop}
When $A, A^!$ are Koszul, the Priddy resolution is $G$-equivariant for any
group of graded $\kk$-algebra automorphisms acting on $A$ (and hence on $A^!$).
\end{prop}
\begin{proof}
   This follows because the differential acts by
   $c=\sum_{j=1}^n x_j \otimes y_j$ in $A_1 \otimes A^!_1=V \otimes V^*$, and $c$ is $G$-fixed: 
 under the $G$-equivariant isomorphism $V \otimes V^* \cong \End_\kk(V)$ that sends $v \otimes f$ to $\varphi:V \rightarrow V$ given by $\varphi(w)=f(w) \cdot v$, one can check that $c \mapsto 1_V$, which is a $G$-fixed element of $\End_\kk(V)$. 
\end{proof}

This gives a version of Corollary~\ref{cor:Koszul-dual-Hilbert-series},
regarding the {\it equivariant Hilbert series} in $R_\kk(G)[[t]]$
\begin{equation}
\label{generic-equivariant-Hilbert-series}
\Hilb_\eq(A,t):=\sum_{i=0}^\infty [A_i] t^i.
\end{equation}

\begin{cor}(cf. \cite[Prop.~8.1]{JosuatvergesNadeau})
\label{cor:equivariant-Koszul-dual-Hilbert-series}
   Let $A, A^!$ be Koszul dual algebras, both with the action of a group $G$ of graded $\kk$-algebra automorphisms.  Then one has this identity in  $R_\kk(G)[[t]]$:
   \begin{equation}
   \label{reciprocal-equivariant-Hilbert-series}
   \Hilb_\eq(A,t) \cdot \Hilb_\eq((A^!)^*,-t)=1
   \end{equation}
   Equivalently, $[A_0]=[(A^!_0)^*]=[\one_G]$ and 
   one has these identities in $R_\kk(G)$ for $d\ge 1$:
   \begin{equation}
       \label{equivariant-Priddy-strand-identity}
   \sum_{i=0}^d (-1)^i [A_{d-i}] \cdot [(A^!_i)^*] = 0
   \end{equation}
   which can be rewritten as this recurrence for $[(A^!_d)^*]$:
   \begin{equation}
   \label{equivariant-recurrence-for-shrieks}
   [(A^!_d)^*]=\sum_{i=1}^{d} (-1)^{i-1} [A_i] \cdot [(A^!_{d-i})^*]
   \end{equation}
   and this unraveled formula:
   \begin{equation}
\label{unraveled-shriek-recurrence}
[(A^!_d)^*] =
\sum_{\substack{\alpha=(\alpha_1,\ldots,\alpha_\ell):\\
\alpha_1+\cdots+\alpha_\ell=d}} (-1)^{d-\ell} [A_{\alpha_1}]
[A_{\alpha_2}]\cdots [A_{\alpha_\ell}].
\end{equation}
This last sum runs over all (strict) ordered compositions $\alpha=(\alpha_1,\ldots,\alpha_\ell)$ of $d$,
of any length $\ell \geq 1$,
that is, $\alpha_i$ are positive integers summing to $d$.
   
\end{cor}
\begin{proof}
It suffices to prove \eqref{equivariant-Priddy-strand-identity}, which follows from the $G$-equivariance and exactness of \eqref{exact-strand-of-Priddy-resolution} .
\end{proof}

\begin{example}
\label{standard-example-exterior-symmetric-4}
Continuing Examples~\ref{standard-example-exterior-symmetric}, \ref{standard-example-exterior-symmetric-2}, \ref{standard-example-exterior-symmetric-3},
the Koszul algebras
$A=\Sym(V), \, A^!=\wedge(V^*)$ carry the action of $G=GL(V)$. 
There is a ring homomorphism from $R_\kk(G)$ to the ring $$
\Lambda_\kk(\zz):=\Lambda_\kk(z_1,\ldots,z_n)=\kk[z_1,\ldots,z_n]^{\symm_n}
$$
of symmetric polynomials in $n$ variables with $\kk$ coefficients, mapping the class $[U]$
of a $\kk G$-module $U$ to $\mathrm{trace}(g|_U)$ where $g=\mathrm{diag}(z_1,\ldots,z_n)$ in $GL(V)$ is the diagonal matrix in $GL(V)$ having $g(x_i)=z_i \cdot x_i$ in $V$ for $i=1,2,\ldots,n$,
so that $g(y_i)=z_i^{-1} \cdot y_i$
in $V^*$.

Applying this homomorphism to \eqref{reciprocal-equivariant-Hilbert-series} gives a standard identity 
$H(t) E(-t) = 1$ in $\Lambda_\kk(\zz)[[t]]$,
where 
\begin{align*}
H(t):=\sum_{k=0}^\infty h_k(z_1,\ldots,z_n) t^k=\prod_{j=1}^n \frac{1}{1-z_j t},\\
E(t):=\sum_{k=0}^n e_k(z_1,\ldots,z_n) t^k
=\prod_{j=1}^n (1+z_j t).
\end{align*}
This can be viewed as the specialization
of a well-known identity in the ring of symmetric
functions in infinitely many variables $\Lambda:=\Lambda_\ZZ(z_1,z_2,\ldots)$ with integer coefficients, relating the two sets of
algebraically independent generators
$\{h_1,h_2,\ldots\}$ and $\{e_1,e_2,\ldots\}$; see \cite{Macdonald}*{Chap.~1, eqn.~(2.6)}, \cite{StanleyEC2}*{Thm.~7.6.1}.  Rewritten as in
\eqref{equivariant-recurrence-for-shrieks}, one has 
$
e_0=h_0=1
$
and
$
e_d=\sum_{i=1}^{d} (-1)^{i-1} h_i \cdot e_{d-i}
$
for all 
$d \geq 1.$
Due to their algebraic independence, any symmetric function identities in $\Lambda$ among $\{h_i\}, \{e_i\}$ lead to the same identities relating $\{[A_1],[A_2],\ldots\},  
\{[(A^!_1)^*],[(A^!_2)^*],\ldots\}$ in $R_\kk(G)$ for any Koszul algebra $A$ over any field $\kk$.  For example,
a special case of the {\it Jacobi-Trudi identity} 
\cite[Chap.~1, eqn.~(3.4)]{Macdonald}, \cite[Thm.~7.16.1]{StanleyEC2} expresses the 
$\{e_k\}$ in terms of the $\{h_k\}$:
$$
e_d = \det\left[ 
\begin{matrix}
    h_1 & h_2 & h_3 & \cdots & & \\
    1   & h_1 & h_2 & \cdots & & \\
    0   & 1   & h_1 & \cdots & &\\
    0   & 0   &  1  &       & & \\
    \vdots&\vdots& &\ddots & \ddots\\
    0   &  0  &  0  & \cdots &1 &h_1
\end{matrix}
\right]
= \sum_{\alpha=(\alpha_1,\ldots,\alpha_\ell)} (-1)^{d-\ell} h_{\alpha_1} h_{\alpha_2} \cdots h_{\alpha_\ell},
$$
where $\alpha$ runs over all compositions of $d$.  
One now recovers the unraveled formula~\eqref{unraveled-shriek-recurrence} for $[(A^!_d)^*]$.
\end{example}

\subsection{A Koszul branching relation}
We wish to lift several combinatorial recurrences to {\it branching rules} for
Koszul algebras $A$ and their Koszul duals $A^!$.
Recall from Proposition~\ref{easy-facts-on-Grothendieck-ring-prop}(iv) that for any subgroup $H$ of a group $G$, the map $[U] \mapsto [U\downarrow^G_H]$ induces a ring map 
$(-)\downarrow: R_\kk(G) \rightarrow R_\kk(H)$.

\begin{prop}
\label{prop: general-branching}
    Let $A,B$ be two Koszul $\kk$-algebras, with actions of groups $G, H$, where $H$ is a subgroup of $G$, and let $\mathcal X$ be a $\kk H$-module.
    Then in $R_\kk(H)$, one has
    \begin{align*}
    &[A_i\downarrow] = [B_i]+[{\mathcal X}] \cdot [B_{i-1}]\\
    \text{ if and only if } \quad &
    [(A^!_i)^*\downarrow] = [(B^!_i)^*]+[{\mathcal X}]\cdot \left( [(A^!_{i-1})^*\downarrow] \right)\\
\text{ if and only if } \quad 
    &[A^!_i\downarrow] = [B^!_i]+[{\mathcal X}^*]\cdot \left( [A^!_{i-1}\downarrow] \right)
\end{align*}
\end{prop}

\begin{proof} The last equivalence uses the properties of the ring automorphism $(-)^*: R_\kk(G) \rightarrow R_\kk(G)$ from Proposition~\ref{easy-facts-on-Grothendieck-ring-prop}(iii),(iv),(v).  Hence it suffices to prove the first equivalence. 

   Introduce a few abbreviated notations
   \begin{align*}
       a_i&:=[A_i] \text{ and } a^{!*}_i:=[(A^!_i)^*] \text{ in }R_\kk(G),\\
       b_i&:=[B_i] \text{ and } b^{!*}_i=[(B^!_i)^*] \text{ in }R_\kk(H),\\
       \bar{a}_i&:=[A_i\downarrow] \text{ and } \bar{a}^{!*}_i:=[(A^!_i)^*\downarrow] \text{ in }R_\kk(H),\\
       x&:=[{\mathcal X}]\text{ in }R_\kk(H)
    \end{align*}
    along with analogous
     generating functions in $R_\kk(G)[[t]]$ and $R_\kk(H)[[t]]$, such as $a(t):=\sum_i a_i t^i$, and similarly $b(t),a^{!*}(t),b^{!*}(t),\bar{a}(t)$.
In this notation, the first equivalence of the proposition asserts 
$$
\bar{a}_i = b_i+x b_{i-1} \Leftrightarrow
\bar{a}^{!*}_i = b^{!*}_i + x \bar{a}^{!*}_{i-1}.
$$
Note that one has these three relations, coming
from Corollary~\ref{cor:equivariant-Koszul-dual-Hilbert-series} for the Koszul algebras $A, B$, and applying the ring map $(-)\downarrow$ to the first relation:
\begin{align*}
a^{!*}(t) a(-t)&=1\\
b^{!*}(t) b(-t)&=1\\
\bar{a}^{!*}(t) \bar{a}(-t)&=1
\end{align*}
This lets one compute as follows:
\begin{align*}
&\bar{a}_i = b_i+x b_{i-1} \\
\Leftrightarrow \quad &\bar{a}(t) = (1+xt)\cdot b(t)\\
 \Leftrightarrow \quad &\frac{1}{\bar{a}(-t)} = \frac{1}{1-xt} \cdot \frac{1}{b(-t)}\\
\Leftrightarrow \quad&\bar{a}^{!*}(t) = \frac{1}{1-xt} \cdot b^{!*}(t)   \\
 \Leftrightarrow \quad&(1-xt) \cdot \bar{a}^{!*}(t) = b^{!*}(t)\\
 \Leftrightarrow \quad&\bar{a}^{!*}_i- x \bar{a}^{!*}_{i-1} = b^!_i \\
 \Leftrightarrow \quad&\bar{a}^{!*}_i = b^{!*}_i + x \bar{a}^{!*}_{i-1}. \qedhere
\end{align*}
\end{proof}

\begin{example}
Continuing Example~\ref{standard-example-exterior-symmetric-3}, the symmetric group $G=\symm_n$ acts on the Koszul dual algebras
$A(n):=\kk[x_1,\ldots,x_n]=\Sym(V)$ and
$A(n)^!=\wedge(y_1,\ldots,y_n)=\wedge(V^*)$
by permuting variables. 
One can apply Proposition~\ref{prop: general-branching} with 
   $B=A(n-1)=\kk[x_1,\ldots,x_{n-1}],
    B^!=\wedge(y_1,\ldots,y_{n-1})$,
which are both $\kk H$-modules for $H=\symm_{n-1}$,
and with $\cX=\one_H$ the trivial $\kk H$-module. Recalling the notation $\multichoose{n}{i}:=\binom{n+i-1}{i}$,
one then sees that the proposition lifts the equivalence of these two versions of the 
Pascal recurrence 
\begin{align*}
\binom{n}{i}&=\binom{n-1}{i} +\binom{n-1}{i-1},\\
\multichoose{n}{i}&=\multichoose{n-1}{i} +\multichoose{n}{i-1},
\end{align*}
to an equivalence of statements on
restricting $A(n)_i, A(n)^!_i$ from $\symm_n$
to $\symm_{n-1}$: 

\begin{align*}
[A(n)^!_i\downarrow] &=[A(n-1)^!_i] + [A(n-1)^!_{i-1}],\\
[A(n)_i\downarrow] &=[A(n-1)_i] + [A(n)_{i-1}\downarrow].
\end{align*}

Both also follow from segregating the degree $i$ monomials in $\kk[x_1,\ldots,x_n]$ or $\wedge(y_1,\ldots,y_n)$, counted by the left sides, into monomials {\it not divisible} by the last variable $x_n, y_n$, versus those {\it divisible} by it.
\end{example}

\section{Review of noncommutative, commutative, exterior Gr\"obner bases}
\label{sec: grobner}
We review here some of the theory of Gr\"obner bases for two-sided ideals $I$ in noncommutative, commutative and exterior algebras over a field $\kk$, emphasizing aspects that are special to the 
situation where $I$ is homogeneous, and/or quadratic.  Useful references for the
\begin{itemize}
    \item commutative theory: Cox, Little and O'Shea \cite{CoxLittleOShea}, Adams and Loustaunau \cite{AdamsLoustaunau}, Eisenbud \cite[Ch.~15]{Eisenbud},
    \item exterior algebra theory: Aramova, Herzog and Hibi \cite{AramovaHerzogHibi}, Stokes \cite{Stokes},
    \item noncommutative theory:
    Bokut and Chen \cite{BokutChen},
    Mora \cite{Mora}, Ufnarovskii \cite[\S 2]{Ufnarovskij}, Polishchuk and Positselski \cite[Ch. 4]{PolishchukPositselski}, Shepler and Witherspoon \cite[\S  3]{SheplerWitherspoon}.
\end{itemize}

\subsection{Monomial orders, initial forms, and initial ideals}

Fix a positive integer $n$, and abbreviate
the free associative, commutative, and
exterior algebras $R$ in $n$ variables $z_1,\ldots,z_n$ as follows:
\begin{align*}
\kk\langle \zz \rangle &:= \kk \langle z_1,\ldots,z_n \rangle,\\
\kk[ \zz ] &:= \kk[ z_1,\ldots,z_n ],\\
\wedge( \zz ) &:= \wedge( z_1,\ldots,z_n).
\end{align*}
The set of monomials in each these rings $R$ will be denoted
\begin{align*}
\mons( \kk\langle \zz \rangle) 
&:=
\{ z_{i_1} z_{i_2} \cdots z_{i_\ell} : \ell \geq 0 \text{ and }(i_1,\ldots,i_\ell)\in [n]^\ell \}\\
\mons(\kk[ \zz ]) 
&:=\{ \zz^\aa= z_1^{a_1} z_2^{a_2} \cdots z_n^{a_n} : \aa=(a_1,\ldots,a_n) \in \NN^n\} \\
\mons(\wedge( \zz )) 
&:=\{ \zz_S= z_{i_1} \wedge z_{i_2} \wedge \cdots \wedge z_{i_\ell} : S=\{i_1 < i_2 < \cdots < i_\ell\} \subseteq [n]\}. 
\end{align*}

\begin{definition} (Monomial orders, initial forms, initial ideals)
    A linear ordering $\prec$ on 
$\mons(R)$ for any of the above three rings $R$ is called a {\it monomial ordering} if 
\begin{itemize}
    \item it is a well-ordering: there are no infinite descending chains $m_1 \succ m_2 \succ m_3 \succ \cdots$, and 
    \item whenever $m \prec m'$, then 
    $m_1m m_2 \prec m_1 m' m_2$ for any other monomials $m_1,m_2$.
    \end{itemize}
Having fixed a monomial order $\prec$ on one of these
rings $R=\kk\langle \zz \rangle, \kk[ \zz ], \wedge( \zz )$, write any ring element
as a finite $\kk$-linear sum of monomials $m$ with nonzero coefficients $c_m$ in $\kk$
\begin{align*}
f=\sum_{m \in \mons(R)} c_m m \,\,
=c_{m_0} \cdot m_0 + 
\sum_{\substack{m \in \mons(R):\\m \prec m_0}} c_m m
\end{align*}
and then define $m_0$ to be its unique {\it $\prec$-initial term} or {\it $\prec$-leading monomial},
denoted 
$
\init_\prec(f):=m_0.
$
Given a (two-sided) ideal $I \subset R$, define its {\it $\prec$-initial ideal} to be this two-sided monomial ideal of $R$:
$$
\init_\prec(I):=( \init_\prec(f): f \in I ).
$$
\end{definition}

\begin{definition}
Given a monomial order $\prec$ on one of $R=\kk\langle \zz \rangle, \kk[ \zz ], \wedge( \zz )$,
and a two-sided ideal $I \subset R$, one says that a  subset $\GG \subset I$ is a {\it Gr\"obner basis (GB) for $I$ with respect to $\prec$} if 
$$
\init_\prec(I) = ( \{ \init_\prec(g): g \in \GG\} )=: ( \init_\prec(\GG) ).
$$
Equivalently, every $f$ in $I$ has $\init_\prec(f)=m_0$ (left-right) divisible by at least one $\init_\prec(g)=m$ for some $g$ in $\GG$, meaning that $m_0=m_1 m m_2$ for some $m_1,m_2$ in $\mons(R)$. One calls a Gr\"obner basis $\GG$ {\it reduced}
if for each pair $g \neq g'$ in $\GG$, none of the monomials $m$ appearing in $g$ with nonzero coefficient are divisible by $\init_\prec(g')$.
\end{definition}

Gr\"obner bases for $I$ exist, but may need to 
be infinite when working in $R=\kk \langle \zz \rangle$.
For example, 
$\GG_0=I$ itself
always
gives a GB for $I$, but is infinite as long as $I \neq \{0\}$.
The fact that a GB for an ideal always generates the ideal will follow from a certain {\it division algorithm}. 

\begin{definition} ($\GG$-standard monomials and the division algorithm) Call a monomial $m$ in $\mons(R)$ a {\it $\GG$-standard monomial with respect to $\prec$} if it is (left-right) divisible by {\it none} of $\{ \init_\prec(g) : g \in \GG\}$.

The {\it division algorithm on $R$ with respect to $\GG$ and $\prec$} starts with any $f$ in $R$ and  produces a remainder $r$ having $f \equiv r \bmod I$ (and written $f \rightarrow_\GG r$) which is a $\kk$-linear combination of $\GG$-standard monomials, as follows.  Assuming $f=\sum_m c_m m$ contains any monomials which are {\it not} $\GG$-standard, pick the $\prec$-largest such monomial $m$, and write it as $m=m_1 m' m_2$ where $m'=\init_\prec(g)$ for some\footnote{Without loss of generality, assume that all $g$ in $\GG$ are $\prec$-monic, meaning that $\init_\prec(g)$ has coefficient $+1$ in $g$.} $g$ in $\GG$.  Then replace $f$ by 
$$
f':=f-c_m \cdot m_1 \cdot g \cdot m_2
$$
which has $f \equiv f' \bmod I$.  Repeat the process with $f'$.  One can show that, because $\prec$ is a well-ordering, this algorithm
will eventually terminate with a remainder $r$
that contains only $\GG$-standard monomials.
However, the remainder $r$ may not be unique, due to
choices of which element $g$ in $\GG$ has $m'=\init_\prec(g)$ dividing the non-$\GG$-standard
term $m$ of $f$ at each stage.
\end{definition}

The following equivalent conditions
defining Gr\"obner bases are
standard verifications.

\begin{prop}
\label{prop:various-Grobner-equivalents}
Fixing $\prec$ and the two-sided ideal $I \subset R$, the following are equivalent for $\GG \subset I$:
\begin{enumerate}
    \item[(i)]  $\GG$ is a GB for $I$ with respect to $\prec$.
    \item[(ii)]  The division algorithm $f \rightarrow_\GG r$ always gives the same remainder $r$ for $f$.
    \item[(iii)]  One has $f \in I$ if and only if 
    $f \rightarrow_\GG 0$, regardless of choices in the division algorithm.
    In particular, $\GG$ generates $I$.
    \item[(iv)]  The (images of the) $\GG$-standard monomials 
    with respect to $\prec$ give a $\kk$-basis
    for $R/I$.
\end{enumerate}
\end{prop}

The GB condition has a useful rephrasing for {\it homogeneous} ideals $I$, meaning $I=\bigoplus_{d=0}^\infty (I \cap R_d)$.

\begin{prop}
\label{prop: graded-GB-via-Hilb-trick}
For a homogeneous two-sided ideal $I \subset R$,
a subset $\GG \subset I$ forms a GB of $I$ 
with respect to $\prec$ if and only if
$
 \Hilb(S/(\init_\prec(\GG)),t) = \Hilb(S/I,t).
$
\end{prop}
\begin{proof}
By definition $\GG \subset I$ is a GB if and only if
the inclusion $(\init_\prec(\GG)) \subseteq \init_\prec(I)$
is an equality. This occurs if and only if the graded $\kk$-algebra surjection 
$
R/(\init_\prec(\GG)) \twoheadrightarrow R/\init_\prec(I)
$
is a $\kk$-vector space isomorphism in each degree.
By dimension-counting, this occurs if and only if
$$
\Hilb(S/(\init_\prec(\GG)),t) = \Hilb(S/\init_\prec(I),t)
$$
However one also has
$\Hilb(S/\init_\prec(I),t)=\Hilb(S/I,t)$, since
the Gr\"obner basis
$\GG_0:=I$ itself 
has its
$\GG_0$-standard monomials giving a (homogeneous) $\kk$-basis
for both $S/\init_\prec(I)$ by definition, and for $S/I$
by Proposition~\ref{prop:various-Grobner-equivalents}(iv).
\end{proof}

There are some advantages to working with Gr\"obner bases
in the commutative polynomial algebra $\kk[\zz]$ and exterior algebra $\wedge(\zz)$, 
where GBs for ideals are always finite,
and can be computed via versions of {\it Buchberger's algorithm}.  One can always view quotients  $\kk[\zz]/I$ and $\wedge(\zz)/I$ as quotients
of $\kk\langle \zz \rangle$ via
the surjections
$$
\begin{array}{rcl}
 \kk\langle \zz \rangle \overset{\pi}{\longrightarrow} \kk[\zz] &\text{ with }
 &
 \ker(\pi)=(z_iz_j-z_jz_i : 1 \leq i < j \leq n)\\
 \kk\langle \zz \rangle \overset{\pi}{\longrightarrow}
\wedge(\zz) &\text{ with }&
\ker(\pi)=(z_iz_j+z_jz_i : 1 \leq i < j \leq n) + (z_i^2: 1 \leq i\leq n)   
\end{array}
$$
In other words, $\kk[\zz]/I$ or $\wedge(\zz)/I$ is isomorphic to $\kk\langle \zz \rangle /\pi^{-1}(I)$.
Note that since the
\begin{itemize}
    \item {\it commutators} 
$[z_i,z_j]_+:=z_i z_j-z_jz_i$,
\item {\it anti-commutators} $[z_i,z_j]_-:=z_i z_j + z_jz_i$, and \item squares $z_i^2$
\end{itemize}
that generate $\ker(\pi)$ are homogeneous and quadratic, this means that if $I$ is a homogeneous ideal of $\kk[\zz]$ or $\wedge(\zz)$,
then $\pi^{-1}(I)$ will be a homogeneous two-sided ideal
of $\kk\langle \zz \rangle$.  Similarly, if $I$ is a quadratic ideal, then the same holds for $\pi^{-1}(I)$,
and $\kk\langle \zz \rangle /\pi^{-1}(I)$ will be a quadratic algebra.

This leads to one of the most common techniques for proving
Koszulity.

\begin{theorem} 
Consider ($2$-sided) ideals $I$ inside any of the three rings
$R=\kk\langle \zz \rangle, \kk[\zz],\wedge(\zz)$.
\begin{enumerate}
    \item[(i)] {\rm{(Fr\"oberg \cite{Froberg-original})}} 
    The quotient $R/I$ by any
    quadratic monomial ideal $I$ is Koszul.
    \item[(ii)] \rm{\cite[\S  4]{Froberg-survey}, \cite[Thm. 8.14]{mcculloughPeevaSurvey}, \cite[\S  3]{Peeva}}  If $I$ has a quadratic Gr\"obner basis $\GG$ with respect to some monomial order $\prec$ on $R$, then $R/I$ is Koszul.
\end{enumerate}
\end{theorem}
\begin{proof}
For assertion (i), Fr\"oberg's main result in \cite{Froberg-original} proves Koszulity of a general class of algebras $A$, containing as special cases the quadratic monomial quotients $R/I$ for any such $R$.

Assertion (ii) for the commutative case where $R=\kk[\xx]$ 
is credited in \cite{EisenbudReevesTotaro} to Fr\"oberg's result (i) 
``and a deformation argument noticed by Kempf and others".  This 
deformation argument is written down explicitly by Peeva in \cite[Thm.~22.9(3)]{Peeva}, proving the following assertion.
Given a graded $A$-module $M$, produce a free (left-)$A$-module resolution
$0 \leftarrow M \leftarrow F_0 \leftarrow F_1 \leftarrow \cdots$
which is {\it minimal} in the sense that the differentials have entries in $A_+$.
Then define the {\it graded Betti number} $\beta_{ij}^A(M)$ to be 
the number of free summands of the form $A(-j)$ appearing in
the $i^{th}$ resolvent
$F_i=\bigoplus_{j} A(-j)^{\beta_{ij}^A(M)}$.  Thus Koszulity of $A$
may be rephrased as $\beta_{ij}^{A}(\kk)=0$ for $j \neq i$.
Then one has 
\begin{equation}
\label{eq:commutative-Betti-number-bound-by-deformation}
\beta_{ij}^{\kk[\xx]/I}(\kk) \leq \beta_{ij}^{\kk[\xx]/\init_\prec I}(\kk)
\end{equation}
for any monomial order $\prec$ on $\kk[\xx]$.
Since $I$ having a quadratic Gr\"obner basis with respect to $\prec$
implies that $\kk[\xx]/\init_\prec I$ is Koszul by assertion (i),
this implies that $\kk[\xx]/I$ itself is Koszul.

Assertion (ii) for the anticommutative case where $R=\wedge(\zz)$ is asserted in Peeva \cite[\S 3, p.~613]{Peeva}, indicating that the exterior analogue of
\eqref{eq:commutative-Betti-number-bound-by-deformation}
can be proven by a similar deformation argument, using Gr\"obner basis theory over exterior algebras, and similar in
spirit to \cite[Prop.~1.8]{AramovaHerzogHibi}, \cite[p.~4369]{EisenbudPopescuYuzvinsky}.  This argument employs the exterior
analogue of a (commutative) flat deformation result as in
Eisenbud \cite[Thm.~15.17]{Eisenbud}, along the lines of
Murai \cite[Lemmas~2.1, 2.2]{Murai}.

Assertion (ii) for the noncommutative case where $R=\kk\langle \xx \rangle$
is asserted in \cite[\S 4]{Froberg-survey}. It is proven for
certain kinds of noncommutative quadratic Gr\"obner bases called {\it PBW bases}
in \cite[\S5]{Priddy} and \cite[Ch.~4]{PolishchukPositselski}.  It is also proven for quadratic Gr\"obner bases with respect to {\it degree orderings} on monomials in $\kk\langle \zz\rangle$ in
J\"ollenbeck and Welker \cite[Cor.~4.9]{JollenbeckWelker}.  A
proof for general term orders $\prec$ on $\kk\langle \zz \rangle$ was 
written down recently in an unpublished preprint of
Backelin \cite{Backelin}.
\end{proof}

\section{Matroids, oriented matroids, and supersolvability}
\label{sec: supersolvable}

The Koszul algebras of interest to us are
{\it Orlik-Solomon algebras} of matroids and {\it graded Varchenko-Gel'fand algebras} of oriented matroids,
in the case where the matroids are supersolvable.
We therefore review here the basics of matroids,  oriented matroids, and supersolvability.

\subsection{Matroid and oriented matroid review}
\label{subsec:matroid-review}

A useful reference for matroids is Oxley \cite{Oxley}, and for oriented matroids is
Bj\"orner, Las Vergnas, Sturmfels, White and
Ziegler \cite{BLSWZ}.

A matroid $M$ (respectively, oriented matroid $\OM$) on ground set $E=\{1,2,\ldots,n\}$ is an abstraction of the linear dependence information about a list of vectors $v_1,v_2,\ldots,v_n$ in a vector space over a field $\kk$ (respectively, $\kk=\RR$), forgetting the coordinates of the vectors themselves,
but recording which subsets are linearly dependent (respectively, the $\pm$ signs in their linear dependences).  One way to record this information is with the matroid or oriented matroid's {\it circuits},
abstracting the minimal dependences.

\begin{definition}
A {\it matroid} $M$ on ground set $E=\{1,2,\ldots,n\}$ is defined by its collection $\cC \subset 2^E$ of {\it circuits}, satisfying these axioms:
\begin{enumerate}
\item[{\sf C1.}] $\varnothing \not\in \cC$
\item[{\sf C2.}] If $C,C'$ in $\cC$,
and $C \subseteq C'$ then $C=C'$
\item[{\sf C3.}]  If $C,C'$ in $\cC$,
and $e \in C \cap C' \subsetneq C,C'$,
then there exists $C'' \in \cC$ with
$C'' \subseteq C \cup C' \setminus \{e\}$.
\end{enumerate}

An {\it oriented matroid} $\OM$ on ground set $E=\{1,2,\ldots,n\}$ is defined by its collection $\cC^\pm=\{(C_+,C_-)\}$ of {\it signed circuits} which are pairs $(C_+,C_-)$ of disjoint subsets $C_+ \sqcup C_- \subseteq E$, satisfying these axioms:
\begin{enumerate}
\item[{\sf C1${}^\pm$.}] $(\varnothing,\varnothing) \not\in \cC^\pm$
\item[{\sf C2${}^\pm$.}] If $(C_+,C_-)$ in $\cC^\pm$, then  $(C_-,C_+)$ in $\cC^\pm$
\item[{\sf C3${}^\pm$.}] If $(C_+,C_-),(C'_+,C'_-)$ in $\cC^\pm$,
and $C_+ \cup C_- \subseteq C'_+ \cup C'_-$ then 
$(C'_+,C'_-)=(C_+,C_-)$ or $(C_-,C_+)$.
\item[{\sf C4${}^\pm$.}]  If $(C_+,C_-),(C'_+,C'_-)$ in $\cC^\pm$ and
$e \in C_+ \cap C'_-$, then there exists $(C''_+,C''_-) \in \cC^\pm$ 
with $C'' \subseteq C \cup C' \setminus \{e\}$ having
    $C''_+ \subseteq \left(  C_+ \cup C'_+ \right) \setminus \{e\}$, and
    $C''_- \subseteq \left( C_- \cup C'_- \right) \setminus \{e\}$.
\end{enumerate}
    
\end{definition}
One can check that every oriented matroid $\OM$ with signed circuits $\cC^\pm$ gives rise to a matroid $M$ having circuits $\cC:=\{C_+ \cup C_-: (C_+,C_-) \in \cC^\pm\}$;  one calls the matroid $M$ {\it orientable} whenever it comes from such an oriented matroid $\OM$, and one calls $\cC$ the {\it (matroid) circuits} of $\OM$. 

One calls $M$ a {\it representable matroid} (over the field $\kk$) if there
exists a list of vectors $v_1,v_2,\ldots,v_n$ in a $\kk$-vector space such that 
the subsets $C$ in $\cC$ index the {\it minimal
dependent subsets} $\{v_j\}_{j \in C}$, that is, $\sum_{j \in C} c_j v_j=0$ for
some $c_j$ in $\kk$, but every proper subset of $\{v_j\}_{j\in C}$ is independent.
Similarly, $\OM$ is a {\it representable} oriented matroid if additionally $\kk=\RR$
and the pairs $(C_+,C_-)$ in $\cC^\pm$ give the subsets $C^+=\{j:c_j > 0\}, \, C^-=\{j:c_j < 0\}$.
for all such minimal dependent subsets of $v_1,\ldots,v_n$.

A matroid $M$ on ground set $E$
can also be specified by its collection of {\it flats} $\cF =\{F\} \subseteq 2^E$,
where $F \subseteq E$ is a flat if every circuit $C$ in $\cC$ with
$|C \cap F|=|C|-1$ has $C \subseteq F$.  We will consider $\cF$ as a poset
ordered via inclusion.  This poset turns out to always be a 
{\it geometric lattice}, meaning that  
\begin{itemize}
\item any pair of flats $F,F'$ have a {\it meet} (greatest lower bound) $F \wedge F'=F \cap F'$ and a {\it join} (least upper bound) $F \vee F'$,
\item it is an {\it atomic lattice} in the sense that every flat $F$ has 
$$
F=\bigvee_{\text{ atoms }G \leq F} G,
$$ where {\it atoms} are flats that cover the unique bottom element, and
\item it is {\it upper semimodular}, meaning that there is a {\it rank function}
$r: \cF \rightarrow \{0,1,2,\ldots\}$ satisfying 
\begin{equation}
\label{upper-semimodular-inequality}
 r(F \vee F') \leq r(F)+r(F')-r(F \wedge F').
\end{equation}
\end{itemize}
The {\it rank} of the matroid $M$ is defined to be $r(M):=r(E)$.

It will also be convenient later (in Definition~\ref{Kohno-relations-defn} below) to note that every oriented matroid $\OM$ on $E$ of rank $r$ can be specified via its {\it chirotope}.  This is a function
$\chi_\OM: E^r \rightarrow \{0,\pm 1\}$ satisfying
certain axioms;  see \cite[\S 1.9, 3.5]{BLSWZ}), and the values $\chi_\OM(i_1,i_2,\ldots,i_r)$ are defined only up to an overall rescaling by $\pm 1$.  In the case where $\OM$ is realized by
vectors $v_1,v_2,\ldots,v_n$, then
$\chi_\OM(i_1,i_2,\ldots,i_r)$ 
is the $\{0,\pm 1\}$-valued sign of the determinant of the $r \times r$ matrix having $v_{i_1},v_{i_2},\ldots,v_{i_r}$
as its columns.

In studying Orlik-Solomon and Varchenko-Gel'fand rings,
it will turn out (see Remark~\ref{rmk: simplicity-is-no-restriction} below) that we lose no generality by restricting to matroids and oriented matroids which are {\it simple}, meaning that they have no {\it loops} (= singleton circuits $C=\{i\}$) and no {\it parallel elements} (= circuits $C=\{i,j\}$ of size two).  Consequently, their matroid structure $M$
is completely determined by the poset of flats $\cF$ up to isomorphism, whose
unique bottom element will be the empty flat $F=\varnothing$, and whose atoms
at rank $1$ are the singleton flats $F=\{1\},\{2\},\ldots,\{n\}$, identified
with the ground set $E$.

\subsection{Supersolvability}
\label{supersolvability-review}
We will be focussing on matroids that satisfy the strong condition of supersolvability, reviewed here.

\begin{definition}
Say that a flat $F$ in a matroid $M$ is {\it modular}
if one always has equality in \eqref{upper-semimodular-inequality}:
$$
r(F \vee F') = r(F)+r(F')-r(F \wedge F') \text{ for all }F' \in \cF.
$$
A matroid $M$ is called {\it supersolvable} if the poset $\cF$ contains a complete flag $\underline{F}$ of modular flats
$$
\underline{F}:=(\varnothing =F_0 \subsetneq F_1 \subsetneq \cdots \subsetneq F_{r(M)-1} 
\subsetneq F_{r(M)}=E).
$$
\end{definition}

We consider as examples the (strict) subset of supersolvable matroids among the {\it uniform matroids}, which we recall here.

\begin{definition} ({\it Uniform matroids})  
The {\it uniform matroid}
$M=U_{r,n}$ of rank $r$ on ground set $E=\{1,2,\ldots,n\}$ has
circuits $\cC$ equal to all $(r+1)$-element subsets of $E$.  Its poset of flats $\cF$ is obtained from $2^E$, the {\it Boolean algebra} of rank $n$, by removing all subsets of cardinalities $r,r+1,\cdots,n-2,n-1$.
\end{definition}

\begin{remark}
The uniform matroid $U_{r,n}$ is represented by any list of $n$ vectors $v_1,v_2,\ldots,v_n$ in $\kk^r$ that are sufficiently generic, in the sense that every $r$-element subset $\{v_{i_1},\ldots,v_{i_r}\}$ is linearly independent.  This imposes restrictions on the cardinality of the field $\kk$, depending upon $n$ and $r$, but means that $U_{r,n}$ is always representable over an infinite field, such as $\kk=\RR$, and hence is always orientable.  Nevertheless, some of these orientations $\OM$ of $M=U_{r,n}$ can behave differently, for example in their group of automorphisms $\Aut(\OM)$.  
In the examples of this section, we will consider only the unoriented matroid $M=U_{r,n}$.
\end{remark}

It is not hard to see that the uniform matroid $M=U_{r,n}$ is
\begin{itemize}
\item simple if and only if $(r,n) = (0,0)$ (the empty matroid), $(r,n) = (1,1)$, or $r\geq 2$; and 
\item simple and supersolvable if and only if $(r,n) = (0,0)$, $(r,n) = (1,1)$, or $r=2$ and $n\geq 2$.
\end{itemize}

\begin{example}
\label{ex: Boolean-matroid-1}
The {\it Boolean matroid} $M=U_{n,n}$ on ground set $E=\{1,2,\ldots,n\}$ has no
circuits, that is, $\cC=\varnothing$, and poset of flats $\cF=2^E$.  Every flat $F$ is modular, so
every complete flag $\underline{F}$ of flats is modular and $M$ is supersolvable.
\end{example}

\begin{example}
\label{ex: rank-two-matroid-1}
Every rank two simple matroid is isomorphic to a uniform matroid $M=U_{2,n}$ on $E=\{1,2,\ldots,n\}$, with this flat poset $\cF$:

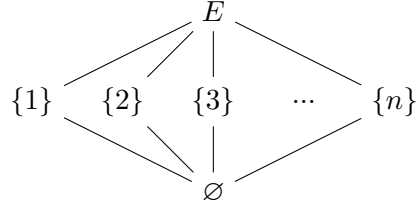
\begin{center}
\begin{tikzpicture}
  [scale=.30,auto=left]
  \node (E) at (8,12) {$E$};
  \node (1) at (0,8) {$\{1\}$};
  \node (2) at (4,8) {$\{2\}$};
   \node (3) at (8,8) {$\{3\}$};
   \node (dots) at (12,8) {$...$};
  \node (n) at (16,8) {$\{n\}$};
  \node (empty) at (8,4) {$\varnothing$};
  \foreach \from/\to in {empty/1,empty/2,empty/3,empty/n,1/E,2/E,3/E,n/E}
    \draw (\from) -- (\to);
\end{tikzpicture}
\end{center}
Again, every flat $F$ is modular, and every complete flag $\varnothing \subset \{i\} \subset E$ shows that $M$ is supersolvable.
\end{example}

Our original motivation came from {\it braid matroids}.

\begin{example} 
\label{graphic-braid-example}
({\it Supersolvable graphic matroids} and {\it braid matroids}) 
Let $G$ be a graph on vertex set $\{1,2,\ldots,n\}$ with edge set $E \subseteq \{\{i,j\}: 1 \leq i<j \leq n\}$ which is {\it simple},
that is, $G$ no self-loops and no parallel edges.  Then $G$ gives rise to a simple {\it graphic} matroid $M$
(and oriented matroid $\OM$) represented by the list of vectors
$\{v_{ij}=e_i-e_j\}_{\{i,j\} \in E} \subset \RR^n$, where $e_1,\dots,e_n$ are standard basis vectors.  The matroid circuits $\cC$ are indexed by subsets $C \subseteq E$ of edges that form a cycle within $G$.  Stanley showed \cite[Prop. 2.8]{stanley1972supersolvable}
that this graphic matroid is supersolvable if and only $G$ is a {\it chordal} graph, meaning that for every minimal cycle of edges 
$u_1-u_2-\cdots-u_{\ell-1}-u_{\ell}-u_1$ in $G$ having $\ell \geq 4$, there will be another edge $\{u_i,u_j\}$ of $G$ with $i \not\equiv j\pm 1\bmod{\ell}$ forming a chord. 

In particular, the {\it complete graph} $K_n$ on $n$ vertices with all $\binom{n}{2}$ edges is a chordal graph, and its graphic matroid is called the {\it braid matroid} $\braid_n$ {\it on $n$ strands}.  Its poset of flats $\cF$ is isomorphic to the lattice $\Pi_n$ of all set partitions $\pi=(B_1,\ldots,B_\ell)$ of $\{1,2,\ldots,n\}=\sqcup_{i=1}^\ell B_i$,
with ordering by refinement:  $\pi \leq \pi'$
if for every block $B_i$ of $\pi$ there exists some block $B'_{i'}$ of $\pi'$ having $B_i \subseteq B'_{i'}$. The flat $F$ corresponding to $\pi$ contains all edges $\{i,j\}$ whose end vertices $i,j$ lie in the same block $B_k$ of $\pi$.  The modular flats correspond to partitions $\pi$ with at most one non-singleton block.  For example, one modular complete flag $\underline{F}$ of flats 
corresponds to the set partitions $\pi_1 < \pi_2 < \cdots < \pi_n$ where 
$$
\pi_k:=\{ \{1,2,\ldots,k\}, \{k+1\}, \{k+2\},\ldots,\{n-1\}, \{n\}\}.
$$
\end{example}

\section{Orlik-Solomon and Varchenko-Gel'fand rings}
\label{sec:OS-and-VG-rings}
We review here the Orlik-Solomon algebra of a matroid $M$
and graded Varchenko-Gel'fand algebra\footnote{Also called the {\it Cordovil algebra} in \cite{MatherneMiyataProudfootRamos}.} of an oriented matroid $\OM$. 
Useful references for Orlik-Solomon algebras are Dimca \cite[Ch. 3]{Dimca}, Dimca and Yuzvinsky \cite{DimcaYuzvinsky}, Orlik and Terao \cite[Ch. 3]{OrlikTerao}, Yuzvinsky \cite{Yuzvinsky}.  Useful references for graded Varchenko-Gel'fand algebras are Brauner \cite[Secs. 3.3, 5.2]{Brauner}, Cordovil \cite{Cordovil}, Dorpalen-Barry \cite{Dorpalen-Barry}, Dorpalen-Barry, Proudfoot and Wang \cite{DorpalenBarryProudfootWang}, Moseley \cite{Moseley}, Varchenko and Gel'fand \cite{VarchenkoGelfand}.

For the remainder of this section, let $\kk$ be any commutative ring with $1$.

\begin{definition}({\it Orlik-Solomon algebra})
For a simple matroid $M$ on $E=\{1,2,\ldots,n\}$,
define its {\it Orlik-Solomon algebra} over $\kk$ as an anti-commutative quotient 
$$
\OS(M):=\wedge(x_1,\ldots,x_n)/I_{\OS(M)}
$$
where $\wedge(x_1,\ldots,x_n)$ is the exterior algebra over $\kk$
on $n$ generators.  The Orlik-Solomon ideal 
\begin{equation}
\label{OS-ideal-generators1}
I_{\OS(M)}=
(\partial(x_C): C \in \cC)
\end{equation}
has one generator $\partial(x_C)$ for each circuit $C=\{c_1,c_2,\ldots,c_k\}$ in $\cC$, with $\partial(x_C)$  defined by 
\begin{equation}
\label{OS-ideal-generators2}
\partial x_C := \sum_{j=1}^{k} (-1)^{j-1} 
x_{c_1} \wedge \cdots \wedge x_{c_{j-1}} \wedge \widehat{x_{c_j}} 
\wedge x_{c_{j+1}} \cdots \wedge x_{c_k}.
\end{equation}
\end{definition}

\begin{definition}({\it Graded Varchenko-Gel'fand ring})
For a simple oriented matroid $\OM$ on $E=\{1,2,\ldots,n\}$,
define its {\it graded Varchenko-Gel'fand ring} over $\kk$ as the commutative quotient 
$$
\VG(\OM):=\kk[x_1,\ldots,x_n]/I_{\VG(\OM)}
$$
where $\kk[x_1,\ldots,x_n]$ is the polynomial algebra over $\kk$.  The graded Varchenko-Gel'fand ideal 
\begin{equation}
\label{VG-ideal-generators1}
I_{\VG(\OM)}=
(x_1^2,\ldots,x_n^2) + (\partial^\pm(x_C): C \in \cC)
\end{equation}
contains the squares $\{ x_i^2\}_{i=1}^n$ along with one generator $\partial^\pm(x_C)$ for each circuit $C$ in $\cC$, with $\partial^\pm(x_C)$ defined by choosing one of the two signed circuits\footnote{The choice is immaterial -- making the other choice replaces $\partial^\pm(x_C)$ by its negative.}
$(C_+,C_-)$ in $\cC$ with $C=C_+ \cup C_-$, and setting
\begin{equation}
\label{VG-ideal-generators2}
\partial^\pm(x_C) := \sum_{c_j \in C_+ \cup C_-} \sgn{C}{c_j}  \cdot
x_{c_1}  \cdots  x_{c_{j-1}}  \widehat{x_{c_j}} 
 x_{c_{j+1}} \cdots x_{c_k}.
\end{equation}
Here $\sgn{C}{c_j}=\pm 1$, namely $+1$ when $c_j \in C_+$
and $-1$ when $c_j \in C_-$.
\end{definition}

\begin{remark}
\label{rmk: simplicity-is-no-restriction}
    Our assumption that $M, \OM$ are simple really presents no restriction.  In either case,
    \begin{itemize}
        \item a loop $i$ in $E$ would give a circuit $C=\{i\} \in \cC$, causing the collapse $\OS(M)=0=\VG(\OM)$ since $I_{\OS(M)}$ or $I_{\VG(\OM)}$ contains the generator $\partial(x_C)=1$ or $\partial^\pm(x_C)=1$, and
        \item parallel elements $i,j$ in $E$ would give rise to a circuit $C=\{i,j\} \in \cC$, making $x_i=\pm x_j$ in the rings $\OS(M)$ or $VG(\OM)$  because $I_{\OS(M)}$ or $I_{\VG(\OM)}$ contains a generator $\partial(x_C)$ or $\partial^\pm(x_C)$ of 
        of the form $x_i \pm x_j$.
    \end{itemize}
Thus our assumption in Section~\ref{sec: koszul-algebras} that
our standard graded $\kk$-algebras are minimally generated by the variables $x_1,\ldots,x_n$ is consistent with assuming that $M, \OM$ are simple matroids.
\end{remark} 

\subsection{Flat decomposition}
\label{sec: flat-decomposition}
An important feature of both $\OS(M)$ and $\VG(\OM)$
is that their $\NN$-grading is refined by a $\kk$-vector space decomposition
indexed by the matroid flats $F$ in $\cF$.

\begin{definition}
\label{def: flat-decomp-for-polynomials}
Given matroid $M$ or oriented matroid $\OM$ on $E = \{1,\dots, n\}$ with flats $\cF$, abbreviating the variable sets $\xx=(x_1,\ldots,x_n)$, consider the $\kk$-vector space decompositions
\begin{align*}
T(V) &=\kk\langle \xx\rangle= \bigoplus_{F \in \cF} \underbrace{T(V)_F}_{=\kk\langle \xx\rangle_F},\\
& \\
\Sym(V)& = \kk[\xx]=\bigoplus_{F \in \cF} \underbrace{\Sym(V)_F}_{=\kk[\xx]_F},\\
 & \\
\wedge(V) &= \wedge(\xx) = \bigoplus_{X\in \cF} \underbrace{\wedge(V)_F}_{= \wedge(\xx)_F},\\
\end{align*}
where $\kk\langle \xx\rangle_F, 
\kk[\xx]_F, \wedge(\xx)_F$ are the
$\kk$-spans of monomials $x_{j_1} x_{j_2} \cdots x_{j_k}$ with $\{j_1\}  \vee \cdots \vee \{j_k\}=F$.

\end{definition}

Both $\OS(M), \VG(\OM)$ inherit these $\kk$-vector space decompositions by flats; for $\OS(M)$, see
\cite{OrlikTerao}*{Thm 3.26, Cor. 3.27}, \cite{DimcaYuzvinsky}*{\S2.3}, \cite{Yuzvinsky}*{\S2.3}, and for $\VG(\OM)$ see \cite{Brauner}*{Theorem 5.5}.
\begin{prop}
\label{prop: flat-decomp}
For a matroid $M$ or oriented matroid $\OM$, the ideals $I_{\OS(M)}, I_{\VG(\OM)}$ are homogeneous with respect to the decomposition
in Definition~\ref{def: flat-decomp-for-polynomials}, that is,
\begin{align*}
I_{\OS(M)} &=\bigoplus_{F \in \cF} \kk(\xx)_F \cap I_{\OS(M)},\\
I_{\VG(\OM)} &=\bigoplus_{F \in \cF} \kk[\xx]_F \cap I_{\VG(\OM)}.
\end{align*}
Hence they induce $\kk$-vector space decompositions of the quotients $\OS(M),\VG(\OM)$:
\begin{align}
\label{OS-flat-decompsition}
\OS(M) &= \bigoplus_{F\in \cF} \OS(M)_F,\\
\label{VG-flat-decomposition}
\VG(\OM) &= \bigoplus_{F\in \cF} \VG(\OM)_F.
\end{align}
\end{prop}

We note here an implication for quadratic duals that will become important later, in Section~\ref{sec: holonomy-lie-algebra}.  When considering
$\OS(M),\VG(\OM)$ as quotients $A=\kk\langle \xx\rangle/I$ of the tensor algebra for a two-sided ideal $I$, the quadratic
part $I_2 \subset T^2(V)=\kk\langle \xx\rangle_2$ inherits the flat decomposition $I_2=\bigoplus_{F \in \cF} [T^2(V)_F \cap I]$
from $T^2(V)=\bigoplus_{F \in \cF} T^2(V)_F$.
On the other hand, if one defines the
analogous flat decomposition for the dual tensor algebra and its  dual variables $\yy=(y_1,\ldots,y_n)$
$$
T(V^*) =\kk\langle \yy\rangle= \bigoplus_{F \in \cF} \underbrace{T(V)_F}_{=\kk\langle \yy\rangle_F},
$$
then the pairing
$
T^2(V^*) \times T^2(V) \rightarrow \kk
$
from \eqref{bilinear-pairing-on-2-tensors} 
makes $T^2(V^*)_F$ and $T^2(V)_{F'}$ orthogonal for $F \neq F'$.  This implies that the computation of
$J_2:=I_2^\perp$ can be done flat-by-flat:
\begin{equation}
\label{flat-by-flat-shriek-computation}
J_2=\bigoplus_{F \in \cF} [T^2(V^*)_F \cap J_2]
\quad \text{ where } \quad [T^2(V^*)_F \cap J_2] :=  [ T^2(V)_F \cap I_2 ]^\perp.
\end{equation}
In particular, whenever $A=\OS(M),\VG(\OM)$ are Koszul, or even just quadratic algebras $A=\kk \langle \xx \rangle/I$ with $I=(I_2)$,
their quadratic duals $A^!=\kk \langle \xx \rangle/J$ where $J=(J_2)=(I_2^\perp)$ inherit a flat decomposition:
\begin{equation}
\label{flat-decomposition-for-shrieks}
A^!=\bigoplus_{F \in \cF} A^!_F.
\end{equation}

\subsection{Symmetry}
\label{sec:symmetry-of-OS-and-VG}

Symmetries of a matroid $M$ or oriented matroid $\OM$ lead to $\kk$-algebra
automorphisms of $\OS(M)$ or $\VG(\OM)$,
as we explain next.

\begin{definition}
Let $M$ be a matroid  on $E=\{1,2,\ldots,n\}$
with circuits $\cC$.
A permutation $\sigma$ in the symmetric group $\symm_n$ is an {\it automorphism} of $M$, written $\sigma \in \Aut(M)$, if
$\sigma(\cC)=\cC$, that is,
for every $C$ in $\cC$, one has $\sigma(C) \in \cC$.
\end{definition}

One can then check that for any
matroid $M$ and $\sigma$ in $\Aut(M)$, if $\sigma$
acts on $\wedge(x_1,\ldots,x_n)$ by 
permuting subscripts of the variables, that is,
$
\sigma(x_i):=x_{\sigma(i)},
$
then the generator $\partial(x_C)$ for the Orlik-Solomon ideal $I_{\OS(M)}$ has 
$$
\sigma(\partial(x_C)) =\pm \partial(x_{\sigma(C)}).
$$
Consequently, $\sigma$ preserves $I_{\OS(M)}$ and induces a graded $\kk$-algebra automorphism of $\OS(M)$.

\begin{definition}
\label{defn: OM-automorphisms}
Let $\OM$ be an oriented
matroid on $E=\{1,2,\ldots,n\}$.  
Its automorphism group $\Aut(\OM)$
will be a subgroup of the {\it hyperoctahedral group} 
$\symm^\pm_n$;  this is the set of all
{\it signed permutations} $\sigma$ of $\{\pm 1,\pm 2,\ldots,\pm n\}$, meaning those permutations
which commute with the involution $+i \leftrightarrow -i$, or in other words, 
$\sigma(\pm i)=-\sigma(\mp i)$. 
As notation, for $i,j \in \{1,2,\ldots,n\}$,
define
\begin{align*}
|\sigma(i)|&:=j  \quad \text{ if }\sigma(+i) \in \{\pm j\},\\
\epsilon(\sigma(i))&=
\begin{cases} +& \text{ if }\sigma(+i) =+j,\\
-& \text{ if }\sigma(+i) =-j.
\end{cases}
\end{align*}
Then a signed permutation $\sigma$ is an {\it automorphism} of $\OM$ if
for every signed circuit $(C_+,C_-)$ in $\cC^\pm$, the following pair $(C'_+,C'_-)$ is also a signed circuit in $\cC^\pm$, where

\begin{equation}
\label{automoprhic-image-of-a-signed-circuit}
\begin{aligned}
    C'_+&:=\{|\sigma(i)|: i \in C_+ \text{ and }\epsilon(\sigma(i))=+\} \sqcup
\{|\sigma(i)|: i \in C_- \text{ and } \epsilon(\sigma(i))=-\},
\\
C'_-&:=\{|\sigma(i)|: i \in C_- \text{ and } \epsilon(\sigma(i))=+\} \sqcup
\{|\sigma(i)|: i \in C_+ \text{ and } \epsilon(\sigma(i))=-\}.
\end{aligned}
\end{equation}

\end{definition}

For $\sigma$ in $\Aut(\OM)$, let $\sigma$
act on $\kk[x_1,\ldots,x_n]$
via 
$$
\sigma(x_i):=\epsilon(\sigma(i)) \cdot x_{|\sigma(i)|}.
$$
One can then check that for
signed circuits $(C_+,C_-), (C'_+,C'_-)$ related as in \eqref{automoprhic-image-of-a-signed-circuit},
if $C=C_+ \cup C_-$ and $C'=C'_+ \cup C'_-$, then
the generator $\partial^\pm(x_C)$ for the ideal 
$I_{\VG(\OM)}$ has 
$$
\sigma(\partial^\pm(x_C)) =\pm \partial^\pm(x_{C'}).
$$
Consequently, $\sigma$ gives rise to a graded $\kk$-algebra automorphism of $\VG(\OM)$.

In this way, when $M, \OM$ have some group $G$ of automorphisms, we consider $A=\OS(M), \VG(\OM)$ as graded $\kk G$-modules, and study
their equivariant Hilbert series as in \eqref{generic-equivariant-Hilbert-series}.
Similarly, when these algebras $A$ are Koszul,
we will study the equivariant Hilbert series for their Koszul dual $A^!$.  Note that in the dual setting, the dual variables $y_1,\ldots,y_n$ that give a basis for $V^*$ obey the same rules 
\begin{align*}
    \sigma(y_i)&=y_{\sigma(i)} \text{ for }\OS(M)^!,\\
    \sigma(y_i)&=\epsilon(\sigma(i)) \cdot y_{|\sigma(i)|} \text{ for }\VG(\OM)^!.
\end{align*}
This is because $V^*$ carries the contragredient representation to $V$, where the matrix for the action of $\sigma$
in the basis of $y_1,\ldots,y_n$ is the inverse transpose $(A^{-1})^t$ of the matrix $A$ for its action on $x_1,\ldots,x_n$.  However, signed (or unsigned) permutation matrices $A$ are orthogonal: $(A^{-1})^t=A$.

\subsection{Gr\"obner bases and broken circuits}
\label{sec:broken-circuits}
It turns out that the above generators for the ideals presenting $\OS(M)$ and $\VG(\OM)$ are actually Gr\"obner bases, with easily-identified standard monomials.

\begin{definition}
Given a matroid $M$ on $E=\{1,2,\ldots,n\}$ and any circuit $C=\{c_1 < c_2 < \cdots < c_k\}$ in $\cC$, the associated {\it broken circuit} is 
$$
C \setminus \{\min(C)\}= C \setminus \{c_1\}=\{c_2< \cdots <c_k\}.
$$
A subset $I \subset E$ is an {\it NBC (no-broken-circuit) set} if
it contains none of the sets $\{ C \setminus \{\min(C)\} \}_{C \in \cC}$.
\end{definition}

\begin{theorem}
Fix a matroid $M$ and oriented matroid $\OM$ on $E=\{1,2,\ldots,n\}$,
with circuits $\cC$. Choose any monomial orders $\prec$ on 
$\wedge(x_1,\ldots,x_n)$ and $\kk[x_1,\ldots,x_n]$ having $x_1 \prec x_2 \prec \cdots \prec x_n$.

\begin{enumerate}
    \item[(i)] \cite[Thm. 2.8]{Yuzvinsky} The generators $\GG=\{\partial(x_C)\}_{C \in \cC}$ in \eqref{OS-ideal-generators1} form a Gr\"obner basis for $I_{\OS(M)}$ with respect to $\prec$.
    \item[(ii)] \cite[Thm 1]{Dorpalen-Barry}
    The generators $\GG=\{x_i^2\}_{i=1}^n \cup \{\partial^\pm(x_C)\}_{C \in \cC}$ in \eqref{VG-ideal-generators1} form a Gr\"obner basis for $I_{\VG(\OM)}$ with respect to $\prec$.
\end{enumerate}
Furthermore, in both cases, if $C=\{c_1 <c_2< \cdots < c_k\}$ in $\cC$, 
then the $\prec$-initial term $\init_\prec(\partial(x_C))$
or $\init_\prec(\partial^\pm(x_C))$ is the monomial 
$x_{c_2} \cdots x_{c_k}$,
supported on the broken circuit associated to $C$.
Consequently, in either case, the $\GG$-standard monomials 
are the NBC monomials 
$$
\{ x_I=x_{i_1} \cdots x_{i_\ell} : 
\text{NBC sets }I=\{ i_1,\ldots,i_\ell \} \subseteq E\}.
$$
In particular, $\OS(M)$ and $\VG(\OM)$ have the same Hilbert series, given by
$$
\Hilb(\OS(M),t) = \Hilb(\VG(\OM),t) = \sum_{\text{NBC sets } I \subseteq E} t^{|I|}.
$$
\end{theorem}

\begin{remark}
\label{rem: NBC-bases-respect-flats}
    One can readily check that the NBC standard monomial bases for $\OS(M),\VG(\OM)$ respect the flat decompositions \eqref{OS-flat-decompsition}, \eqref{VG-flat-decomposition} in this sense: for each flat $F \in \cF$, the components $\OS(M)_F, \VG(M)_F$ both have as $\kk$-bases the monomials 
    $\{x_I:  I \text{ an NBC set with }\vee_{i \in I} \{i\}=F\}$.
\end{remark}

For supersolvable $M$, one has {\it quadratic} Gr\"obner bases,
making $\OS(M), \VG(\OM)$ Koszul, as we explain next.  Bj\"orner, Edelman and Ziegler \cite{BEZ} gave a useful alternate characterization of the modular complete flags of flats witnessing supersolvability.  To state it, recall that a flat $F$ with $r(F)=r(M)-1$ is called a {\it coatom} in $\cF$.  Also recall that for a matroid $M$ on $E$ and subset $A \subseteq E$, the {\it restriction} $M|_A$ is the matroid on ground set $A$ defined with circuits 
$
\{ C \in \cC: C \subseteq A \}.
$
\begin{prop} \cite[Thm. 4.3]{BEZ}
\label{thm:BEZ}
Let $M$ be a simple matroid on ground set $E$.
\begin{itemize}
\item[(i)] For flats $F$ which are coatoms, being a modular element is equivalent to the following condition:  for any $j \neq k$ in $E \setminus F$, there exists $i$ in $F$ with $\{i\} \leq \{j\} \vee \{k\}$.
\item[(ii)] The flats in a complete flag 
$\underline{F}=(\varnothing =F_0 \subsetneq F_1 \subsetneq \cdots \subsetneq F_{r(M)-1} 
\subsetneq F_{r(M)}=E)
$ 
are all modular if and only if $F_{i-1}$ is
a modular coatom within $M|_{F_i}$
for each $i=1,2,\ldots,r(M)$.
\end{itemize}
\end{prop}

Bj\"orner and Ziegler \cite{BjornerZiegler} later elaborated on this,
proving the following.

\begin{prop}\cite{BjornerZiegler}*{Theorem 2.8}
\label{prop: bjorner-ziegler-characterization} 
Let $M$ be any simple matroid of rank $r$ on ground set $E$.
The following are equivalent:
\begin{enumerate}
\item[(i)] $M$ is supersolvable, say with a modular
complete flag of flats $\underline{F}=(F_i)_{i=0,1,\ldots,r}$.
 
\item[(ii)] There exists an ordered set partition $\underline{E}=(E_1, E_2, \dots, E_r)$ of $E=E_1 \sqcup \cdots \sqcup E_r$ such that if $j,k$ in $E_q$ with $j \neq k$, then there exists $p<q$ and $i$ in $E_p$ with $C=\{i,j,k\}$ in $\cC$.

\item[(iii)] One can reindex/order $E=\{1< 2 < \cdots < n\}$ so that the minimal broken circuits (with respect to inclusion) are all of size $2$.
\end{enumerate}

Furthermore, when these conditions hold, 
\begin{itemize}
    \item[(a)] a modular flag $\underline{F}$
as in (i) gives an
ordered set partition $\underline{E}$ as in (ii) via $E_i:=F_i \setminus F_{i-1}$, and
\item[(b)] an ordered set partition $\underline{E}$ as in (ii) gives an ordering $\prec$ on $E$ as in (iii) by extending the partial order that makes
elements of $E_p$ come $\prec$-earlier than elements of $E_q$ when $p<q$,
\item[(c)] the minimal broken circuits with respect inclusion are all pairs of the form $\{j,k\}$ in some set $E_q$ for $q=1,2,\ldots,r$;
hence the NBC sets $I \subset E$
are the subsets containing at most one element from each $E_p$ for $p=1,2,\ldots,r$.
\end{itemize}
\end{prop}

\begin{definition}
    For a supersolvable matroid $M$, with $\underline{F}, \underline{E}$ as in Proposition~\ref{prop: bjorner-ziegler-characterization},    
    denote by
    $\cC_{\BEZ}(\underline{E}) \subseteq \cC$  the circuits $C=\{i,j,k\}$ with $i \in E_p$ and $j \neq k \in E_q$ for $p<q$ from Proposition~\ref{prop: bjorner-ziegler-characterization}(ii).
\end{definition}

\begin{cor}
\label{cor: primal-quadratic-GB}
Let $M, \OM$ be supersolvable simple matroids or oriented matroids on $E$, with $\underline{E}$ as in Proposition~\ref{prop: bjorner-ziegler-characterization}.   Fix a field $\kk$, and monomial
orderings $\prec$ on $\wedge(x_1,\ldots,x_n), \kk[x_1,\ldots,x_n]$
with $x_1 \prec x_2 \prec \cdots \prec x_n$. 
\begin{itemize}
\item \cite{Peeva}, \cite[\S  6.3]{Yuzvinsky} $I_{\OS(M)}$ has quadratic Gr\"obner basis $\GG=\{\partial(x_C)\}_{C \in \cC_{\BEZ}}$, 
where
\begin{equation}
\label{supersolvable-OS-quadratic-GB}
\partial(x_C)= x_i \wedge x_j - x_i \wedge x_k + \underline{x_j \wedge x_k}.
\end{equation}
\item \cite{Dorpalen-Barry} $I_{\VG(\OM)}$ has quadratic Gr\"obner basis 
$\GG=\{x_i^2\}_{i=1}^n \cup \{\partial^\pm(x_C)\}_{C \in \cC_{\BEZ}}$, 
where
\begin{equation}
\label{supersolvable-VG-quadratic-GB}
\partial^\pm(x_C)=\sgn{C}{k} \cdot x_i  x_j 
+ \sgn{C}{j} \cdot x_i x_k 
+ \sgn{C}{i} \cdot \underline{x_j x_k}.
\end{equation}
\end{itemize}
In both cases, 
\begin{itemize}
    \item the $\prec$-initial terms of the elements of $\GG$ are shown underlined above, 
    \item the $\GG$-standard monomial basis 
    $\{x_I\}$ are indexed by the NBC sets $I \subseteq E$, which are exactly those sets containing at most one element from each $E_p$ for $p=1,2,\ldots,r$,
    \item $\OS(M), \VG(\OM)$ are Koszul algebras, 
    \item with the same Hilbert series
\begin{equation}
\label{supersolvable-primal-Hilb}
\Hilb(\OS(M),t) = 
\Hilb(\VG(\OM),t) = 
(1+e_1 t) (1+e_2t) \cdots (1+e_r t)
\end{equation}
where $e_p = |E_p|$ for $p=1,2,\ldots,r$.
\end{itemize}
\end{cor}

The integers $(e_1,e_2,\ldots,e_p)$ are often called the {\it exponents} of the supersolvable matroid $M$, due to their connection with the theory of {\it free hyperplane arrangements} and the exponents
of {\it reflection arrangements}; see Orlik and Terao \cite[\S4.2]{OrlikTerao}.

\subsection{Quadratic Gr\"obner basis for the Koszul dual}
\label{sec:GB-for-shrieks}

We next prove a counterpart to Corollary~\ref{cor: primal-quadratic-GB}
for the Koszul duals $A^!$ of $A=\OS(M), \VG(\OM)$
in the supersolvable case.  Since $A=\OS(M)$ or $\VG(\OM)$ are Koszul algebras, one can view them
as noncommutative quotients $A=\kk\langle x_1,\ldots,x_n\rangle/I$, and form their
Koszul duals   $A^!=\kk\langle y_1,\ldots,y_n\rangle/J$, as in Section~\ref{sec: koszul-algebras}.
Certain relations in $A^!$ will play a key role.

\begin{definition}
\label{Kohno-relations-defn}
Let $M$ be a simple matroid on $E=\{1,2,\ldots,n\}$.  For each rank two flat $F \subset E$ and each $j$ in $F$, define an element of $\kk\langle \yy\rangle :=\kk\langle y_1,\ldots,y_n\rangle$ by
\begin{align}
\label{unoriented-Kohno-relation}
r(j,F):=\sum_{k \in F \setminus \{j\}} [y_j,y_k]_+
=\sum_{k \in F \setminus \{j\}} (y_jy_k-y_k y_j).
\end{align}
Let $\OM$ be a simple oriented matroid on $E=\{1,2,\ldots,n\}$. For each rank two flat $F \subset E$, pick one of the two chirotopes
$\chi_{\OM|_F}:F^2 \rightarrow \{0,\pm 1\}$  on the restriction $\OM|_F$, which are the same up to
the overall scaling by $\pm 1$. Then for each $j$ in $F$ define an element of $\kk\langle \yy\rangle$ by
\begin{align}
\label{oriented-Kohno-relation}
r^\pm(j,F):=\sum_{k \in F \setminus \{j\}} 
\chi_{\OM|_F}(j,k) \cdot [y_j,y_k]_-
=\sum_{k \in F \setminus \{j\}} 
\chi_{\OM|_F}(j,k) \cdot (y_jy_k + y_k y_j).
\end{align}
\end{definition}

The relations \eqref{unoriented-Kohno-relation} appear in work of Kohno 
\cite{kohno1983holonomy}
presenting the {\it holonomy Lie algebra} for the complement of any complex hyperplane arrangement; see Section~\ref{sec: holonomy-lie-algebra} for further discussion. As far as we know, relations \eqref{oriented-Kohno-relation} are new.  Certain subsets of these relations in \eqref{unoriented-Kohno-relation} or \eqref{oriented-Kohno-relation} play a distinguished role in the supersolvable case.

\begin{definition}
\label{retrograde-pair-defn}
Let $M, \OM$ be supersolvable simple matroids or oriented matroids, and $\underline{E}=(E_1,\ldots,E_r)$ a choice of an ordered partition of its ground set $E$ as in Proposition~\ref{prop: bjorner-ziegler-characterization}.
Call $(j,i)$ in $E^2$ a {\it retrograde (ordered) pair} with respect to $\underline{E}$ if
$i \in E_p$ and $j \in E_q$ with $p<q$.

For each retrograde pair $(j,i)$, let $F:=\{j\} \vee 
\{i\}$ be the rank two flat that they span, and  denote by $r(j,i), r^\pm(j,i)$ the following two relations, equivalent to $r(j,F)$
from \eqref{unoriented-Kohno-relation} and  $r^\pm(j,F)$ from  \eqref{oriented-Kohno-relation}:
\begin{align}
    \label{unoriented-retrograde-relation}
r(j,i)&:=\underline{y_j y_i}-y_i y_j+\sum_{k \in F \setminus \{i,j\}} [y_j,y_k]_+.\\
    \label{oriented-retrograde-relation}
r^\pm(j,i)&=  \underline{y_j y_i}-y_i y_j 
+ \chi_{\OM|_F}(j,i) \sum_{k \in F \setminus \{i,j\}} \chi_{\OM|_F}(j,k) 
 \cdot [y_j,y_k]_-.
\end{align}
\end{definition}

The following key point will be
used in the proofs of Theorems~\ref{thm: shriek-presentations} and \ref{thm: shrieks-have-right-NZD}.

\begin{lemma}
    \label{rank-two-flat-lemma}
    In the context of Definition~\ref{retrograde-pair-defn} of a retrograde pair $(j,i)$ with $i \in E_p$
    and $j \in E_q$ for $p<q$,  the rank two flat $F:=\{j\} \vee \{i\}$ has $F \setminus \{i,j\} \subset E_q$.
    
    Consequently, \eqref{unoriented-retrograde-relation} and  \eqref{oriented-retrograde-relation} can
     be viewed as rewriting rules that replace the underlined term $y_j y_i$ by the term $y_i y_j$ together with a sum
     of monomials $y_j y_k, y_k y_j$ whose subscripts $j,k$ both lie in $E_q$.
\end{lemma}
\begin{proof}
    Any $k \in F \setminus \{i,j\}$ leads to a circuit $C=\{i,j,k\}$ since $M$ is a simple matroid and $F$ has rank two.
    As $j > i$, one knows $j \neq \min C$, so the associated broken circuit $B \subset C$ is either $B=\{j,i\}$ or $B=\{j,k\}$. But assertion (c) in Proposition~\ref{prop: bjorner-ziegler-characterization} implies
    $B$ contains a pair lying in some set $E_{q'}$.  This implies
    $q'=q$, and $\ell \neq i$ since $i \in E_p \neq E_q$.
    Thus $B=\{j,k\}$, and $k$ lies in $E_q$.
\end{proof}

\begin{theorem}
\label{thm: shriek-presentations}    
Let $M, \OM$ be matroids and oriented matroids which are supersolvable, with ground set $E=\{1,2,\ldots,n\}$ and $\underline{E}$ as in Proposition~\ref{prop: bjorner-ziegler-characterization}.  Consider the Koszul algebras $A=\OS(M)$ or $\VG(\OM)$, and their Koszul dual
$A^!=\kk\langle y_1,\ldots,y_n\rangle/J$.
Then there exist monomial orderings $\prec$
on $\kk\langle y_1,\ldots,y_n\rangle$
with these properties.
\begin{enumerate}
\item[(i)] $A^!=\OS(M)^!=\kk\langle \yy \rangle/J$
has 
$
\{ r(j,F): j \in F \text{ a rank two flat }\}
$
as a Gr\"obner basis for $J$, and a reduced  Gr\"obner basis 
$$
\GG:=\{ r(j,i):\text{ retrograde pairs }(j,i)\}
$$
with the $\prec$-initial term of 
$r(j,i)$ underlined in \eqref{unoriented-retrograde-relation}.


\item[(ii)]  $A^!=\VG(\OM)^!=\kk\langle \yy \rangle/J$
has 
$
\{ r^\pm(j,F): j \in F \text{ a rank two flat }\}
$
as a Gr\"obner basis for $J$, and a reduced  Gr\"obner basis 
$$
\GG:=\{ r^\pm(j,i):\text{ retrograde pairs }(j,i)\}
$$
with the $\prec$-initial term of 
$r^\pm(j,i)$ underlined in \eqref{oriented-retrograde-relation}.

\end{enumerate}
In particular, 
\begin{itemize}
    \item[(iii)] their ideals $J$ share the same initial monomials
$
\{y_j y_i : \text{retrograde pairs }(i,j)\}
,
$
\item[(iv)] and hence the same $\GG$-standard monomial $\kk$-basis 
for $A^!$, of the form
$
\{ m_1 \cdot m_2 \cdots m_{r-1} \cdot m_r\}
$ 
where each $m_p$ is any noncommutative monomial in the variable set $\{y_j\}_{j \in E_p}$,
\item[(v)]
and they have the same Hilbert series
\begin{equation}
    \label{supersolvable-dual-Hilb}
\Hilb(\OS(M)^!,t)=
\Hilb(\VG(\OM)^!,t)=
\frac{1}{(1-e_1t)(1-e_2t) \cdots (1-e_rt)}
\end{equation}
where $e_p=|E_p|$ are the exponents from Corollary~\ref{cor: primal-quadratic-GB}.
\end{itemize}
\end{theorem}

\begin{proof}
First let us specify a monomial order $\prec$ on $\kk\langle y_1,\ldots,y_n\rangle$ for which the underlined terms in \eqref{unoriented-retrograde-relation}, \eqref{oriented-retrograde-relation} are their $\prec$-initial terms. 
Recall that our indexing has $i<j$ for each retrograde pair $(j,i)$.  We claim that it suffices to let $\prec$ be a graded version of a lexicographic order having $y_1 \succ y_2 \succ \cdots \succ y_n$ that reads monomials from the right.  More precisely, this means that for two unequal monomials 
\begin{align*}
    m =y_{i_1} \cdots y_{i_d},\\
    m'=y_{j_1} \cdots y_{j_e},
\end{align*}
one has
$m \prec m'$ if either $\deg(m) = d < e =\deg(m')$, or if $d=e$ and there exists some $k\in \{1,2,\ldots,d\}$
with $i_d=j_d,i_{d-1}=j_{d-1},\ldots,i_{k+1}=j_{k+1}$
but $i_k > j_k$.  It follows from Lemma~\ref{rank-two-flat-lemma} that for any retrograde pair $(j,i)$ with $F=\overline{\{j,i\}}$, every $k$ in $F \setminus \{i,j\}$ lies in $E_q$, so that $k > i$ and $y_j y_k \prec y_j y_i$.  Since also $j > i$, this makes
$y_j y_i$ the $\prec$-initial term 
in either \eqref{unoriented-retrograde-relation}
or \eqref{oriented-retrograde-relation}.

We next check that the relations $r(j,F), r^\pm(j,F)$
lie in $J_2=I_2^\perp \subset V^* \otimes V^*$,
with the pairing
defined by $(y_i y_j, x_k x_\ell)=\delta_{(i,j),(k,\ell)}$.
We  do the check here for $r^\pm(j,F)$;  the check for $r(j,F)$ is similar, but slightly easier.  One must check that $r^\pm(j,F)$ is orthogonal to three types of
generators of $I$ in $\VG(\OM)=\kk \langle \xx \rangle/I$:
\begin{align}
\label{squares-generators}
   x_k^2  & \text{ for } k=1,2,\ldots,n,\\
\label{commutator-generators}
   x_k x_\ell-x_\ell x_k
         & \text{ for }1 \leq k \leq \ell \leq n,\\
\label{signed-3-circuit-generators}
     \partial^\pm(C):=  \sgn{C}{m} x_k x_\ell
        +  \sgn{C}{\ell} x_k x_m 
        +  \sgn{C}{k} x_\ell x_m  & 
       \text{ for circuits }
         C=\{k,\ell,m\}
         \text{ of size three.}
\end{align}

Note that $r^\pm(j,F)$ pairs to zero with any commutator in \eqref{commutator-generators}, because  $r^\pm(j,F)$ is a sum of anti-commutators $[y_a, y_b]_-=y_a y_b + y_b y_a$.
Note also that whenever quadratic monomials $f(\yy), g(\xx)$ have disjoint $E^2$-support sets
\begin{align*}
\supp f(\yy)&:=\{ (i,j) \in E^2: y_i y_j \text{ appears in }f\text{ with nonzero coefficient}\},\\
\supp g(\xx)&:=\{ (k,\ell)\in E^2: x_k x_\ell \text{ appears in }f\text{ with nonzero coefficient}\},
\end{align*}
then one will have $(f(\yy),g(\xx))=0$. This already implies $r^\pm(j,F)$ pairs to zero with the $x_k^2$ in \eqref{squares-generators}.  It also shows that 
in order for $r^\pm(j,F)$ to have nonzero pairing with some $\partial^\pm(x_C)$ in \eqref{signed-3-circuit-generators}, one must have that $C=\{k,\ell,m\}$ satisfies $F=\{k\}\vee \{\ell\} \vee \{m\}$, and furthermore one must have $j \in C$.  In other words,
without loss of generality, $C=\{j,\ell,m\} \subset F$.  It remains to check that $r^\pm(j,F)$ still pairs to zero with $\partial^\pm(x_C)$ 
in this situation. Calculating the pairing, one finds 
\begin{align}
\notag
&( \partial^\pm(x_C), r^\pm(j,F))\\
\notag
&=
\left(
 \sgn{C}{m} x_j x_\ell
        +  \sgn{C}{\ell} x_j x_m 
        +  \sgn{C}{j} x_\ell x_m \,\, , \,\,
\sum_{ h \in F\setminus \{j\} } 
\chi_{\OM|_F}(j,h) \cdot [y_j,y_h]_-
\right)\\
\label{rank-two-OM-sign-vanishing}
&=
  \sgn{C}{m} \cdot \chi_{\OM|_F}(j,\ell)
+
\sgn{C}{\ell} \cdot \chi_{\OM|_F}(j,m).
\end{align}
Vanishing of the sum in \eqref{rank-two-OM-sign-vanishing} can be checked based on cases for the signed circuit $C=C_+ \sqcup C_-$ supported by $C=\{j,\ell,m\}$. One can relabel so that $|C_+|\geq|C_-|$, and hence $(|C_+|,|C_-|)=(3,0)$ or $(2,1)$.  As the indices $\ell,m$ play a symmetric role in \eqref{rank-two-OM-sign-vanishing}, one may
assume without loss of generality that the oriented matroid $\OM|_{\{j,\ell,m\}}$  matches that of one of these vector configurations in $\RR^2$:
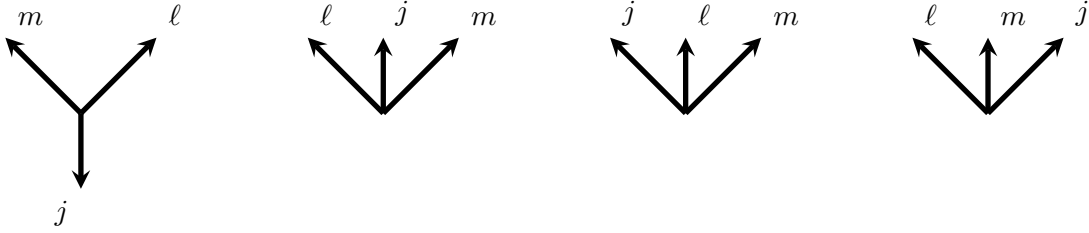
\begin{center}
\begin{tikzpicture}
  \draw[line width=2pt,black,-stealth](0,0)--(1,1) node[anchor=south west]{$\ell$};
    \draw[line width=2pt,black,-stealth](0,0)--(-1,1) node[anchor=south west]{$m$};
  \draw[line width=2pt,black,-stealth](0,0)--(0,-1) node[anchor=north east]{$j$};
 \draw[line width=2pt,black,-stealth](4,0)--(5,1) node[anchor=south west]{$m$};
    \draw[line width=2pt,black,-stealth](4,0)--(3,1) node[anchor=south west]{$\ell$};
  \draw[line width=2pt,black,-stealth](4,0)--(4,1) node[anchor=south west]{$j$};

  \draw[line width=2pt,black,-stealth](8,0)--(9,1) node[anchor=south west]{$m$};
    \draw[line width=2pt,black,-stealth](8,0)--(7,1) node[anchor=south west]{$j$};
  \draw[line width=2pt,black,-stealth](8,0)--(8,1) node[anchor=south west]{$\ell$};

  \draw[line width=2pt,black,-stealth](12,0)--(13,1) node[anchor=south west]{$j$};
    \draw[line width=2pt,black,-stealth](12,0)--(11,1) node[anchor=south west]{$\ell$};
  \draw[line width=2pt,black,-stealth](12,0)--(12,1) node[anchor=south west]{$m$};
\end{tikzpicture}
\end{center}
\noindent
In each case, one can check that the sum in \eqref{rank-two-OM-sign-vanishing} vanishes.

Once one has checked that the elements of $\GG$ lie in $J_2$,
Proposition~\ref{prop: graded-GB-via-Hilb-trick}
together with the following Hilbert series calculations will show that they form a quadratic (noncommutative) GB for $J$.  First note
that the $\cG$-standard monomials $m$ in $y_1,\ldots,y_n$ are those that
avoid all factors $y_j y_i$ in which $(j,i)$ are a retrograde pair, and these are exactly the monomials described in (iv).  Thus one has  
\begin{align*}
\Hilb(\kk\langle\yy\rangle / (\init_\prec(\GG)) ,t)
&= \sum_{\substack{\GG\text{-standard}\\ \text{monomials }m}} t^{\deg(m)} \\
&\overset{(a)}{=} \left( 1 + e_1 t + e_1^2 t^2 + e_1^3t^3+\cdots \right)
\cdots \left( 1 + e_r t + e_r^2 t^2 + e_r^3t^3+\cdots \right) \text{ with }e_p=|E_p|\\
&= \frac{1}{(1-e_1t) \cdots (1-e_rt)} \\
&\overset{(b)}{=}\frac{1}{\Hilb(A,-t)}\\
&\overset{(c)}{=}\Hilb(A^!,t)
=\Hilb(\kk\langle\yy\rangle/J,t).
\end{align*}
where equalities (a), (b), (c) above are justified as follows. Equality (a) follows from the description in (iv) of $\GG$-standard monomials as 
$m=m_1 \cdot m_2 \cdots m_r$
where $m_p$ is any noncommutative monomial 
in the variable set $\{y_j\}_{y \in E_p}$.
Equality (b) comes from \eqref{supersolvable-primal-Hilb}, and equality (c)  from Corollary~\ref{cor:Koszul-dual-Hilbert-series}.

Finally, to see that $\cG$ is a {\it reduced} Gr\"obner basis, note that Lemma~\ref{rank-two-flat-lemma} implies that for each retrograde pair $(j,i)$, the initial term $y_j y_i$ for the relations $r(j,i), r^\pm(j,i)$ cannot appear as a term in any of the other $r(k,\ell), r^\pm(k,\ell)$ with $(k,\ell) \neq (j,i)$.
\end{proof}

\subsection{Acyclicity and injectivity}
\label{sec: acyclicity-injectivity}

As an application of the Gr\"obner basis presentations for the algebras $A^!=\OS(M)^!, \VG(\OM)^!$ in Theorem~\ref{thm: shriek-presentations}, we explore a counterpart to 
an interesting
fact about $A=\OS(M), \VG(\OM)$:  their
Hilbert series contains a factor of $1+t$,
\begin{equation}
\label{primal-Hilbert-series-has-(1+t)-factor}
\Hilb(\OS(M),t)=\Hilb(\VG(\OM),t)=(1+t)\cdot H(t)
\end{equation}
and the remaining polynomial factor
$H(t) \in \ZZ[t]$ always has {\it nonnegative}  coefficients.  

This fact has several explanations: combinatorial, topological, and algebraic.  One algebraic explanation views  the Orlik-Solomon algebra $A=\OS(M)$ as an algebraic cochain complex
\begin{equation}
\label{OS-algebra-as-complex}
0 \rightarrow A_0 \overset{d}{\rightarrow} A_1 \overset{d}{\rightarrow} \cdots \overset{d}{\rightarrow} A_{r-1} \overset{d}{\rightarrow} A_{r} \rightarrow 0 
\end{equation}
whose differential $d$ is given by multiplication by an element $x=\sum_{i=1}^n c_i x_i$
in $A_1$.  The fact that $A$ is a quotient of an exterior algebra implies that $x^2=0$ in $A$,
so that indeed $d\circ d=0$.  

\begin{theorem} \cite[Thm. 7.2]{Yuzvinsky}
\label{thm: OS-exactness}
 The cochain complex \eqref{OS-algebra-as-complex} on $A=\OS(M)$ is exact whenever $x=\sum_{i=1}^n c_i x_i$ has
coefficients $c_i$ satisfying the following {\it genericity} condition:
    $\sum_{i \in F}c_i \neq 0$ in $\kk$ for all flats $F$ whose restriction $M|_F$ is not a nontrivial direct sum.
\end{theorem}

Thus whenever $x$ is generic, multiplication by $x$ on $A=OS(M)$ is ``as injective as possible", given the constraint that $x^2=0$.
This algebraically interprets the factor $H(t)$ in \eqref{primal-Hilbert-series-has-(1+t)-factor}, since
$tH(t)$ is the Hilbert series for the subspace of cocycles (= coboundaries) in the above cochain complex.

For $M, \OM$ supersolvable, the Koszul duals $A^!=\OS(M)^!,\VG(\OM)^!$
inherit a similar factorization
\begin{equation}
\label{shriek-Hilbert-series-has-a-factor}
\Hilb(A^!,t)=
\frac{1}{\Hilb(A,-t)}=\frac{1}{1-t} \cdot \frac{1}{H(-t)}=
(1+t+t^2+t^3+\cdots) \cdot H(-t)^{-1}.
\end{equation}
There is nothing that says, {\it a priori}, the rightmost factor $H(-t)^{-1}$ above should have
nonnegative coefficients.  However, this is a consequence of our next result.

\begin{definition}
    Let $M,\OM$ be supersolvable matroids or oriented matroids of rank $r$ 
    on the ground set $E=\{1,2,\ldots,n\}$,
with partition $\underline{E}$ as in Proposition~\ref{prop: bjorner-ziegler-characterization}.
For a fixed $d\geq 1$, say that the power sum $p_d(\yy)=\sum_{i=1}^n c_i y_i^d \in A^!_d \subset A^!=\OS(M)^!$ or $\VG(\OM)^!$ is {\it \underline{E}-generic} if
for each $q=1,2,\ldots,r$,
there exists $i \in E_q$ with the coefficient $c_i\neq 0$.

\end{definition}

\begin{theorem}
\label{thm: shrieks-have-right-NZD}
Let $M,\OM$ be supersolvable matroids or oriented matroids of rank $r$ on $E$,
with partition $\underline{E}$ as in Proposition~\ref{prop: bjorner-ziegler-characterization}.
Then for either $A^!=\OS(M)^!$ or $\VG(\OM)^!$,
right-multiplication $a \longmapsto ay$
by any \underline{E}-generic element $p_d(\yy)$ in $A^!_d$
gives an injective map $A^!  \longrightarrow  A^!$.
That is, every $\underline{E}$-generic $y$ is a right-non-zero-divisor on $A$.
\end{theorem}

\begin{proof}

Proceed by induction on the rank $r$.  In the base case $r=1$, the ring $A^!=\kk\langle y \rangle \cong \kk[y]$ is a univariate polynomial ring, and $y^d$ is a nonzero element of $A^!_d$, so $y^d$ is a nonzero divisor.

Preparing for the inductive step, segregate $E=F \sqcup E_r$
where $F:=F_{r-1}=E_1 \sqcup E_2 \sqcup \cdots \sqcup E_{r-1}$ is the modular coatom in the modular flag $\underline{F}$, and define the {\it early} and {\it late} variables:
$$
\{y_1,\ldots,y_n\}
=\underbrace{\{y_i\}_{i \in F}}_{\text{early}} \,\,  \sqcup \,\, 
\underbrace{\{y_j\}_{j \in E_r}}_{\text{late}}.
$$
Note that, by Theorem~\ref{thm:BEZ}, the restriction $M|_F$ is a rank $r-1$ supersolvable matroid
to which induction applies.  Also, note that the presentations in Theorem~\ref{thm: shriek-presentations} and the standard monomial bases
show that the early variables
generate a subalgebra of $A^!$ isomorphic to the Koszul dual $\OS(M|_F)^!$ or $\VG(\OM|_F)^!$.

The standard monomial basis shows that every $a$ in $A^!$ has a unique decomposition
\begin{equation}
\label{unique-decomp-by-last-hand-monomial}
a=\sum_{m} a(m) \cdot m
\end{equation}
where $m$ runs over all monomials in the late variables, and each $a(m)$ lies in the subalgebra generated by the early variables.
Grouping this more coarsely via $\deg(m)$, one obtains
a unique decomposition 
\begin{equation}
\label{late-degree-decomposition}
a=\sum_{\ell=0}^\infty a^{(\ell)} \quad \text{ where }
a^{(\ell)}:=\sum_{\substack{m:\\ \deg(m)=\ell}} a(m) \cdot m.
\end{equation}

In particular, 
$
p_d(\yy)=\sum_{i=1}^n c_i y_i^d = y^{(0)} + y^{(d)}.
$

Let $A^!_{(\ell)}$ denote the set of 
elements of the form $a^{(\ell)}$ above, so
there is a $\kk$-vector space decomposition
$$
A^!=\bigoplus_{\ell=0}^\infty A^!_{(\ell)}
$$
and also define
\begin{equation}
\label{handy-direct-sum-decomposition}
A^!_{(\geq \ell)}:=\bigoplus_{p=\ell}^\infty A^!_{(p)}
= A^!_{(\ell)} \oplus A^!_{(\geq \ell+1)}.
\end{equation}

We will use these two facts, justified below:
\begin{align*}
A^!_{(0)} \cdot A^!_{(0)} &\subseteq A^!_{(0)},\\
A^!_{(\geq p)} \cdot A^!_{(\geq q)}
&\subseteq A^!_{(\geq p+q)}.
\end{align*}
These follow ultimately from 
Lemma~\ref{rank-two-flat-lemma}, as we now explain.  One can use the Gr\"obner basis relations $r(j,i),r^\pm(j,i)$ for retrograde pairs $(j,i)$ appearing in Theorem~\ref{thm: shriek-presentations} as rewriting rules, performing the
division $f \rightarrow_{\GG} r$ and 
rewriting
$f$ as a sum $r$ of
$\GG$-standard monomials.  Lemma~\ref{rank-two-flat-lemma} implies that at each division step, one is always replacing
\begin{itemize}
    \item quadratic initial terms with no late variables by a sum of terms
with no late variables, 
\item quadratic initial terms with one late variable by a sum of terms with one or two late variables. 
\end{itemize}

Continuing the inductive step, assume $a \in A_q^!$ has
$a\cdot p_d(\yy)=0$, and we want to conclude that $a=0$.
Writing $a=\sum_{\ell=0}^q a^{(\ell)}$ as in \eqref{late-degree-decomposition},
we will show each $a^{(\ell)}=0$ via an inner induction on $\ell$.

In the inner induction base case $\ell=0$, write
$$
0=a\cdot p_d(\yy)={a^{(0)} \cdot y^{(0)}}
+ \underbrace{a^{(0)} \cdot y^{(d)} 
+ \sum_{\ell=1}^q a^{(\ell)} \cdot p_d(\yy)}_{\in A^!_{(\geq d)}}.
$$
so that $0 \equiv a^{(0)} \cdot y^{(0)} \bmod A^!_{(\geq d)}$.
By the direct sum decomposition \eqref{handy-direct-sum-decomposition}, this means $a^{(0)} \cdot y^{(0)}=0$.
By induction on the rank applied to $M|_F$, since $y^{(0)}$ is still generic for $M|_F$, this implies $a^{(0)}=0$.

In the inner inductive step,
assume $a\cdot p_d(\yy)=0$ and that 
$a^{(0)} = a^{(1)} = \cdots = a^{(\ell-1)} = 0$,
that is, $a$ lies in $A^!_{(\geq \ell)}$.  We wish to show that $a^{(\ell)} = 0$. 
Write
\begin{align*}
    0=a\cdot y &= (a^{(\ell)} + a^{(\ell+1)} + \cdots )\cdot (y^{(0)} + y^{(1)})\\
    &= a^{(\ell)}\cdot y^{(0)} 
    + \underbrace{a^{(\ell)}\cdot y^{(d)} 
    + ( a^{(\ell+1)} + a^{(\ell+2)} + \dots)\cdot p_d(\yy)}_{\in A^!_{(\geq \ell+1)}}
\end{align*}
so that $0 \equiv  a^{(\ell)}\cdot y^{(0)} \bmod A^!_{(\geq \ell+d)}$.

Write $a^{(\ell)} = \sum_{m} a(m) m$ as in \eqref{late-degree-decomposition}, so that $m$ runs
through all degree $\ell$ monomials in the late variables.  Note that for any early variable $y_i$ and any monomial $m$ of degree $\ell$ in the late variables, the division algorithm $f \rightarrow_\GG r$ and the form of the relations $r(j,i), r^\pm(j,i)$ in $\GG$ (again using Lemma~\ref{rank-two-flat-lemma}) will rewrite
$$
m \cdot y_i^d \equiv y_i^d \cdot m \bmod A^!_{(\geq \ell+d)}.
$$
Since $y^{(0)}$ is a sum of early variables,
similarly 
$
m \cdot y^{(0)} \equiv y^{(0)} \cdot m \bmod A^!_{(\geq \ell+d)},
$
which implies
\begin{align*}
    a^{(\ell)}\cdot y^{(0)}
    = \sum_{m} a(m) \cdot m \cdot y^{(0)} 
    \equiv \sum_{m} a(m) \cdot y^{(0)} \cdot m  \bmod A^!_{(\geq \ell+d)}
\end{align*}

Hence one concludes that
$0\equiv \sum_{m} a(m) \cdot y^{(0)} \cdot m \bmod A^!_{(\geq \ell+d)}$.
Since $\sum_{m} a(m)\cdot y^{(0)} \cdot m$ lies in $A^!_{(\ell)}$, by the direct sum decomposition \eqref{handy-direct-sum-decomposition}, it must vanish.  But by the uniqueness in \eqref{unique-decomp-by-last-hand-monomial}, this implies each $a(m) \cdot y^{(0)}=0$.  Then by induction on $r$, each $a(m)=0$.
Hence $a^{(\ell)}=0$, as desired.
\end{proof}

One has the following corollary to
Theorems~\ref{thm: OS-exactness} and \ref{thm: shrieks-have-right-NZD}.

\begin{cor}
\label{cor: equiv-hilb-factorizations}
Let $M, \OM$ be a matroid or oriented matroid, 
and $G$ a group of automorphisms, that is, 
a subgroup of $\Aut(M)$ or $\Aut(\OM)$.
Consider $G$ as a group of $\kk$-algebra automorphisms of $A:=\OS(M)$ or $\VG(\OM)$.

\begin{enumerate}
\item[(i)] Any $x \in A_1$ which is $G$-fixed and generic
in the sense of Theorem~\ref{thm: OS-exactness}
(e.g., $x=\sum_{i=1}^n x_i$ when $\kk$ has characteristic zero) gives rise to
a factorization in $R_\kk(G)[[t]]$
$$
\Hilb_{\eq}(\OS(M),t) 
= (1+t) \cdot H(t) 
$$
where $tH(t)$ is the equivariant Hilbert series
for the cocycles (=coboundaries) of the $\kk G$-module complex in \eqref{OS-algebra-as-complex}.

\item[(ii)]
Assuming that $M$ is supersolvable with
decomposition $\underline{E}$ as in Proposition~\ref{prop: bjorner-ziegler-characterization}, any $y \in A^!_1$ which is $G$-fixed and $\underline{E}$-generic
(e.g., $y=\sum_{i=1}^n  y_i$)
gives a factorization in $R_\kk(G)[[t]]$
$$
\Hilb_{\eq}(\OS(M)^!,t) 
=  \frac{1}{1-t} \cdot H^!(t)
$$
where $H^!(t)$ is the equivariant Hilbert series for
the quotient $\kk G$-module $A^!/A^! y$.

\item[(iii)]
Assuming that $\OM$ is supersolvable with
decomposition $\underline{E}$ as in Proposition~\ref{prop: bjorner-ziegler-characterization}, any $p_2(\yy) \in A^!_2$ which is $G$-fixed and $\underline{E}$-generic\footnote{One cannot always find such $G$-fixed $\underline{E}$-generic elements in $A_1^!$, e.g., $\sum_{i=1}^n y_i$ is $\underline{E}$-generic, but not always $G$-fixed.}
(e.g., $p_2(\yy)=\sum_{i=1}^n  y_i^2$)
gives a factorization in $R_\kk(G)[[t]]$
$$
\Hilb_{\eq}(\VG(\OM)^!,t) 
=  \frac{1}{1-t^2} \cdot H^!(t)
$$
where $H^!(t)$ is the equivariant Hilbert series for
the quotient $\kk G$-module $A^!/A^!\cdot p_2(\yy)$.
\end{enumerate}
\end{cor}

\noindent
Examples of the factorizations in the various parts of Corollary~\ref{cor: equiv-hilb-factorizations} appear later:
\begin{itemize}
    \item Part (i) is illustrated by \eqref{toy-OS-factorization}, \eqref{rank-one-illustration-of-OS-factorization} \eqref{rank-two-OS-equivariantly}.
    \item Part (ii) is illustrated by \eqref{toy-OS-shriek-factorization}, \eqref{rank-one-illustration-of-OS-shriek-factorization}, \eqref{rank-two-OS-shriek-equivariantly}.
    \item Part (iii) is illustrated by \eqref{rank-one-illustration-of-VG-factorization}, \eqref{type-B-inflated-Lusztig-Stanley-formula}. 
\end{itemize}

\section{Examples:  Boolean matroids and matroids of low rank}
\label{sec: example-section}

Before developing further theory for supersolvable matroids and oriented matroids, we digress to discuss
the action of symmetries in a few of our earlier examples, illustrating the results so far.

\subsection{Boolean matroids}
\label{sec: Boolean-matroids-revisited}
We return to Example~\ref{ex: Boolean-matroid-1}
and the Boolean matroid $M=U_{n,n}$. Although
$M=U_{n,n}$ is orientable,
we will focus here on $\OS(M)$, where a bit more is known about the action of symmetries, rather than on $\VG(\OM)$. The discussion of $\VG(\OM)$ is deferred
to Example~\ref{Boolean-matroids-re-revisited} later.

The Boolean matroid $M$ of
of rank $n$ has no circuits,
so $A=\OS(M)=\wedge V=\wedge(x_1,\ldots,x_n)$,
and $A^!=\OS(M)^!=\Sym V=\kk[y_1,\ldots,y_n]$, swapping
the roles of $A,A^!$ from Examples~\ref{standard-example-exterior-symmetric}, \ref{standard-example-exterior-symmetric-2}, \ref{standard-example-exterior-symmetric-3}.
Here  $\Aut(M)=\symm_n$, and both $V, V^*$ carry the defining
representation of $\symm_n$ permuting the subscripts of the variables $x_i$ or $y_i$.

Thus $V$ is the defining representation of $\symm_n$ by permutation matrices.  Assuming that $\kk$ has characteristic zero, $V, V^*$ both decompose into irreducible $\kk \symm_n$-modules as 
$$
V \cong V^* \cong \specht^{(n)} \oplus \specht^{(n-1,1)}
$$
where $\specht^\lambda$ denotes the simple $\kk \symm_n$-module indexed by a partition $\lambda$ of $n$;  here $\specht^{(n)}$ is the trivial $\symm_n$-representation, while $\specht^{(n-1,1)}$ is the irreducible reflection representation of $\symm_n$.
Consequently, in this situation, 
\begin{align}
\label{wedge-tensor-decomp}
    A&= \wedge( \specht^{(n)} \oplus \specht^{(n-1,1)} ) 
    =  \wedge \specht^{(n)} \otimes \wedge \specht^{(n-1,1)} \\
\label{sym-tensor-decomp}
    A^!&= \Sym ( \specht^{(n)} \oplus \specht^{(n-1,1)} ) 
    =  \Sym \specht^{(n)} \otimes \Sym \specht^{(n-1,1)} ,
\end{align}
and the factorizations
in Corollary~\ref{cor: equiv-hilb-factorizations} become
\begin{align}
\label{toy-OS-factorization}
    \Hilb_\eq(A,t)&= (1+t) \Hilb_\eq( \wedge \specht^{(n-1,1)},t) \\
\label{toy-OS-shriek-factorization}
    \Hilb_\eq(A^!,t)&= \frac{1}{1-t} \Hilb_\eq( \Sym \specht^{(n-1,1)},t).
\end{align}
Both \eqref{toy-OS-factorization} and \eqref{toy-OS-shriek-factorization} can be refined to explicit $\kk\symm_n$-irreducible expansions.  For \eqref{toy-OS-factorization}, since it is known that $\wedge^i \specht^{(n-1,1)} \cong \specht^{(n-i,1^i)}$, one has
$$
\Hilb_\eq( \wedge \specht^{(n-1,1)},t) 
=\sum_{i=0}^{n-1} [\specht^{(n-i,1^i)} ] \, t^i.
$$
For \eqref{toy-OS-shriek-factorization}, one can extend the tensor decomposition \eqref{sym-tensor-decomp}.  The $\symm_n$-invariant subalgebra of $\kk[\yy]$ is $\kk[\yy]^{\symm_n}=\kk[e_1,e_2,\ldots,e_n]$ where $e_k=e_k(\yy)$ is the $k^{th}$ {\it elementary symmetric function} in the variables $\yy$, and the theory of Cohen-Macaulay rings gives a graded
$\kk\symm_n$-module tensor product decomposition
$$
\kk[\yy] \cong \kk[e_1,e_2,\ldots,e_n] \,\, \otimes \,\,  
\kk[\yy]/(e_1,e_2,\ldots,e_n)
$$
where $\kk[\yy]/(e_1,e_2,\ldots,e_n)$ is the type $A$ {\it coinvariant algebra}.  Hence one has
\begin{equation}
\label{Lusztig-Stanley-formula}
    \begin{aligned}
    \Hilb_\eq(A^!,t) &= \Hilb(\kk[e_1,e_2,\ldots,e_n],t) \cdot 
\Hilb_\eq(\kk[\yy]/(e_1,e_2,\ldots,e_n),t)\\
&=\frac{1}{(1-t)(1-t^2) \cdots (1-t^n)} \cdot 
\sum_Q [\specht^{\lambda(Q)}] \, t^{\mathrm{maj}(Q)}
\end{aligned}
\end{equation}
where the sum on the right, due to Lusztig and Stanley \cite[Prop.~4.11]{Stanley-invariants}),
has 
$Q$ running over all {\it standard Young tableaux} with $n$ cells, with $\lambda(Q)$ the partition shape of $Q$, and $\mathrm{maj}(Q)$ the sum of all entries $i$ in $Q$ for which $i+1$ appears
weakly southwest of $i$ (using English notation for tableaux).

We note for future reference in Section~\ref{sec: braid-OS-permutation-reps} that $\symm_n$ permutes the monomial basis 
$
\{ \yy^\aa= y_1^{a_1} \cdots y_n^{a_n} : \aa \in \NN^n\}
$
of $A^!=\kk[\yy]$, making each graded component $A^!_i$ of $A^!$ a {\it permutation representation}.

\subsection{Rank one matroids}
\label{sec: rank-one-matroids}
A simple rank one matroid $M$ has ground set $E=\{e\}$ of size one and no circuits.  It is always orientable, and has
\begin{align*}
    A&=\OS(M) \cong \VG(\OM) = \kk[x]/(x^2)\\
    A^! &= \kk[y].
\end{align*}
The only difference between $M, \OM$ arises when one takes into account symmetries.  The matroid $M$ has no nontrivial automorphisms, while the oriented matroid $\OM$ carries the action of the two-element group $G=\Aut(\OM) =\symm^\pm_1 \cong \ZZ/2\ZZ$.  Assuming that the characteristic of $\kk$ is not $2$, then the generator
of $G$ negates both $x,y$ when it acts on $A=\VG(\OM)=\kk[x]/(x^2)$ or  $A^!=\VG(\OM)^!=\kk[y]$.
Denoting the class of this nontrivial $1$-dimensional representation by $\epsilon$ in the Grothendieck ring
$$
R_\kk(G)\cong \ZZ[\epsilon]/(\epsilon^2-1)
$$
then in the power series ring $R_\kk(G)[[t]]$ one has 
\begin{align}
\label{rank-one-illustration-of-OS-factorization}
\Hilb_\eq(\OS(M),t)&=1+t,\\
\notag
\Hilb_\eq(\VG(\OM),t)&=1+\epsilon t,\\
\notag
& \\
\label{rank-one-illustration-of-OS-shriek-factorization}
\Hilb_\eq(\OS(M)^!,t)
&=1+ t + t^2 + t^3 + \cdots = \frac{1}{1-t}\\
\notag
\Hilb_\eq(\VG(\OM)^!,t)&=1+\epsilon t + \epsilon^2 t^2 + \epsilon^3 t^3 + \cdots  =\frac{1}{1-\epsilon t}\\
\label{rank-one-illustration-of-VG-factorization}
&=1+\epsilon t + t^2 + \epsilon t^3 + \cdots = \frac{1+\epsilon t}{1-t^2}.
\end{align}

\subsection{Rank two matroids}
\label{sec: rank-two-revisited}
As discussed in Example~\ref{ex: rank-two-matroid-1},
a simple rank two matroid $M$ on ground set 
$E=\{1,2,\ldots,n\}$ is always orientable, and supersolvable. Any rank $1$ flat, such as $F=\{1\}$ is a
modular coatom, and one has the corresponding set partition decomposition $\underline{E}=(E_1,E_2)=(\{1\},\{2,3,\ldots,n\})$ with $(e_1,e_2)=(1,n-1)$.
Therefore,
\begin{equation}
\label{rank-two-matroid-OS-Hilbs}
\begin{aligned}
\Hilb(\OS(M),t)&=\Hilb(\VG(\OM),t)=(1+t)(1+(n-1)t)
=1+nt +(n-1)t^2,\\
\Hilb(\OS(M)^!!,t)&=\Hilb(\VG(\OM)^!,t)\\
&=\frac{1}{(1-t)(1-(n-1)t)}\\
&=1 + (1+(n-1)) t + 
 (1+(n-1)+(n-1)^2) t^2 + \cdots \\
  &=\sum_{i=0}^\infty f(n,i) t^i \,\, \text{ where }f(n,i):=\sum_{j=0}^i (n-1)^i=
\frac{(n-1)^{i+1}-1}{n-2}.
\end{aligned}
\end{equation}

In considering symmetries, it is somewhat easier to compute with
$\OS(M)$, rather than $\VG(\OM)$.  The matroid $M$ has as its symmetries the full symmetric group $G=\Aut(M)=\symm_n$, arbitrarily permuting $E=\{1,2,\ldots,n\}$.
It is also helpful to introduce a notation $\varphi_\lambda$ for the class $\left[\kk[\symm_n/\symm_\lambda]\right]$ within $R_\kk(\symm_n)$ of the $\symm_n$-permutation representation on the cosets of the Young subgroup 
$
\symm_\lambda:=\symm_{\lambda_1} \times \cdots \times \symm_{\lambda_\ell}
$
where $\lambda=(\lambda_1,\lambda_2,\ldots,\lambda_\ell)$ is a partition of $n=|\lambda|:=\sum_{i=1}^\ell \lambda_i$.  Hence, if $\kk$ were a field of characteristic zero (which we do {\it not} assume here), then this class $\varphi_\lambda$ corresponds to the product of 
complete homogeneous symmetric functions 
$$
h_\lambda:=h_{\lambda_1} \cdots h_{\lambda_\ell}
$$
under the {\it Frobenius characteristic isomorphism}
$R_\kk(\symm_n) \cong \Lambda_n$, where $\Lambda_n$ are the degree $n$ homogeneous symmetric functions $\Lambda(z_1,z_2,\ldots)_n$ in infinitely many variables.

One finds that $\OS(M)_1$ carries the defining permutation representation of $\symm_n$ permuting coordinates in $\kk^n$,
whose class in $R_\kk(\symm_n)$ is $\varphi_{(n-1,1)}$.
Introducing the $\kk \symm_n$-submodule
\begin{equation}
\label{char-free-reflection-rep-defn}
\specht^{(n-1,1)}:=\{\xx \in \kk^n: x_1+\cdots+x_n=0\},
\end{equation}
the quotient $\varphi_{(n-1,1)}/\specht^{(n-1,1)}$ carries the trivial
$\kk \symm_n$-module $\specht^{(n)}$, giving this identity in $R_\kk(\symm_n)$:
\begin{equation}
    \label{defining-S_n-rep-decomposed}
[\OS(M)_1]=\varphi_{(n-1,1)}
=[\specht^{(n)}]+[\specht^{(n-1,1)}]
=1 + [\specht^{(n-1,1)}].
\end{equation}
The Hilbert series in \eqref{rank-two-matroid-OS-Hilbs} have the following equivariant lifts to $R_\kk(\symm_n)[[t]]$:
\begin{align}
\label{rank-two-OS-equivariantly}
\Hilb_\eq(\OS(M),t)&=(1+t)(1+[\specht^{(n-1,1)}] t)
=1+\varphi_{(n-1,1)} t + [\specht^{(n-1,1)}] t^2,\\
\label{rank-two-OS-shriek-equivariantly}
\Hilb_\eq(\OS(M)^!,t)&=\frac{1}{(1-t)(1-[\specht^{(n-1,1)}] t)}\\
\notag
&=1 + (1+[\specht^{(n-1,1)}]) t + 
 (1+ [\specht^{(n-1,1)}] +[\specht^{(n-1,1)}]^2) t^2 + \cdots \\
\label{rank-two-OS-shriek-as-geometric-series}
&=\sum_{i=0}^\infty t^i \left( \sum_{k=0}^i [\specht^{(n-1,1)}]^k \right)
\end{align}

We find the next proposition somewhat unexpected. 
\begin{prop}
\label{prop: rank-two-OS-shrieks-are-perm-reps}
    In $R_\kk(\symm_n)$, the element
    $
    [\OS(M)^!_i] = \sum_{k=0}^i [\specht^{(n-1,1)}]^k 
    $$
    $
    is the class of a permutation representation, expressible in the following form:

    \[[\OS(M)^!_i]
=\begin{cases} 
\varphi_{(n)}+
\displaystyle\sum_{d=2}^i a_{d,i}\,  \varphi_{(n-d, 1^d)}, & i\text{ even},\\
\varphi_{(n-1,1)}+
\displaystyle\sum_{d=2}^i b_{d,i}\,  \varphi_{(n-d, 1^d)}, & i\text{ odd},
\end{cases}\]
where $\{a_{d,i}\}, \{b_{d,i}\}$ are positive integers, independent of $n$, given by sums of Stirling numbers:  
\begin{align*}
a_{d,i}&=\sum_{k=\lfloor\frac{d}{2}\rfloor}^{\frac{i}{2}} S(2k-1,d-1) \text{ for }i\text{ even,}\\
b_{d,i}&=\sum_{k=\lfloor\frac{d-1}{2}\rfloor}^{\frac{i-1}{2}} S(2k,d-1) \text{ for }i\text{ odd.}
\end{align*}

\end{prop}
\begin{proof}  The following identity is established in \cite[Prop.~7.6, Thm~7.7]{Sundaram-reflection-rep}  for $j\ge 2$. 
\begin{equation}
\label{key-assertion-on-powers-of-reflection-rep}
    [\specht^{(n-1,1)}]^{j}+  [\specht^{(n-1,1)}]^{j-1}=\sum_{d=2}^j S(j-1,d-1) \,\, \varphi_{(n-d, 1^d)}.
\end{equation}
Since the proofs in \cite{Sundaram-reflection-rep} are phrased in terms of symmetric functions, over a ground field of characteristic zero, we explain  why this  identity still holds  in the Grothendieck ring $R_\kk(\symm_n)$ for arbitrary fields $\kk$. 
 Sundaram   constructs an explicit $\kk\symm_n$-module  realizing the $j$th  tensor power of the $\symm_n$-permutation module $V_{1,n}$, whose class is $\varphi_{(n-1,1)}$, and decomposes it in terms of the coset permutation submodules $V_{d,n}=\left(\kk \symm_d \right) \uparrow_{ \symm_{n-d} \times \symm_1^d }^{\symm_n}$ whose class is $\varphi_{(n-d,1^d)}$, 
 obtaining \cite[Eqn. 18, Lemma.~6.1]{Sundaram-reflection-rep})
 \begin{equation}\label{eqn:Eq.18-Lemma6.1-Sun21}
 V_{1,n}^{\otimes j}=\sum_{d=1}^{\min{(n,j)}} S(j, d)\, V_{d,n}.
 \end{equation}

In addition we will use the following three facts (1),(2),(3). 
 \begin{enumerate}
\item $[\specht^{(n-1,1)}\downarrow^{\symm_n}_{\symm_{n-1} \times \symm_1}]=[V_{1,n-1}]$. 
We show this as follows.   

First, if $\specht^{(n-1)}$ is the span of a fixed standard basis vector, and $V_{1,n-1}$ is the span of the $n-1$ non-fixed standard basis vectors, we clearly have 
$V_{1,n}\downarrow =
\specht^{(n-1)} \oplus V_{1,n-1},$ which in turn gives  
\begin{equation}\label{eqn:restrict-V1n-A} 
[V_{1,n}\downarrow] =
[\specht^{(n-1)}] + [V_{1,n-1}].
\end{equation}

Now recall the definition in \eqref{char-free-reflection-rep-defn}
of $\specht^{(n-1,1)}$ and the  Grothendieck group identity
\eqref{defining-S_n-rep-decomposed}.
The discussion around  \eqref{char-free-reflection-rep-defn}
 and 
\eqref{defining-S_n-rep-decomposed} in effect establishes the existence of a short exact sequence of $\kk\symm_n$-modules 
\[0 \rightarrow \specht^{(n-1,1)}\rightarrow V_{1,n} \rightarrow \specht^{(n)} \rightarrow 0,\]
which restricts to the same sequence as $(\symm_{n-1}\times \symm_1)$-modules.
Hence we have 
\begin{equation}\label{eqn:restrict-V1n-B} 
[V_{1,n}\downarrow ]=[\specht^{(n)}\downarrow] +[\specht^{(n-1,1)}\downarrow ]
\end{equation}

Since $\specht^{(n)}\!\downarrow\,=\specht^{(n-1)}$, comparing \eqref{eqn:restrict-V1n-A} and \eqref{eqn:restrict-V1n-B} we  obtain $[\specht^{(n-1,1)}\downarrow ]=[V_{1,n-1}].$

\item \cite[Corollary 4.3.8, Part (2)]{webb} Transitivity of induction;
 \item  \cite[Corollary 4.3.8, Part (4)]{webb} For a finite group $G$ and subgroup $H$, and $\kk G$-module $U$, $\kk H$-module $V$, over any field $\kk$,
 \begin{equation*}\label{tensor-restriction-induction-relation}
U\otimes (V\uparrow_H^G)\, \cong (U\downarrow_H\otimes V)\uparrow_H^G.
\end{equation*}
\end{enumerate}

In the present situation we have $G=\symm_n$, $H=\symm_{n-1}\times\symm_1$, $V=1_{\symm_{n-1}\times\symm_1}$, so that the class of $V\uparrow_H^G$ is $\varphi_{(n-1,1)}$,  and 
$U={\specht^{(n-1,1)}}^{\otimes j-1}$. 
Following the proof of  \cite[Prop.~7.6]{Sundaram-reflection-rep}),  we have 
\begin{align*}
[\specht^{(n-1,1)}]^j
+[\specht^{(n-1,1)}]^{j-1}
&=[\specht^{(n-1,1)}]^{j-1} 
\left(
[\specht^{(n)}]
+[\specht^{(n-1,1)}]
\right)\\
&=[U \otimes V_{1,n}] \quad \text{ by \eqref{defining-S_n-rep-decomposed} and definition of }U,\\
&=[U \otimes \left( \one_H \right)\uparrow_H^G] \quad \text{ by definition of }H,G,V_{1,n},\\
&=[\left( U\downarrow^G_H \right)\uparrow_H^G] \quad \text{ by item (3) above,} \\
&=\left[
\left( (\specht^{ (n-1,1) })^{ \otimes (j-1)} \downarrow^{\symm_n}_{\symm_1 \times \symm_{n-1}} \right)
 \uparrow_{\symm_1 \times \symm_{n-1}}^{\symm_n}\right] \\
&=[V_{1, n-1}^{\otimes j-1}\uparrow_{\symm_1 \times \symm_{n-1}}^{\symm_n}] \ \text{by item (1) above,} \\
&=\sum_{d'=1}^{\min{(n-1,j-1)}} S(j-1,d')\,[V_{d',n-1}\big\uparrow_{\symm_{n-1}}^{\symm_n}]\ \text{using~\eqref{eqn:Eq.18-Lemma6.1-Sun21}}\notag\\
&=\sum_{d'=1}^{\min{(n-1,j-1)}} S(j-1,d') \, [V_{d'+1,n}]\ \text{by item (2) above,}\notag\\
&=\sum_{d=2}^{\min{(n,j)}} S(j-1, d-1)\,[V_{d,n}].\notag
\end{align*}

Hence for $i\ge 2$, we have 
$$
\begin{aligned}
[\OS(M)^!_i]= \sum_{k=0}^i [\specht^{(n-1,1)}]^k
&=\begin{cases}
[\specht^{(n)}]+\sum_{k=1}^{i/2}\left([\specht^{(n-1,1)}]^{2k}+  [\specht^{(n-1,1)}]^{2k-1} \right),& i\text{ even},\\
 [\specht^{(n)}]+[\specht^{(n-1,1)}]+\sum_{k=1}^{(i-1)/2} \left([\specht^{(n-1,1)}]^{2k+1}+  [\specht^{(n-1,1)}]^{2k} \right),& i\text{ odd},
\end{cases}\\
&= \begin{cases}
\varphi_{(n)}+
\sum_{k=1}^{i/2}\, (\sum_{d=2}^{2k} S(2k-1,d-1)\, \varphi_{(n-d, 1^d)}),& i\text{ even},\\
\varphi_{(n-1,1)}+
\sum_{k=1}^{(i-1)/2}\, (\sum_{d=2}^{2k+1} S(2k,d-1)\, \varphi_{(n-d, 1^d)}),& i\text{ odd}.
\end{cases}
\end{aligned}
$$
Interchanging the order of summation then gives the assertion of the proposition for $i \geq 2$. For $i=0,1$ it is easy to check that $[\OS(M)^!_0]=\varphi_{(n)}$ and $[\OS(M)^!_1]=\varphi_{(n-1,1)}$.
\end{proof}

\begin{remark}
Rather than all of $\symm_n$, one might restrict the action on $\OS(M)$ to the {\it dihedral group}\footnote{
These are symmetries of the rank two {\it oriented} matroid $\OM$, although we are ignoring the action on
$\VG(\OM)$ here.} 
$$
W=I_2(n)=\langle r,s: r^n=s^2=1,srs=r^{-1} \rangle
$$
of order $2n$.
Since $\symm_n$ acts on each $\OS(M)^!_i$ via permutation representations, the same must hold for $W$ by restriction.  One can check via character computations (omitted here)
that $\OS(M)^!_i$ is always a nonnegative combination of these four permutation
representations:
\begin{itemize}
\item the {\it trivial} representation, 
\item the {\it defining} representation $\rho_{\mathrm{def}}$ on $E=\{1,2,\ldots,n\}$ with $r=(12 \cdots n)$, $s(i)=n+1-i$,
\item 
the {\it regular representation} $\regrep:=\kk W$, and
\item when $n$ is even, the {\it half-regular representation} $\halfregrep:=\kk [W/Z_W]$ where $Z_W:=\{1,r^{\frac{n}{2}}\}$.
\end{itemize}
One has these expansions in $R_\kk(W)$, where $f(n,i):=\dim_\kk \OS^!_i=
\frac{(n-1)^{i+1}-1}{n-2}$ as in \eqref{rank-two-matroid-OS-Hilbs}:
$$
    [\OS(M)^!_i] = \begin{cases}
\frac{1}{2n} \left[ f(n,i) -1 \right] \cdot \regrep + 1 & \text{ if }n\text{ is odd, and }i \text{ is even},\\
    \frac{1}{2n} \left[ f(n,i) -n \right] \cdot \regrep + \rho_{\mathrm{def}} & \text{ if }n\text{ is odd, and }i \text{ is odd},\\ 
    & \\
\frac{1}{n} \left[ f(n,i) - n \cdot \frac{i}{2}- 1 \right] \cdot \halfregrep + \frac{i}{2} \cdot  \rho_{\mathrm{def}} + 1 & \text{ if }n\text{ is even, and }i \text{ is even},\\
    \frac{1}{n} \left[ f(n,i) - n \cdot \frac{i+1}{2} \right] \cdot \halfregrep + \frac{i+1}{2} \cdot  \rho_{\mathrm{def}} & \text{ if }n\text{ is even, and }i \text{ is odd}.\\    
    \end{cases}
    $$
\end{remark}

\begin{remark}
    It was observed earlier that uniform matroids
    $U_{r,n}$ are supersolvable if and only if $r \in \{1,2,n\}$.  This means that $M=U_{3,n}$ for $n \geq 4$ are {\it not supersolvable}, and in fact, $A=\OS(M)$ are not Koszul algebras, and not even quadratic.  If one nevertheless tries to define virtual $\symm_n$-characters 
    $\{ [A^!_i] \}_{i \geq 0}$ in terms of the genunine
    characters $\{[A_i]\}_{i \geq 0}$ 
    via the recurrence \eqref{equivariant-recurrence-for-shrieks},
    then already $[A^!_3]$ are not genuine characters once $n \geq 4$.
\end{remark}

\section{Branching rules for supersolvable matroids}
\label{sec: supersolvable-branching-rules}

Let $M, \OM$ be supersolvable matroids or
oriented matroids of rank $r$ on ground set $E$, with a modular complete flag $\underline{F}$ and decomposition $\underline{E}$ as in Proposition~\ref{prop: bjorner-ziegler-characterization}.  If $F=F_{r-1}$ denotes the modular coatom within
the flag $\overline{F}$, then the restrictions $M|_F, \OM|_F$ are again supersolvable.  Furthermore, the formulas \eqref{supersolvable-primal-Hilb}, \eqref{supersolvable-dual-Hilb} for the Hilbert series of the rings $A=\OS(M), \VG(\OM)$ and their Koszul duals $A^!$ show that they are closely related to the Hilbert series of the same rings $B, B^!$ for $M|_F, \OM|_F$:
\begin{align*}
\Hilb(A,t)=(1+e_1t) \cdots (1+e_{r-1}t) (1+e_rt) &=\Hilb(B,t) \cdot (1+e_rt)\\
\Hilb(A^!,t)=\frac{1}{(1-e_1t) \cdots (1-e_{r-1}t) (1-e_rt)} 
&=\Hilb(B^!,t) \cdot \frac{1}{1-e_rt}
\end{align*}
which one can rewrite suggestively as follows, for comparison with Proposition~\ref{prop: general-branching}:

\begin{align}
\label{suggestive-primal-Hilb-rewriting}
\Hilb(A,t)&=\Hilb(B,t) +t \cdot e_r \cdot \Hilb(B,t)\\
\label{suggestive-shriek-Hilb-rewriting}
\Hilb(A^!,t)&=\Hilb(B^!,t) + t \cdot e_r \cdot \Hilb(A^!,t).
\end{align}
This suggests considering  a group $G$ of automorphisms of $M$ or $\OM$, and how its
action on $A, A^!$ restricts to the setwise $G$-stabilizer subgroup of the modular coatom $F$
$$
H:=\{g \in G: g(F)=F\}.
$$
Note that $H$ also permutes the ground set elements $E_r:=E \setminus F$; in the case where $G$ acts on the oriented matroid $\OM$, so that $G$ acts via signed permutations in $\symm_n^\pm$ on $E$ as in Definition~\ref{defn: OM-automorphisms}, then $H$ acts via signed permutations on $E_r$.  This gives rise to either a permutation or signed permutation $\kk H$-module $\mathcal{X}:=\kk[E_r]$, which in particular is self-contragredient.  Our goal in the next two subsections is to prove Theorem~\ref{thm: supersolvable-branching-exact-sequences} below,
which not only lifts \eqref{suggestive-primal-Hilb-rewriting}, \eqref{suggestive-shriek-Hilb-rewriting}
to these two branching relations in $R_\kk(H)$
\begin{align}
\label{supersolvable-primal-branching-relation}
    [A_i\downarrow] &= [B_i]+[{\mathcal X}] \cdot [B_{i-1}]\\
\label{supersolvable-shriek-branching-relation}
    [A^!_i\downarrow] &= [B^!_i]+[{\mathcal X}]\cdot \left( [A^!_{i-1}\downarrow ]\right).
\end{align}
(equivalent by Proposition~\ref{prop: general-branching} as $\cX^* \cong \cX$ ), but also lifts them to short exact sequences.

\begin{theorem}
\label{thm: supersolvable-branching-exact-sequences}
With the above notations, and letting $\kk$ be any field, one has the following short exact sequences of graded $\kk H$-modules:
\begin{itemize}
\item[(i)]
$ 0 \longrightarrow B \longrightarrow A\downarrow^G_H \longrightarrow \mathcal{X} \otimes B(-1)\longrightarrow 0, $
\item[(ii)]
$ 0 \longrightarrow \mathcal{X} \otimes \left( A^!\downarrow^G_H\right)(-1) \longrightarrow A^!\downarrow^G_H \longrightarrow B^! \longrightarrow 0. $
\end{itemize}

\end{theorem}

\subsection{Proof of Theorem~\ref{thm: supersolvable-branching-exact-sequences}(i)} 

The injective maps  
$B \rightarrow A$ from 
Theorem~\ref{thm: supersolvable-branching-exact-sequences}(i) are instances of injections of Orlik-Solomon and Varchenko-Gel'fand algebras coming from their flat decompositions 
\begin{align*}
\OS(M) &= \bigoplus_{F \in \cF} \OS(M)_F,\\
\VG(\OM) &= \bigoplus_{F \in \cF} \OS(M)_F
\end{align*}
discussed in
Section~\ref{sec: flat-decomposition}.
The NBC monomial $\kk$-bases
from Remark~\ref{rem: NBC-bases-respect-flats}
show that for any flat $F$ in $\cF$, one
has $\kk$-algebra inclusions  (see \cite[Prop.~3.30]{OrlikTerao}, \cite[Prop.~2.5]{Yuzvinsky}, and
\cite[Prop.~5.6]{Brauner})
\begin{align*}
\OS(M|_F) \cong \bigoplus_{F' \leq F} \OS(M)_{F'} \hookrightarrow \OS(M),\\
\VG(\OM|_F) \cong \bigoplus_{F' \leq F} \VG(\OM)_{F'} \hookrightarrow \VG(\OM).
\end{align*}

Since $A=\OS(M)$ or $\VG(\OM)$ and $B=\OS(M|_F)$ or $\VG(\OM|_F)$,
this explains the injection $B \rightarrow A$
from the sequence Theorem~\ref{thm: supersolvable-branching-exact-sequences}(i).
In fact, the entire sequence actually holds in slightly more generality, and for Orlik-Solomon algebras
was essentially observed by Orlik and Terao \cite[Lemma 3.80]{OrlikTerao}.
One does not need to assume that $M, \OM$ are supersolvable, only that $F$ is a modular coatom within its lattice of flats $\cF$.  Keeping the same notations so that $A=\OS(M), \VG(\OM)$, and $B=\OS(M|_F),\VG(\OM|_F)$,
with $H$ the setwise $G$-stabilizer of $F$
with a group of the autmorphisms $G$, one has the following.

\begin{prop}
Any modular coatom $F$ gives rise to a $\kk$-vector space direct sum decomposition
$$
A = B \oplus 
\left( \bigoplus_{j \in E \setminus F} B x_j\right),
$$
which can also be viewed as a graded $\kk H$-module isomorphism
$$
A \downarrow^G_H \,\, \cong \,\, 
B \,\, \oplus \,\, {\mathcal X} \otimes B (-1),
$$
where ${\mathcal X}$ is the permutation 
or signed permutation representation of $H$ on $E \setminus F$ as before.
\end{prop}
\begin{proof}
First note a consequence of Theorem~\ref{thm:BEZ}: if one orders/indexes $E=\{1,2,\ldots,n\}$ so that $i<j$ whenever $i \in F$ and $j \in E \setminus F$, then every pair $\{j,k\} \subseteq E \setminus F$ with $j \neq k$ is a broken-circuit, coming from the $3$-element circuit $\{i,j,k\}$ with $\{i\}:=F \cap (\{j\} \vee \{k\})$.  This implies NBC subsets for $M$ contain at most one element $j$ of $E \setminus F$, so NBC monomials for $M$ are either of the form 
\begin{itemize}
    \item[(a)] $x_{i_1} \cdots x_{i_p}$ for NBC sets $\{i_1,\ldots,i_p\} \subseteq F$, or 
    \item[(b)] $x_{i_1} \cdots x_{i_p} x_j$  for NBC sets $\{i_1,\ldots,i_p\} \subseteq F$, and $j \in E \setminus F$.
\end{itemize}
Identifying $B$ with $\bigoplus_{F' \subseteq F} \OS(M)_{F'}$ or $\bigoplus_{F' \subseteq F} \VG(\OM)_{F'}$ expresses $A$ as a $\kk$-vector space sum 
$$
A = B + 
\left( \sum_{j \in E \setminus F} B x_j\right).
$$
However, these sums are
{\it direct}, via dimension-counting:  if $e:=|E \setminus F|$, then one has
$$
\dim_\kk A = \dim_\kk B (1+e)
$$
as the $t=1$ specialization of the identity $\Hilb(A,t)=\Hilb(B,t)(1+et)$ (cf. \eqref{suggestive-primal-Hilb-rewriting} above) which follows either from Stanley \cite[Thm.~2]{Stanley-modular} or from Orlik and Terao \cite[Lem.~3.80]{OrlikTerao}.

Note that this dimension count also implies that the NBC monomials in (a),(b) above form $\kk$-bases for $B$ and $\bigoplus_{j \in E\setminus F} Bx_j$, respectively.  This lets one write a $\kk$-vector space isomorphism
$$
{\mathcal X} \otimes B 
\overset{f}{\longrightarrow}
\bigoplus_{j \in E\setminus F} 
Bx_j
$$
as follows: naming the $\kk$-basis elements $\{t_j: j \in E\setminus F\}$
for the permutation or signed permutation representation ${\mathcal X}$ of $H$, let the isomorphism $f$ map $t_j \otimes x_{i_1} \cdots x_{i_p} \longmapsto x_{i_1} \cdots x_{i_p} x_j$.  Since this means that
$f(t_j \otimes b) = b x_j$ for $b \in B$, the $H$-equivariance follows from this calculation: by definition, $g \in H$ has $g(t_j)=\pm t_k$ for $j,k \in E \setminus F$
if and only if $g(x_j)=\pm x_k$, with the
same $\pm$ signs for both.
\end{proof}

\subsection{Proof of Theorem~\ref{thm: supersolvable-branching-exact-sequences}(ii)} 

The surjective map $A^! \rightarrow B^!$ within the exact sequence of Theorem~\ref{thm: supersolvable-branching-exact-sequences}(ii)
is simple to define.  As before, let
$M, \OM$ be supersolvable matroids or oriented matroids on ground set $E=\{1,2,\ldots,n\}$ and let $\underline{E}$ be as in Proposition~\ref{prop: bjorner-ziegler-characterization}, with
$F=F_{r-1}=E_1 \sqcup \cdots \sqcup E_{r-1}$
and $E_r=E\setminus F$.
Let $A^!=\OS(M)^!$ or $\VG(\OM)^!$,
and $B^!=\OS(M|_F)^!$ or $\VG(\OM|_F)^!$.

\begin{prop}
\label{prop: zeroing-variables-surjection}
The surjective $\kk$-algebra map 
$$
\begin{array}{rcll}
\kk\langle y_1,\ldots,y_n\rangle
&\longrightarrow&
\kk \langle y_i \rangle_{i \not\in E_r}&\\
y_i&\longmapsto& y_i&\text{ if }i \not\in E_r\\
y_j&\longmapsto&0 &\text{ if }j \in E_r,
\end{array}
$$
induces a surjective $\kk$-algebra map $A ^!\twoheadrightarrow B^!$.
\end{prop}
\begin{proof}
Check that in the quadratic Gr\"obner basis presentation
of Theoerem~\ref{thm: shriek-presentations} for $A^!$, if a quadratic term is divisible by a variable $y_j$ with $j\in E_r$, then every term is divisible by such a variable.
\end{proof}

It only remains to identify the kernel of the surjection in Proposition~\ref{prop: zeroing-variables-surjection}.  Recall that $H$ is the setwise $G$-stabilizer subgroup for the modular coatom $F$,
and $\mathcal X$ is the permutation or signed permutation representation of $H$ as it acts on $E_r$, with the $\kk$-basis of $\mathcal X$ denoted $\{t_j\}_{j \in E_r}$.

\begin{prop}
\label{prop: two-sided-ideal-image}
The $\kk$-linear map
$$
\begin{array}{rcl}
{\mathcal X} \otimes A^! & \longrightarrow & A^!\\
t_j \otimes y_{i_1} \cdots y_{i_p} & \longmapsto & y_{i_1} \cdots y_{i_p} y_j
\end{array}
$$
is injective, with image equal to the kernel of the surjection $A^! \twoheadrightarrow B^!$ in Proposition~\ref{prop: zeroing-variables-surjection}.
\end{prop}
\begin{proof}
Since the surjection $A^! \twoheadrightarrow B^!$ is induced by sending the variables $\{y_j\}_{j \in E_r}$ to zero, its kernel is the two-sided ideal $I=(y_j: j \in E_r) \subset A^!$ that they generate.  As in the proof of Theorem~\ref{thm: shrieks-have-right-NZD}, the presentation for $A^!$
described in Theorem~\ref{thm: shriek-presentations}  and its standard monomial 
$\kk$-basis identify this ideal $I$ as $A^!_{(\geq 1)}$, the span of standard monomials $m_1  m_2 \cdots m_{r-1} \cdot m_r$, with $m_p$ in the variable set $\{y_j\}_{j \in E_p}$, that have $\deg(m_r)\geq 1$.  Classifying such standard monomials according to their rightmost variable $y_j$ shows that $I=A^!_{(\geq 1)}$ is the image of the map in the current proposition.  The standard monomial basis for $A^!$ also shows that this map is injective.  
\end{proof}

Noting that the maps in Propositions~\ref{prop: zeroing-variables-surjection} and \ref{prop: two-sided-ideal-image} are both $H$-equivariant proves Theorem~\ref{thm: supersolvable-branching-exact-sequences}(ii).

\section{Homotopy and holonomy Lie algebras}
\label{sec: deviations}

In Section \ref{sec: koszul-algebras}, we defined a standard graded $\kk$-algebra to be Koszul if it had a (left-)free resolution of $\kk=A/A_+$ which is linear. It turns out (see \cite[Section 2.1]{PolishchukPositselski}) that this definition is equivalent to any of the following conditions:
\begin{enumerate}[(a)]
    \item $\Ext^i_A(\kk,\kk)_j = 0$ for $i\neq j$,
    \item $A$ is quadratic and $A^! \cong \Ext^\bullet_A(\kk,\kk)$,
    \item $A$ is quadratic and $(A_i^!)^\ast \cong \Tor_i^A(\kk,\kk)$,
    \item $A$ is generated by $A_1$ and the algebra $\Ext_A^\bullet(\kk,\kk)$ is generated by $\Ext^1_A(\kk,\kk)$.
\end{enumerate}

The next proposition explains why 
quadratic algebras $A$ that are either commutative or anti-commutative have quadratic duals $A^!$ which inherit a Hopf algebra structure from the tensor algebra, making $A^!$ the universal enveloping algebra of a graded Lie (super)-algebra.  It can be considered an elaboration on \cite[\S I.2 Examples 4,5]{PolishchukPositselski}.

\begin{prop}
\label{prop:P-and-P-implicit-statement}
When a quadratic algebra $A$ is anti-commutative (resp. commutative), its quadratic dual $A^!$ is not just a $\kk$-algebra, but actually a co-commutative Hopf algebra (resp. co-commutative signed Hopf algebra, in the sense of Cartier-Patras \cite[\S 3.9]{CartierPatras}).
Hence by the Cartier-Milnor-Moore Theorem, $A^!$ is the universal enveloping algebra $\cU(\cL)$ of the graded Lie algebra (resp. Lie superalgebra)  $\cL \subset A^!$ which is its $\kk$-subspace of primitive elements.
\end{prop}
\begin{proof}
Since $A^!:=\kk\langle \yy \rangle/J$ for a
two-sided (algebra) ideal $J$, it suffices to check
that $J$ is also a {\it co-ideal} for the co-product $\Delta$ on the Hopf algebra $H:=\kk\langle \yy \rangle=T(V)$, that is, $\Delta(J) \subseteq H\otimes J + J \otimes H$.  We give the argument for the case where $A$ is commutative; 
 the anti-commutative case is similar.  

 Since $J=H J_2 H$ is generated as a two-sided ideal by $J_2$, and since $\Delta: H \rightarrow H \otimes H$ is an algebra morphism, it suffices to check that $\Delta(J_2) \subseteq H\otimes J + J \otimes H$.
 Note that since the quadratic algebra $A=\kk \langle \xx \rangle/I$ is commutative, it must
 be that $I=(I_2)$ has $I_2$ containing the $\kk$-span of all commutators $[x_i,x_j]_+$.  Consequently, $J_2=I_2^\perp$ lies
 in the perp space of the span of all such commutators, which is the $\kk$-span of all anti-commutators $[y_i,y_j]_-$, allowing $i=j$ here.  
 
 {\bf Claim:} every anti-commutator $[y_i,y_j]_-$ is {\it primitive}, 
 meaning $\Delta [y_i,y_j]_-=1\otimes [y_i,y_j]_- + [y_i,y_j]_- \otimes 1$. 
 
 Assuming the claim,  every $j \in J_2$ is also primitive, so $\Delta j = 1\otimes j + j \otimes 1 \in H \otimes J + J \otimes H$, as desired.  Checking the claim is a standard calculation: when $a=y_i,$ and $b=y_j$ are both primitive, and of odd degree, then their anti-commutator is also primitive:
\begin{align*}
\Delta [a, b]_-
=\Delta(a b + b a)
& =(1 \otimes a + a \otimes 1)
(1 \otimes b + b \otimes 1)+
(1 \otimes b + b \otimes 1)(1 \otimes a + a \otimes 1)\\
&=1\otimes ab - b\otimes a + a\otimes b + ab \otimes 1
+ 1\otimes ba - a\otimes b + b\otimes a+ ba \otimes 1\\
&=1\otimes ab + ab \otimes 1
+ 1\otimes ba + ba \otimes 1\\
&=1 \otimes (ab+ba) + (ab+ba) \otimes 1
=1 \otimes  [a, b]_- +  [a, b]_- \otimes 1. \qedhere
\end{align*}
\end{proof}

\begin{remark}
If a quadratic algebra $A$ is neither commutative nor anti-commutative, then $A^!$ might not inherit a Hopf algebra structure from the tensor algebra. Consider the quadratic algebra
    $$
    A^! = \k[x,y]/(x^2+y^2) = \kk\langle x,y\rangle/(x^2+y^2, xy-yx).
    $$
    This is the quadratic dual of $A = \kk\langle x,y\rangle / (xy+yx,y^2-x^2)$, which is neither commutative nor anti-commutative.
    Notice that if the characteristic of $\kk$ is not equal to $2$, the ideal $J = (x^2+y^2, xy-yx) $ is \textit{not} a co-ideal for the coproduct $\Delta$ on the tensor algebra $H = \kk\langle x,y\rangle$:
    $$
    \Delta(x^2+y^2) = (x^2+y^2)\otimes 1 + 2 (x\otimes x +  y\otimes y) + 1\otimes (x^2+y^2),
    $$
    which one can check does not lie in $H\otimes J + J\otimes H$.
\end{remark}


When $A=\bigoplus_{d=0}^\infty A_d$ is an associative standard graded $\kk$-algebra, so generated by $A_1$, and
is either commutative or anti-commutative, the
Yoneda algebra $\Ext^\bullet_A(\kk,\kk)$ 
has a natural coproduct giving it the structure of a graded Hopf algebra.  Therefore, $\Ext^\bullet_A(\kk,\kk)$ is also the universal enveloping algebra of a graded Lie (super-)algebra. See Avramov \cite[\S 10.1]{Avramov} for a discussion when $A$ is commutative, and Denham and Suciu \cite[\S 1]{DenhamSuciu} for the case where $A$ is anti-commutative.

\begin{definition}\label{def: homotopy-lie-algebra} 
In the above setting, the \textit{homotopy Lie algebra} $\pi_A$ is the graded Lie algebra or Lie superalgebra of primitive elements in the Yoneda algebra  $\Ext^\bullet_A(\kk,\kk)$ of $A$, that is, 
$$
\cU(\pi_A) \cong \Ext^\bullet_A(\kk,\kk).
$$
\end{definition}

\subsection{The holonomy Lie algebra} 
\label{sec: holonomy-lie-algebra}
Let $A=\bigoplus_{d=0}^\infty A_d$ be an associative graded $\kk$-algebra, with $\k$-basis $x_1,\dots, x_n$ for $V=A_1$; for the moment we do not assume that $A$ is generated by $A_1$. Then the \textit{decomposable} elements of $A_2$ are defined to be those in the image of the multiplication map
$$
\phi: A_1\otimes A_1 \to A_2.
$$
Letting $V=A_1^*$ have dual basis $y_1,\ldots,y_n$,
if one considers the dual of this multiplication map
$$
\phi^\ast: A_2^\ast \to (A_1\otimes A_1)^\ast \cong A_1^\ast \otimes A_1^\ast,
$$
then one has these identifications:
\begin{equation}
\label{eq: image-dual-multiplication-map}
\begin{aligned}
\im(\phi^\ast) 
&\cong \left(\frac{A_1\otimes A_1}{\ker\phi}\right)^\ast \cong \{f\in A_1^\ast \otimes A_1^\ast : f(\ker \phi) = 0\},\\
&\cong \left(\frac{T^2(V)}{I_2}\right)^\ast
\cong I_2^\perp =: J_2.
\end{aligned}
\end{equation}
Here we consider $J_2=I_2^\perp$ as a subspace of $T^2(V^*)=V^* \otimes V^*$,
with pairing $T^2(V^*) \times T^2(V) \rightarrow \kk$ just as in \eqref{bilinear-pairing-on-2-tensors}.
Now just as the proof of Proposition~\ref{prop:P-and-P-implicit-statement}, if one further assumes that $A$ is commutative (resp. anti-commutative), then $I_2$ contains the $\kk$-span of all commutators $[x_i,x_j]_+$
(resp. anti-commutators $[x_i,x_j]_-$).  
Consequently, $J_2=I_2^\perp$ lies
 in the perp space of the span of all such commutators or anti-commutators, which is 
 the $\kk$-span of all anti-commutators $[y_i,y_j]_-$,  allowing $i=j$ (resp. all commutators $[y_i,y_j]_+)$.
 In other words, $\im(\phi^*)=J_2$ is identified with
a subspace of $[A_1^*,A_1^*]_-$ or $[A_1^*,A_1^*]_+$ inside
$$
\Lie(V^*)=\Lie(A_1^*)=\Lie(y_1,\ldots,y_n) \subset T^*(A_1^*)
$$
where $\Lie(y_1,\ldots,y_n)$  denotes the {\it free Lie algebra} (resp. {\it free Lie superalgebra}) on $y_1,\ldots,y_n$ when $A$ is anti-commutative (resp. commutative).


\begin{definition}
\label{def: holonomy-lie-algebra} 
In the above context of an associative graded $\kk$-algebra $A$ which is either commutative or anti-commutative, define the \textit{holonomy Lie algebra} $\mathfrak h_A$ via the quotient
\begin{equation}
\label{eq: definition-holonomy-lie-alg}
\mathfrak h_A = 
\textrm{Lie}(A_1^\ast)/\langle \im(\phi^\ast)\rangle
=\textrm{Lie}(y_1,\ldots,y_n)/\langle J_2 \rangle.
\end{equation}
Here $\Lie(A_1^\ast)$ is the free Lie algebra (resp. free Lie superalgebra) on the $\kk$-basis $y_1, \dots, y_n$ for $V^*$ if $A$ is anti-commutative (resp. commutative), and $\langle J_2 \rangle=\langle I_2^\perp \rangle$ is the Lie ideal generated by $J_2=I_2^\perp$.
\end{definition}

The following result of L\"{o}fwall connects the holonomy and homotopy Lie algebras.

\begin{lemma} \cite{lofwall86}*{Theorem 1.1} 
\label{lem: Lofwall-result}
The universal enveloping algebra $\cU(\mathfrak h_A)$ of the holonomy Lie algebra $\mathfrak h_A$ equals the linear strand
of the Yoneda algebra $\Ext^\bullet_A(\kk,\kk)$.  That is,
\begin{align*}
\cU(\mathfrak h_A) &\cong
\bigoplus_{i\geq 0} \Ext_A^i(\kk,\kk)_i\\
&\subset \bigoplus_{i,j \geq 0} \Ext_A^i(\kk,\kk)_j = \Ext^\bullet_A(\kk,\kk).
\end{align*}
In particular, if $A$ is a Koszul algebra,
so $A^!=\Ext^\bullet_A(\kk,\kk)$ is equal to its own linear strand, one has 
$$
A^!=\cU(\cL) \text{ where }\cL={\mathfrak h}_A=\pi_A.
$$
\end{lemma}

We next give a simple presentation for the holonomy Lie algebra ${\mathfrak h}_A$ when $A= \OS(M)$ or $A=\VG(\OM)$.  In the case of $\OS(M)$, this is a well-known result of Kohno \cite{kohno1983holonomy}, but as far as the authors are aware, for $\VG(\OM)$ the presentation is new.
\begin{theorem} 
\label{thm:holonomy-lie-algebra-presentation}
The holonomy Lie algebra of $\OS(M)$ (resp. $\VG(\OM)$) for any simple (oriented) matroid $M$ (resp. $\OM$) is generated by the relations (\ref{unoriented-Kohno-relation}) (resp. (\ref{oriented-Kohno-relation}).
\end{theorem}
\begin{proof}
We give a proof for $\VG(\OM)$ analogous to 
L\"ofwall's proof \cite{lofwall-matroid} of Kohno's result for $\OS(M)$.
By \Cref{eq: image-dual-multiplication-map}, we can identify $\im(\phi^\ast)$ with $J_2:=I_2^\perp \subset \kk\langle \yy \rangle$, where $A=\kk\langle x_1,\ldots,x_n \rangle/I$.
There are three families of quadratic relations in the ideal $I$  presenting $\VG(\OM)$ to consider:
\begin{align}
\label{eq: variable-squared-relations}
   x_k^2  & \text{ for } k=1,2,\ldots,n,\\
\label{eq: commutativity-relations}
x_k x_\ell-x_\ell x_k
         & \text{ for }1 \leq k \leq \ell \leq n,\\
         \label{eq: relations-rank-2-flat}
     \partial^\pm(C):=  \sgn{C}{m} x_k x_\ell
        +  \sgn{C}{\ell} x_k x_m 
        +  \sgn{C}{k} x_\ell x_m  & 
       \text{ for circuits }
         C=\{k,\ell,m\}
         \text{ of size three.}
\end{align}
Recall from \eqref{flat-by-flat-shriek-computation} in 
\Cref{prop: flat-decomp} that the pairing $T^2(V^*) \otimes T^2(V) \rightarrow \kk$ makes $\kk\langle \xx\rangle_F$ and $\kk\langle \yy \rangle_{F'}$ pair to zero unless $F=F'$,
so that one can compute $J_2=I_2^\perp$ 
flat-by-flat, obtaining 
$$
[J_2 \cap \kk\langle \yy\rangle]_F=[I_2 \cap \kk\langle \xx\rangle]_F^\perp
$$
for all rank $2$ flats $F$ in $\cF$.
From this one sees that it suffices to prove 
the result assuming $\OM=\OM|_F \cong U_{2,n}$,  
a uniform rank $2$ matroid on $E=\{1,2,\ldots,n\}$ with one rank 2 flat $F=E$.

The quadratic part $I_2$ contains these $n + \binom{n}{2} + \binom{n-1}{2}=n^2-n+1$ elements among 
\eqref{eq: variable-squared-relations},
\eqref{eq: commutativity-relations},
\eqref{eq: relations-rank-2-flat}: 
\begin{itemize}
\item $n$ of the form $x_i^2$,
\item $\binom{n}{2}$ of the form $x_i x_j - x_j x_i$ for $i<j$, and
\item $\binom{n-1}{2}$ of the form $\partial^\pm(C)$ for circuits $C=\{1,i,j\}$ with $1 < i < j \leq n$.
\end{itemize}
One can also easily check that they are $\kk$-linearly independent inside $T^2(V)$. 
Consequently one has 
$$
\dim_\kk J_2=\dim_\kk I_2^\perp= \dim_\kk T^2(V^*) - \dim_\kk I_2
\leq n^2- (n^2-n+1)=n-1.
$$
On the other hand, the proof of \Cref{thm: shriek-presentations} showed each of these elements
from \eqref{oriented-Kohno-relation} lies in $I_2^\perp$:
$$
r^\pm(j,E):=\sum_{\substack{1 \leq k \leq n:\\k \neq j}}
\chi_{\OM|_F}(j,k) \cdot [y_j,y_k]_-
=\sum_{\substack{1 \leq k \leq n:\\k \neq j}} 
\chi_{\OM|_F}(j,k) \cdot (y_jy_k + y_k y_j),
$$
But the subset $\{ r^\pm(j,E) \}_{1\leq j\leq n-1}$ gives $n-1$ such elements 
which are linearly independent in $T^2(V^*)$, and therefore they span $J_2=I_2^\perp$.
\end{proof}

\subsection{PBW decomposition} When the Koszul algebra $A$ is commutative
or anti-commutative, we can use variants of the Poincar\'e-Birkhoff-Witt (PBW) Theorem to relate the ($G$-equivariant) Hilbert series for $A^! = \cU(\cL)$ and that of the graded Lie algebra
$\cL=\bigoplus_{d=0}^\infty \cL_d$.

\begin{remark} We state the results in this section assuming that the characteristic of $\kk$ is zero; however, these results  can be extended to arbitrary characteristic by replacing the symmetric algebra $\Sym(V)$ or symmetric powers $\Sym^k(V)$ with the \textit{divided power algebra} $D(V)$ or a divided power $D^k(V)$ in every place it appears.
\end{remark}

\subsubsection{The anti-commutative case}
When $A$ is \textit{anti}-commutative, 
and $\kk$ has characteristic zero, 
the PBW Theorem gives a graded $\kk$-vector space isomorphism 
$A^! = \cU(\cL) \cong \Sym(\cL)$.
Therefore, we have the Hilbert series relation
\begin{equation}\label{eq: hilbert-series-A-shriek}
    \Hilb(A^!,t)
    =\Hilb(\cU(\cL),t)
    =\Hilb(\Sym(\cL),t)
    =\prod_{d=1}^\infty\frac{1}{(1-t^d)^{\varphi_d}},
\end{equation}
where $\varphi_d=\dim_\kk \cL_d$;  see \cite[\S 2.2 Example 2]{PolishchukPositselski}.

\begin{remark}
The \textit{lower central series} (LCS) of a finitely-generated group $G$ is a chain of normal subgroups $G = G_1 \geq G_2 \geq \dots $ defined recursively by $G_k = [G_{k-1},G]$. Kohno \cite{Kohno} used the topological interpretation of the Orlik-Solomon algebra that we will discuss in \Cref{sec: topology} to investigate the LCS of the homotopy group of the complement of a complex hyperplane arrangement. By studying the holonomy Lie algebra of the Orlik-Solomon algebra of the braid arrangement, Kohno proved that the ranks $\varphi_d$ of the successive quotients in the lower central series of the homotopy group of the complement of the braid arrangement satisfy \Cref{eq: hilbert-series-A-shriek}. Falk and Randell \cite{falk-randell-LCS} later showed that this formula also holds for supersolvable arrangements, and Shelton and Yuzvinsky \cite{shelton1997koszul} proved that an LCS formula of the form in \Cref{eq: hilbert-series-A-shriek}  holds if and only if the Orlik-Solomon algebra of the arrangement is Koszul. Peeva \cite{Peeva} gave another proof that the LCS formula of this form holds for supersolvable arrangements using the fact that they have a quadratic Gr\"{o}bner basis.
\end{remark}

Any group $G$ of graded $\kk$-algebra symmetries of $A$, and therefore of $A^!$, will also act
as graded Lie algebra symmetries of $\cL$.  
The PBW Theorem then gives
these equalities in $R_\kk(G)[[t]]$:
\begin{equation}
\label{skew-commutative-primitive-formula}
\begin{aligned}
\Hilb_\eq(A^!,t)=\Hilb_\eq(\cU(\cL),t)
&=\Hilb_\eq(\Sym(\cL),t)\\
 &= \Hilb_\eq\left( \Sym\left( \bigoplus_{d=0}^\infty \cL_d \right),t \right)\\
        &=\sum_{\lambda=(1^{m_1} 2^{m_2} \cdots)} t^{|\lambda|} 
   \prod_{j \geq 1} \left[ \Sym^{m_j}\cL_j \right]. 
\end{aligned}
\end{equation}
In \Cref{sec: rep-stability}, we will use this description to investigate representation stability for $\cL$ in the setting where $A$ is anti-commutative.

\subsubsection{The commutative case}
Similarly, when $A$ is \textit{commutative},
Polishchuk and Positselski discuss in \cite[\S 1.2, Example 4]{PolishchukPositselski} how $A^! = \cU(\cL)$ for the (graded) Lie superalgebra $\cL=\bigoplus_{d=}^\infty \cL_d$ over $\kk$,
in which the parity is induced by the grading, that is,
\begin{align}
\cL_{\mathrm{even}}&=\bigoplus_{d \equiv 0 \bmod{2}} \cL_d,\\
\cL_{\mathrm{odd}}&=\bigoplus_{d \equiv 1 \bmod{2}} \cL_d.
\end{align}

The graded version of the PBW Theorem (see Milnor and Moore \cite[Thm. 5.15] {MilnorMoore}, Ross \cite[Thm. 2.1]{Ross},  Scheunert \cite[\S 2.3 Thm. 1]{Scheunert}),
asserts that when $\kk$ is a field of characteristic zero, one 
has a graded $\kk$-vector space isomorphism
$$
A^! = \cU(\cL) \cong \Sym_\pm (\cL) 
:= \Sym( \cL_{\textrm{even}} ) \otimes \wedge( \cL_{\textrm{odd}} )
$$ 
and hence a Hilbert series relation
$$
\Hilb(A^!,t)
=\Hilb(\cU(\cL),t)
=\Hilb(\Sym_\pm(\cL),t)
=\frac{ \prod_{d \text{ odd}} (1+t^d)^{\varphi_d} }
{ \prod_{d \text{ even}, d\geq 2} (1-t^d)^{\varphi_d} }
$$
where $\varphi_d=\dim_\kk \cL_d$.  

\begin{remark}
For \textit{any} formal power series $P(t) = 1+ \sum_{j\geq 1} b_j t^j$ with $b_j\in \ZZ$, there exist uniquely defined $\varphi_d$ such that
$$
P(t) = \frac{ \prod_{d \text{ odd}} (1+t^d)^{\varphi_d} }
{ \prod_{d \text{ even}, d\geq 2} (1-t^d)^{\varphi_d} }.
$$
If $P(t)=\sum_{d=0}^\infty \dim_\kk \Tor_j^R(\kk,\kk)$ is the {\it Poincar\'e series} of a Noetherian commutative ring $R$ in either
\begin{itemize}
    \item the {\it local setting}, where $(R,\mathfrak m)$ is a local ring with residue field $\kk=R/\mm$, or 
    \item the {\it graded setting}, where $R=\bigoplus_{d=0}^\infty$
    is an $\NN$-graded commutative $\kk$-algebra with $R_0=\kk$,
\end{itemize}
the exponent $\varphi_d$ is called the $d^{th}$ \textit{deviation} of the ring $R$.  This is because the nonvanishing of the $\varphi_d$ measures whether $R$ ``deviates'' from being a regular ring or a complete intersection in precise senses: 
\begin{itemize}
    \item $R$ is \textit{regular} if and only if $\varphi_2 = \varphi_3 = \dots = 0$; see \cite[7.3.2]{Avramov}
    \item $R$ is a \textit{complete intersection} if and only if $\varphi_3 = \varphi_4 = \dots = 0$; see \cite[7.3.3]{Avramov}.
\end{itemize}
Moreover, in the local setting one can always resolve $\kk$ over $R$ via an \textit{acyclic closure}; this was first proven in \cite{gulliksenLevin69}. See \cite[\S 6.3, \S 7, \S 10.2]{Avramov} for an in-depth discussion in the local setting; analogous results hold for commutative Noetherian graded $\kk$-algebras.
Informally, an acyclic closure is built by recursively adjoining formal variables to represent boundaries of any cycles that appear while computing an $R$-free resolution of $\kk$. The number of formal variables that one must adjoin in homological degree $d$ is exactly $\varphi_d$, which predicts the dimension of the $d^{th}$ graded component of the indecomposables within $\Tor^R(\kk,\kk)$.
Since the graded dual of $\Tor^R(\kk,\kk)$ is exactly $\Ext_R(\kk,\kk)$, the space of indecomposables of $\Tor^R(\kk, \kk)$ is the graded dual to the space of primitives in $\Ext_R(\kk,\kk)$, so that $\varphi_d = \dim_\kk \cL_d$.
\end{remark}

Again, in the presence of a group $G$ of  graded $\kk$-algebra symmetries acting on $A, A^!$, one also has these 
equalities in $R_\kk(G)[[t]]$: 
\begin{equation}
\label{commutative-primitive-formula}
\begin{aligned}
\Hilb_\eq(A^!,t)&=\Hilb_\eq(\cU(\cL),t)\\
&=\Hilb_\eq(\Sym_\pm(\cL),t)\\
 &= \Hilb_\eq\left( \wedge\left( \cL_{\mathrm{odd}} \right) \otimes  \Sym\left( \cL_{\mathrm{even}} \right),t \right)\\
        &=\sum_{\lambda=(1^{m_1} 2^{m_2} \cdots)} t^{|\lambda|}  
   \prod_{j\text{ odd}} \left[ \wedge^{m_j}\cL_j \right] \cdot
   \prod_{j\text{ even},j \geq 2} \left[ \Sym^{m_j}\cL_j \right]. 
\end{aligned}
\end{equation}
We will use this description in Section~\ref{sec: rep-stability} to investigate representation stability for $\cL$ in the setting where $A$ is commutative.

\begin{example}
\label{Boolean-matroids-re-revisited}
Let us return to the Boolean matroid $U_{n,n}$
discussed in Example~\ref{ex: Boolean-matroid-1} and Section~\ref{sec: Boolean-matroids-revisited}, but now considered as an {\it oriented matroid} represented by the standard basis vectors $v_1,\ldots,v_n$ in $\RR^n$. Since
the $\{v_i\}$ are linearly independent, there are no circuits, and the graded Varchenko-Gel'fand ring $A=\VG(U_{n,n})$ and its
Koszul dual $A^!$ have these descriptions:
\begin{align*}
A&=\kk[x_1,\ldots,x_n]/(x_1^2,\ldots,x_n^2)\\
&=\kk\langle x_1,\ldots,x_n\rangle/(x_k^2, x_i x_j-x_j x_i)_{1 \leq k \leq n, 1 \leq i < j \leq n}\\
&\\
A^!&=\kk\langle y_1,\ldots,y_n\rangle/
(y_i y_j+y_j y_i)_{1 \leq i < j \leq n}
\end{align*}
The oriented matroid automorphisms is $\Aut(\OM)$ are the full hyperoctahedral group $\symm_n^\pm$, in which a signed permutation $w$ with  $w(v_i)=\pm v_j$ acts on the variables via $w(x_i)=\pm x_j, w(y_i)=\pm y_j$.

We next analyze the graded $\kk \symm_n^\pm$-modules $A, A^!$ when $\kk$ has characteristic zero. To do this, first recall (e.g., from Geissinger and Kinch \cite{Geissinger}, Macdonald \cite[Chap.~1.~App. B]{Macdonald}) that
irreducible $\kk \symm_n^\pm$-modules
$\specht^{(\lambda^+,\lambda^-)}$
are indexed by ordered pairs of partitions 
$(\lambda^+,\lambda^-)$ where 
$|\lambda^+|=n_+,|\lambda|=n_-$ with $n_+ + n_-=n$.
One can construct $\specht^{(\lambda^+,\lambda^-)}$ using the irreducible $\kk \symm_n$-modules $\{ \specht^\mu\}$ as building blocks as follows. Introduce the operation of {\it inflation} $U \longmapsto U\Uparrow$ of a $\kk \symm_n$-module
$U$ to a $\kk\symm_n^\pm$-module by precomposing with the group surjection $\pi: \symm_n^{\pm} \longrightarrow \symm_n$ that ignores the $\pm$ signs in a signed permutation.
Also introduce the one-dimensional character
$\chi_\pm: \symm_n^{\pm} \rightarrow \{\pm 1\}$
sending a signed permutation $w$ to the product of its $\pm 1$ signs, that is, $\chi_\pm(w):=\det(w)/\det(\pi(w))$.
Then starting with irreducible $\kk \symm_n$-modules
$\specht^\lambda$, one builds $\specht^{(\lambda^+,\lambda^-)}$ as follows:
$$
\specht^{(\lambda^+,\lambda^-)}
:= \left(
\specht^{\lambda^+}\Uparrow
\quad \otimes \quad
\left(
\chi_\pm  \otimes  ( \specht^{\lambda^-} \Uparrow )\right)
\right)\big\uparrow_{\symm_{n_+}^\pm \times \symm_{n_-}^\pm}^{\symm_n^\pm}.
$$

For example, this identifies the graded component $A_i$ of $A=\kk[x_1,\ldots,x_n]/(x_1^2,\ldots,x_n^2)$
as the irreducible $\kk \symm_n^\pm$-module
$\specht^{((n-i),(i))}$.  This is because it is a direct sum of the $\binom{n}{i}$ lines which are the $\symm_n^\pm$-images
of the line $L:=\kk \cdot x_1 x_2 \cdots x_i$.
This line $L$ is 
stabilized setwise by the subgroup $\symm_{n-i}^\pm \times \symm_i^\pm$, with the $\symm_{n-i}^\pm$ factor acting trivially, and the
$\symm_i^\pm$ factor acting via $\chi_\pm$.
Hence one has
$$
\Hilb_\eq(A,t)=\sum_{i=0}^n t^i \cdot [\specht^{((n-i),(i))}].
$$

We next analyze $A^!$ as a $\kk \symm_n^\pm$-module.
Since $(x_1^2,\ldots,x_n^2)$ is a regular sequence in $\kk[\xx]$, the quotient $A$ is a complete intersection, and $\cL_1=V^*=\mathrm{span}_\kk\{y_1,\ldots,y_n\}$ and $\cL_2=\mathrm{span}_\kk\{ y_1^2,\ldots,y_n^2\}$.  This gives a $\kk\symm_n^\pm$-module isomorphism
\begin{align}
A^! &\cong \wedge_\kk \cL_1 \otimes_\kk \Sym \cL_2 \notag\\
&\label{boolean-VG-shriek-tensor-factors}
= \wedge_\kk (y_1,\ldots,y_n)  \otimes_\kk \kk[y_1^2,\ldots,y_n^2].
\end{align}

One can analyze each tensor factor in \eqref{boolean-VG-shriek-tensor-factors} separately.  An analysis similar to the
one for $A_i$ gives an $\kk \symm_n^\pm$-module isomorphism
$$
\wedge^i_\kk (y_1,\ldots,y_n)
\cong 
\specht^{((n-i),(1^i))}.
$$
In the other tensor factor of \eqref{boolean-VG-shriek-tensor-factors}, the action of $\symm_n^\pm$ on 
$\kk[y_1^2,\ldots,y_n^2]$ is inflated through
the surjection $\pi:\symm_n^\pm \longrightarrow \symm_n$, letting one compute its $\symm_n^\pm$-equivariant Hilbert series from the one for $\symm_n$ on $\kk[y_1,\ldots,y_n]$ given
in \eqref{Lusztig-Stanley-formula}, and doubling the grading.  The upshot is this equivariant Hilbert series:
\begin{align}
\notag
\Hilb_\eq(A^!,t)&= 
\Hilb_\eq(\kk[y_1^2,\ldots,y_n^2],t) \cdot
\Hilb_\eq(\wedge \{y_1,\ldots,y_n\},t) \\
\notag
&=
\frac{1}{(1-t^2)(1-t^4)\cdots(1-t^{2n})}\left(
\sum_{Q} t^{2 \mathrm{maj}(Q)} \cdot [\specht^{(\lambda(Q),\varnothing)}]
\right)
\left(
\sum_{i=0}^n t^i  \cdot [\specht^{((n-i),(1^i))}]
\right)\\
\label{type-B-inflated-Lusztig-Stanley-formula}
&=
\frac{
\sum_{i=0}^n \sum_{Q} t^{2 \mathrm{maj}(Q)+i} \cdot [\specht^{(\lambda(Q),\varnothing)}]
 \cdot [\specht^{((n-i),(1^i))}]
}{(1-t^2)(1-t^4)\cdots(1-t^{2n})}
\end{align}
where in the sums above, $Q$ ranges over all standard Young tableaux with $n$ cells.
\end{example}

\section{Topological interpretations of $\OS(M), \VG(\OM)$ and Koszul duality}
\label{sec: topology}

Orlik-Solomon algebras $\OS(M)$ have their origins
in the following result.
\begin{theorem}
\label{OS-theorem}
    \cite{OrlikSolomon}
For an arrangement $\cA=\{H_1,\ldots,H_n\}$ of
linear hyperplanes in $\CC^r$ with normal vectors
$v_1,\ldots,v_n$ representing a matroid $M$, the cohomology ring of their complement 
$
X:=\CC^r \setminus \bigcup_{H \in \cA} H
$
has presentation (using any coefficient ring $\kk$) as
$$
H^*(X,\kk) \cong \OS(M).
$$
\end{theorem}

\noindent
An analogue for $\VG(\OM)$ was given
by de Longueville and Schultz \cite{DelonguevilleSchultz} and later 
Moseley \cite{Moseley}.

\begin{theorem}
\label{Moseley-theorem}
  \cite[Cor. 5.6]{DelonguevilleSchultz},  
Moseley \cite{Moseley}
For an arrangement $\cA=\{H_1,\ldots,H_n\}$ of
linear hyperplanes in $\RR^r$ with normal vectors
$v_1,\ldots,v_n$  representing an oriented matroid $\OM$, the cohomology ring of their ``$\RR^3$-thickened" complement 
$
X_{\RR^3}:=\left( \RR^r \otimes_\RR \RR^3\right)
\setminus \bigcup_{H \in \cA} \left( H\otimes_\RR \RR^3 \right)
$
has presentation (using any coefficient ring $\kk$) as
$$
H^*(X_{\RR^3},\kk) \cong \VG(\OM),
$$
with the cohomology concentrated in even
degrees, so that the isomorphism halves the grading.
\end{theorem}

\begin{remark}
\label{rem: thickening-in-general}
The result of de Longueville and Schultz
  \cite[Cor. 5.6]{DelonguevilleSchultz}
proves more generally that, for any $d \geq 2$, under the same assumptions on $\cA \subset \RR^r$,
the ``$\RR^d$-thickened" complement 
$
X_{\RR^d}:=\left( \RR^r \otimes_\RR \RR^d\right)
\setminus \bigcup_{H \in \cA} \left( H\otimes_\RR \RR^d \right)
$
has presentation (using any coefficient ring $\kk$) as
$$
H^*(X_{\RR^d},\kk) \cong 
\begin{cases}
\OS(M) & \text{ for }d=2,4,6,\ldots,\text{ even},\\
\VG(\OM) & \text{ for }d=3,5,7,\ldots,\text{ odd},
\end{cases}
$$
with cohomology concentrated in degrees divisible by $d_1$, so the isomorphism 
divides the grading by $d-1$.  Here $M, \OM$ are
the matroid, oriented matroid associated to the
normal vectors $v_1,\ldots,v_n$.
\end{remark}

\begin{remark}
The type $A_n$ and $B_n$ reflection arrangements are both supersolvable and realizable over $\RR$ (and therefore $\CC$) and therefore we can apply this topological interpretation of the Orlik-Solomon and Varchenko-Gel'fand rings of these families of arrangements. In fact, for the type $A$ reflection arrangements, we can also view the Orlik-Solomon and Varchenko-Gel'fand rings as cohomology rings of configuration spaces of points in $\RR^d$; this perspective will be discussed further in \Cref{subsec: configuration-spaces}.

In general, not every arrangement realizable over $\CC$ is realizable over $\RR$, and there exist matroids (including supersolvable ones) that are represented only
in positive characteristics, and some not representable over any field.  See for example, some of the
matroids discussed in Sections~\ref{sec: projective-geometries} and \ref{sec: dowling}.
One may visualize some of the implications as follows:
$$     \begin{array}{ccc}
     \text{matroid }M & \Leftarrow & \text{oriented matroid }\OM\\
     \Uparrow &  &\Uparrow\\
     \cA\text{ realized over }\CC& \Leftarrow &\cA \text{ realized over }\RR
     \end{array}
     $$
\end{remark}

If the cohomology ring $H^\ast(X,\kk)$ of a simply connected topological space $X$ is a Koszul $\kk$-algebra (as in the case of the Orlik-Solomon and Varchenko-Gel'fand rings for supersolvable arrangements), then the Koszul dual $H^\ast(X,\kk)^!$ can be interpreted as the homology ring $H_\ast(\Omega X,\kk)$ of the based loop space $\Omega X$.

\begin{prop}[See \cites{Berglund, berglund2017free}]
\label{prop: shrieks-for-Koszul-spaces}
Let $X$ be a simply connected topological space such that its cohomology ring $A:=H^\ast(X,\kk)$ is a Koszul $\kk$-algebra over a field $\kk$. Then
$$
A^! = H^\ast(X,\kk)^! = H_\ast(\Omega X, \kk)
$$
where $\Omega X$ is the basepointed loop space of $X$.
\end{prop}

\begin{proof} The authors thank Craig Westerland
for communicating the following proof to them.
Observe that these spaces participate in the path-loop fibration 
$$
\Omega X \rightarrow PX \rightarrow X,
$$
where $PX \coloneqq \{f: I\to X : f(0) = \ast \text{ and } f \text{ continuous}\}$ is the space of based maps from an interval into $X$; note that $PX$ is contractible.
In general, for a Serre fibration $F\rightarrow E \rightarrow B$ having $B$ simply connected, the Eilenberg-Moore spectral sequence for cohomology is
$$
E^{*,*}_2=\Tor_{H^*(B)}(\kk , H^*(E))\Rightarrow  H^*(F)
$$
where $H^*(E)$ is a $H^*(B)$-module by the map in the fibration, and $\kk$ is our base field (or ring, if everything is suitably flat over $\kk$).  
In the case of the path-loop fibration, since $PX$ is contractible,      
$$
\Tor_{H^*(X)}(\kk,\kk) \Rightarrow H^*(\Omega X).
$$
If $\kk$ is a field, 
we can apply $\Hom_\kk(-,\kk)$ to get
$$
\Ext_{H^*(X)}(\kk,\kk) \Rightarrow H_*(\Omega X).
$$
If $H^*(X)$ is Koszul, then $\Ext_{H^*(X)}(\kk,\kk) \cong H^*(X)^!$.  Further, as $\Ext_{H^*(X)}(\kk,\kk)$ is concentrated in diagonal bidegrees, its differentials are zero, so the spectral sequence collapses at $E_2$, giving
$$
H_*(\Omega X,\kk) = H^*(X,\kk)^!. \qedhere
$$    
\end{proof}

\begin{remark}
Under the hypotheses of Proposition~\ref{prop: shrieks-for-Koszul-spaces}, the terminology from Definition~\ref{def: homotopy-lie-algebra} of {\it homotopy Lie algebra} for the $\kk$-subspace of primitives $\cL \subset A^!=\Ext_A^\bullet(\kk,\kk)$ is consistent
with the same terminology in {\it rational homotopy theory}, where a simply connected space $X$ has {\it homotopy Lie algebra} defined as the $\kk$-subspace of primitives $\cL \cong \pi_*(\Omega X)\otimes \kk$ inside the Hopf algebra $H_*(\Omega X,\kk)$; see, e.g., F\'elix, Halperin and Thomas \cite[\S 21(d), Thm. 21.5]{FelixHalperinThomas}.
\end{remark}
\section{Representation stability and Koszul algebras}
\label{sec: rep-stability}

We wish to show how sequences of Koszul algebras $\{A(n)\}_{n\ge 1}$ with $\symm_n$-actions
that satisfy representation stability in the sense of Church and Farb \cite{ChurchFarb}
lead to representation stability for their Koszul
duals $\{A(n)^!\}_{n\ge 1}$, and for the
primitive parts $\{\cL(n)\}_{n\ge 1}$ of the duals.
Useful references on representation stability
are Church and Farb \cite{ChurchFarb}, Church, Ellenberg and Farb \cite{CEF}
and Matherne, Miyata, Proudfoot and Ramos \cite{MatherneMiyataProudfootRamos}.

In this section, $\kk$ is a field of characteristic zero.
Recall this definition from \cite{ChurchFarb}.

\begin{definition}
For a partition $\mu$ of $k$, recall that  $\specht^\mu$ denotes the irreducible $\kk \symm_k$-module indexed by $\mu$.
Given a partition $\lambda=(\lambda_1,\lambda_2,\ldots,\lambda_\ell)$ and $n \geq |\lambda|+\lambda_1$, define a partition of $n$ by 
$$
\lambda[n]:=(n-|\lambda|,\lambda_1,\lambda_2,\ldots,\lambda_\ell).
$$
Say that a sequence $\{V_n\}_{n\ge 1}$ of
    $\kk \symm_n$-modules\footnote{Such a sequence of $\kk \symm_n$-modules is equivalent to what is called an {\it $FB$-module} in \cite{MatherneMiyataProudfootRamos}.} is {\it representation stable} if there is a list of
    (not necessarily distinct) partitions $\{ \lambda^{(i)}\}_{i=1}^t$ and an integer $N$ 
     such that for $n \geq N$, one has
    $$
     V_n = \bigoplus_{i=1}^t \specht^{\lambda^{(i)}[n]}.
    $$
Say that $\{V_n\}_{n\ge 1}$ is {\it representation stable past $N$} when the above equality holds for $n \geq N$.  
\end{definition}

The following easy observations are recorded in \cite[\S 3]{MatherneMiyataProudfootRamos}.

\begin{prop}
When $\{V_n\}, \{W_n\}$ are representation stable sequences, then so is $\{V_n\oplus W_n\}$.  
On the other hand,
if the virtual modules $[U_n]=[V_n]-[W_n]$
come from genuine $\kk \symm_n$-modules $\{U_n\}$, then $\{U_n\}$ is also a representation stable sequence.
\end{prop}

It is less obvious what happens for tensor products.  The following precise version of a result of Murnaghan was proven by Briand, Orellana, Rosas
 \cite[Thm. 1.2]{BriandOrellanaRosas}.

\begin{theorem} 
\label{BriandOrellanaRosasBound}
The sequence
$
\{ \specht^{\alpha[n]} \otimes \specht^{\beta[n]} \}
$
is representation stable past  $|\alpha|+|\beta|+\alpha_1+\beta_1$.
\end{theorem}

This consequence was noted
by Matherne, Miyata, Proudfoot and Ramos 
 \cite[Thm. 3.3]{MatherneMiyataProudfootRamos}.
 
\begin{lemma}
\label{tensoring-sums-stability-bounds-lemma}
If the $\{V_n\}, \{W_n\}$ are representation stable past $A,B$, respectively, then $\{V_n \otimes W_n\}$ is representation stable past $A+B$.
\end{lemma}

The following observation is occasionally useful for pinpointing the onset of representation stability.

\begin{lemma} \cite[Lemma 2.2]{HershReiner}
\label{lem: Hemmer-like}
Let $\{ V_n \}_{n\ge N}$ be
$\kk\symm_n$-modules defined via a finite direct sum
$$
V_n \cong \bigoplus_\mu \left( \specht^{\mu} \uparrow_{\symm_{|\mu|} \times \symm_{n-|\mu|}}^{\symm_n} \right)^{\oplus c_\mu}
$$
with $|\mu|\geq N$,
and integers $c_\mu \geq 1$.  
Then $\{ V_n \}$ is representation stable, 
stabilizing exactly at 
$$
n=\max_\mu \{ |\mu|+\mu_1 \}.
$$
\end{lemma}

We next use some of the foregoing observations to show how representation stability of families of Koszul algebras $\{A(n)\}$ passes
to their Koszul duals $\{A(n)^!\}$.

\begin{cor}
\label{cor: shrieks-have-rep-stability-cor}
Let $\{A(n)\}_{n\ge 1}$ 
be a sequence of Koszul algebras, with
Koszul duals $\{A(n)^!\}$.
\begin{enumerate}
    \item[(i)]
If for each fixed $i \geq 0$, the sequence $\{A(n)_i\}$ is representation stable,
then so is each $\{A(n)^!_i\}$.

\item[(ii)]
If furthermore there exists a constant $c$ (independent of $i$) such that each sequence $\{A(n)_i\}$ is representation stable past $ci$, then each 
$\{A(n)^!_i\}$ is also representation stable past $ci$.
\end{enumerate}
\end{cor}
\begin{proof}
We prove (ii); the proof for (i) is the same, ignoring the bounds involving $ci$.  
But (ii) is immediate using equation \eqref{unraveled-shriek-recurrence} that appeared in \Cref{cor:equivariant-Koszul-dual-Hilbert-series}, along with Lemma~\ref{tensoring-sums-stability-bounds-lemma}, since each factor $A_{\alpha_p}=A_{\alpha_p}(n)$ in the right-hand side of \eqref{unraveled-shriek-recurrence}, now in characteristic zero, has the sequence $\{A_{\alpha_p}(n)\}$ representation stable past $c{\alpha_p}$.
\end{proof}

\begin{example}
Rank two matroids $U_{2,n}$ were discussed in
Example~\ref{ex: rank-two-matroid-1} and Section~\ref{sec: rank-two-revisited}.  Their group of matroid automorphisms is $\Aut(U_{2,n})=\symm_n$, and \eqref{rank-two-OS-equivariantly} showed that as $\kk \symm_n$-modules, one has
\begin{align*}
[\OS(U_{2,n})_0] &\cong [\specht^{(n)}],\\
[\OS(U_{2,n})_1] &\cong [\specht^{(n)}]+[\specht^{(n-1,1)}],\\
[\OS(U_{2,n})_2] &\cong [\specht^{(n-1,1)}],
\end{align*}
which are representation stable past $n=2$.
Consequently, applying Corollary~\ref{cor: shrieks-have-rep-stability-cor}(ii) with $c=2$ implies that the Koszul duals
$\{\OS(U_{2,n})^!_i\}$, which are the $\symm_n$-permutation representations discussed in Proposition~\ref{prop: rank-two-OS-shrieks-are-perm-reps},
should be representation stable past $n=2i$.  In fact, one has the following.

\begin{prop}
\label{prop: rank-two-shriek-stabilization-onset }
For $i \geq 0$, representation stability 
of $\{\OS(U_{2,n})^!_i\}$
starts {\it exactly} at $n=2i$. 
\end{prop}
\begin{proof}
Recall that 
Proposition~\ref{prop: rank-two-OS-shrieks-are-perm-reps}  expresses $[\OS(U_{2,n})^!_i]$
as a nonnegative combination
of classes $\varphi_{(n-d,1^d)}$ for various $d$
in the range $2\leq d \leq i$, where $\varphi_\lambda$ is the class
of the coset representation $\kk[\symm_n/\symm_\lambda]$.
Furthermore, the coefficient 
on $\varphi_{(n-i,1^i)}$ is $1$.
Since one can write
$$
\varphi_{(n-d, 1^d)}
=\left[ \left( \kk \symm_d \right) \uparrow_{\symm_d \times \symm_{n-d}}^{\symm_n} \right] \quad \text{ where }\quad
\kk \symm_d \cong \bigoplus_{\mu: |\mu|=d} (\specht^{\mu})^{\oplus \dim \specht^\mu},
$$
Lemma~\ref{lem: Hemmer-like} shows $\{\varphi_{(n-d,1^d)}\}_n$
stabilizes exactly at $n=2d$, and 
$\{[\OS(U_{2,n})^!_i]\}$ exactly at $n=2i$.
\end{proof}

\end{example}

Recall from Section~\ref{sec: deviations} 
that when $A$ is Koszul and commutative or anti-commutative,
then $A^! = \cU(\cL)$ is the universal enveloping algebra for a graded Lie algebra or superalgebra $\cL$.
We next show how representation stability of families of Koszul algebras $\{A(n)\}$ passes to the Lie (super-)algebras $\{\cL(n)\}$.
The following lemma will be useful for this purpose.

\begin{lemma}
\label{schur-functor-rep-stability-lemma}
For any representation stable sequence
$\{V_n\}$ and any partition $\mu$ giving rise to a Schur functor $\schurfunctor^\mu(-)$, the sequence
$\{\schurfunctor^\mu(V_n)\}$ is also representation stable.  
In particular, for each fixed $m=0,1,2,\ldots$, the
sequences $\{\wedge^m(V_n)\}, \{\Sym^m(V_n)\}$
are representation stable.
\end{lemma}
\begin{proof}
Express the representation stable sequence 
$\{V_n\}$ for $n \gg 0$ as
$
V_n=\bigoplus_{i=1}^t \specht^{\lambda^{(i)}[n]},
$
and proceed by induction on $t$ to show
$\{\schurfunctor^\mu(V_n)\}$ is representation stable for all $\mu$. 

The case $t=1$ was proven by
Church, Ellenberg, Farb \cite[Prop. 3.4.5]{CEF}
who showed that for any partitions $\lambda, \mu$, the sequence $\{ \schurfunctor^\mu(\specht^{\lambda[n]}) \}$ is representation stable. In the inductive step,
write
$
V_n=U_n \oplus \specht^{\lambda^{(i)}[n]}
$
for $n \gg 0$,
where $U_n:= \bigoplus_{i=1}^{t-1} \specht^{\lambda^{(i)}[n]}$, so that induction applies to the representation stable sequence $\{U_n\}$.
Using the general isomorphism (see, e.g., \cite[Thm.~II.4.11]{AkinBuchsbaumWeyman})
$$
\schurfunctor(X \oplus Y) \cong \bigoplus_{\nu \subseteq \mu}
\schurfunctor^\nu(X) \otimes 
\schurfunctor^{\mu/\nu}(Y)
$$
one concludes that, for $n \gg 0$,
$$
\begin{aligned}
\schurfunctor^\mu(V_n)
=\schurfunctor^\mu( U_n \oplus \specht^{\lambda[n]} )
\cong \bigoplus_{\nu \subseteq \mu}
\schurfunctor^\nu(U_n) \otimes 
\schurfunctor^{\mu/\nu}(\specht^{\lambda[n]} )
\cong \bigoplus_{\nu,\mu,\rho}
\left( \schurfunctor^\nu(U_n) \otimes 
\schurfunctor^{\rho}(\specht^{\lambda[n]}) 
\right)^{\oplus c^{\mu/\nu}_\rho}
\end{aligned}
$$
for some nonnegative integer (Littlewood-Richardson) coefficients $c^{\mu/\nu}_\rho$.
By induction on $t$, the sequences $\{ \schurfunctor^\nu(U_n)\}$ are representation stable, and by the $t=1$ case, the same holds for
$\{\schurfunctor^{\rho}(\specht^{\lambda[n]})\}$.
Hence by Theorem~\ref{BriandOrellanaRosasBound},
each summand $\{\schurfunctor^\nu(U_n) \otimes \schurfunctor^{\rho}(\specht^{\lambda[n]})\}$ on the right side is a representation
stable sequence, and the same holds for the entire direct sum.
\end{proof}

We now apply this to the sequences of
Lie (super-)algebras $\{\cL(n)\}$.

\begin{cor}
\label{cor: primitives-inherit-rep-stability}
Let $\{A(n)\}$ be a family of Koszul algebras, all commutative (resp. anti-commutative), with $\{ \cL(n) \}$ defined by
$A(n)^!=\cU(\cL(n))$.  If for
each fixed $i=0,1,2,\ldots$, the sequence $\{A(n)_i\}$ is representation stable, then for each fixed $i=1,2,\ldots$,
the sequence $\{ \cL(n)_i \}$ is also 
representation stable.
\end{cor}
\begin{proof}
In either case where $\{A(n)\}$ are commutative or anti-commutative, use induction on $i$.  In the base case $i=1$, 
one has this string of equalities, justified below:
$$
[\cL(n)_1]\overset{(a)}{=}[A(n)^!_1]\overset{(b)}{=}[\left( A(n)_1\right) ^*]\overset{(c)}{=}[A(n)_1].
$$
Equality (a) comes from comparing coefficients of $t^1$ on either side of
\eqref{commutative-primitive-formula} or \eqref{skew-commutative-primitive-formula}, equality (b) comes from \eqref{equivariant-recurrence-for-shrieks}, and equality (c) comes from the fact that $\kk\symm_n$-modules are all self-contragredient.  Since the right sides $\{A(n)_1\}$ are representation stable, so are the left sides $\{\cL(n)_1\}$ .

In the induction step where $i \geq 2$, rewrite the equalities that come from comparing the coefficient of $t^i$ on either side of \eqref{commutative-primitive-formula} or \eqref{skew-commutative-primitive-formula}, isolating the summand $[\cL(n)_i]$ on the right
corresponding to $\lambda=(i)$.  For fixed $n$, this expresses $\cL(n)_i$ recursively in terms of
$A(n)^!_i$ and $\cL(n)_1,\cL(n)_2,\ldots, \cL(n)_{i-1}$:
$$
[\cL(n)_i]=[A(n)^!_i]- 
\begin{cases}
 \displaystyle \sum_{\substack{\lambda \vdash i:\\\lambda=(1^{m_1} 2^{m_2} \cdots i^{m_i})\\ \lambda \neq (i)}}
  \prod_{1 \leq j <i} [ \Sym^{m_j}(\cL(n)_j) ] 
   \quad\quad\quad\quad\text{ for }A(n)\text{ anti-commutative}.\\
   &\\
\displaystyle \sum_{\substack{ \lambda \vdash i:\\\lambda=(1^{m_1} 2^{m_2} \cdots i^{m_i})\\ \lambda \neq (i)}} 
\prod_{\substack{1 \leq j < i\\j \text{ odd }}} [\wedge^{m_j}(\cL(n)_j) ]
\prod_{\substack{2 \leq j < i\\j\text{ even }}} [\Sym^{m_j}(\cL(n)_j)]  \text{ for }A(n)\text{ commutative}.
\end{cases}
$$
Now use Corollary~\ref{cor: shrieks-have-rep-stability-cor} 
asserting that each sequence $\{A(n)^!_i\}$
is representation stable.  Induction on $i$ shows each sequence $\{\cL(n)_j\}$ for $j \leq i-1$ appearing on the right is representation stable.  
Lemma~\ref{schur-functor-rep-stability-lemma},
then implies the same for all sequences $\{\wedge^{m_j}(\cL(n)_j\}, \{\Sym^{m_j}(\cL(n)_j\}$
appearing on the right.  Then Theorem~\ref{BriandOrellanaRosasBound} gives
the same for their tensor products.  Thus
every summand on the right is a representation
stable sequence in $n$, and hence so is $\{\cL(n)_i\}$.
\end{proof}

\begin{remark}
The above proof shows the following statement for
a sequence of graded Lie (super-)algebras and $\kk\symm_n$-modules 
$\cL(n)$, with universal enveloping algebras $\cU(\cL(n))$:
one has for all $i \geq 1$ that $\{\cL(n)_i\}$
is representation stable if and only if 
one has for all $i \geq 0$ that $\{\cU(\cL(n))_i\}$
is representation stable.
\end{remark}

\begin{remark}
   Unlike Corollary~\ref{cor: shrieks-have-rep-stability-cor}, we have not seriously tried to bound the onset of stability for the sequences $\{\cL(n)_i\}$, in terms of a given bound for the onset of stability in $\{A(n)_i\}$. However, {\tt Sage} computations up to $i = 10$ suggest the following for uniform matroids $U_{2,n}$ of rank $2$.
\end{remark}

\begin{conj}
    Defining $\cL(n)_i$ by $\OS(U_{2,n})^! = \cU(\cL(n))$, the sequence $\{\cL(n)_i\}$ is representation stable past $n = 2i-1$ for fixed $i\geq 3$.
\end{conj}

\begin{remark} Although $[A(n)_i^!]$ is a permutation representation when $A(n) = \OS(U_{2,n})$ by \Cref{prop: rank-two-OS-shrieks-are-perm-reps}, the primitives $[\cL(n)_i]$ are generally \textbf{not} classes of permutation representations. This fails immediately for $n=3$ and $i=2$, where $[\cL(3)_2]$ is the sign representation.
%
Also for braid matroids, if $A(n)=\OS(\Br_n)$,
and $A(n)^!=\OS(\Br_n)^!=\cU(\cL(n))$, one can check 
$$
\cL(n)_2=(\mathrm{sgn}_{\symm_3}\otimes \one_{\symm_{n-3}})\big\uparrow^{\symm_n}_{\symm_3 \times \symm_{n-3}}
$$
which is again not a permutation module.  
\end{remark}

\section{The motivating example:  braid matroids and Stirling representations}
\label{sec: stirling-reps}

As mentioned prior to Example~\ref{graphic-braid-example}, our motivation came from the {\it braid matroids} $M=\braid_n$, which are 
also known as the graphic matroids for complete graphs $K_n$.  They are also known as the matroids represented by the vectors $\{v_{ij}\}_{1 \leq i<j\leq n}$ with $v_{ij}:=e_i-e_j$ in $\RR^n$ which are the (positive) roots in the type $A_{n-1}$ root system, whose normal hyperplanes $H_{ij}=\{ \xx \in \RR^n: x_i=x_j\}$ are the reflecting hyperplanes for the transposition $(i,j)$ in the symmetric group $\symm_n$ when it acts on $V=\RR^n$
by permuting coordinates.  Thus $M=\braid_n$ is orientable,
and abusing notation slightly, we will also denote by $\OM=\braid_n$ the oriented matroid on the ground set $E=\{\{i,j\}: 1 \leq i < j \leq n\}$ 
represented by these vectors $\{ v_{ij} \}$.

\subsection{Comparison with cohomology of configuration spaces}\label{subsec: configuration-spaces}

It turns out that the algebras $\OS(\braid_n), \VG(\braid_n)$ had been studied historically earlier
as the cohomology rings of certain 
{\it configuration spaces} of $n$ ordered (labeled) points
in a space $X$
$$
\Conf_n(X):=\{(x_1,\ldots,x_n) \in X^n: x_i \neq x_j \text{ for }1\leq i< j\leq n\}.
$$
The arrangement of hyperplanes $H_{ij}$ in $V=\RR^n$
described above allow one to identify
$$
\Conf_n(\RR^d) = V \otimes_\RR \RR^d \setminus \bigcup_{1 \leq i < j \leq n} H_{ij} \otimes_\RR \RR^d,
$$
that is, as the complement of subspace arrangements coming from the reflection hyperplane arrangements ``thickened" by tensoring with $\RR^d$ as in Theorems~\ref{OS-theorem}, ~\ref{Moseley-theorem}
and Remark~\ref{rem: thickening-in-general}.
In fact, the special cases of those results for the braid arrangements, along with quadratic presentations for the associated algebras, were known to 
Arnol'd \cite{Arnold} for $\OS(\braid_n)$ and
Cohen \cite{CohenLadaMay} for $\VG(\braid_n)$:
\begin{align*}
\OS(\braid_n) &\cong  
\wedge_\kk(x_{ij}) /  ( x_{ij} x_{ik} - x_{ij} x_{jk} + 
x_{ik} x_{jk}) \\
\VG(\braid_n) &\cong 
\kk[x_{ij}] / 
\,\, (x_{ij} x_{ik} - x_{ij} x_{jk} + x_{ik} x_{jk}, \,\, x_{ij}^2\,\,) 
\end{align*}
Here permutations $\sigma$ in $\symm_n$ act on the variables
by permuting subscripts, that is, $\sigma(x_{ij})=x_{\sigma(i),\sigma(j)}$, but with the convention that $x_{ji}=x_{ij}$ in $\OS(\braid_n)$, and $x_{ji}=-x_{ij}$ in $\VG(\braid_n)$.

Note that that these presentations are consistent with the general presentation coming from supersolvable matroids in Corollary~\ref{cor: primal-quadratic-GB}, using the modular complete flag $\underline{F}$ of flats chosen in
Example~\ref{graphic-braid-example}:  one checks that the corresponding decomposition $\underline{E}=(E_1,E_2,\ldots,E_{n-1})$ of $E=\{\{i,j\}\}_{1\leq i<j \leq n}$ has
\begin{equation}
\label{braid-matroid-hands}
\begin{aligned}
E_1&=\{\{1,2\}\},\\
E_2&=\{\{1,3\},\{2,3\}\},\\
E_3&=\{\{1,4\},\{2,4\},\{3,4\}\},\\
&\quad \vdots\\
E_{n-1}&=\{\{1,n\},\{2,n\},\ldots,\{n-2,n\},\{n-1,n\}\},\\
\end{aligned}
\end{equation}
and the subset of circuits 
$
\cC_{\BEZ}(\underline{E})=
\{ \{i,j\}, \{i,k\} ,\{j,k\} \}_{1 \leq i < j < k \leq n}.
$
Here the NBC monomial basis for either $\OS(\braid_n), \VG(\braid_n)$ are the products of $x_{ij}$ that 
choose at most one $\{i,j\}$ from each set $E_p$ with $p=1,2,\ldots,n-1$ above; Barcelo \cite[Thm. 2.1]{Barcelo} calls this {\it picking at most one finger $x_{ij}$ from each hand $E_p$}.  Since the exponents $e_p=|E_p|$ here are $(e_1,\ldots,e_r)=(1,2,\ldots,n-1)$,
one has these Hilbert series
\begin{align*}
\Hilb(\OS(\braid_n),t)
&=\Hilb(\VG(\braid_n),t)\\
&=(1+t)(1+2t)\cdots (1+(n-1)t)
=\sum_{i=0}^{n-1} t^i \, c(n,n-i)\\
\Hilb(\OS(\braid_n)^!,t)
&=\Hilb(\VG(\braid_n)^!,t)\\
&=\frac{1}{(1-t)(1-2t)\cdots (1-(n-1)t)}
=\sum_{i=0}^{\infty} t^i \, S((n-1)+i,n-1),
\end{align*}
where the coefficients $c(n,k), S(n,k)$ appearing here are the {\it (signless) Stirling numbers of the first kind} $c(n,k)$, counting permutations in $\symm_n$ with $k$ cycles, and the {\it Stirling numbers of the 2nd kind} $S(n,k)$, counting partitions of the set $\{1,2,\ldots,n\}$ into  $k$ blocks.
Comparing coefficients on powers of $t$, one has for either $A(n)=\OS(\braid_n)$ or $\VG(\braid_n)$ that 
\begin{align*}
\dim_\kk A(n)_i & =c(n,n-i),\\
c(n,k)&=\dim_\kk A(n)_{n-k},\\
& \\
\dim_\kk A(n)^!_i &=S((n-1)+i,n-1),\\
S(n,k)&=\dim_k A(k+1)^!_{n-k}.
\end{align*}

\begin{definition} ({\it Stirling representations})
For either $A(n)=\OS(\braid_n)$ or $A(n)=\VG(\braid_n)$, 
call $A(n)_i$ the {\it Stirling representations of the first kind}, and
call $A(n)^!_i$ the {\it Stirling representations of the second kind}.
When emphasizing their dimensions as representations, we will abbreviate them as 
\begin{align*}
    \cnkrep_\OS(n,k)&:=\OS(\braid_n)_{n-k},\\
    \cnkrep_\VG(n,k)&:=\VG(\braid_n)_{n-k}, \text{ both }\kk\symm_n\text{-modules},\\
    & \\
    \Snkrep_\OS(n,k)&:=\OS(\braid_{k+1})^!_{n-k},\\
    \Snkrep_\VG(n,k)&:=\VG(\braid_{k+1})^!_{n-k}, \text{ both }\kk\symm_{k+1}\text{-modules}.
\end{align*}
\end{definition}

\begin{remark}
\label{rem: Stirling-number-coincidence}
The coincidence between $\dim_\kk A(k+1)^!_{n-k}$ and $S(n,k)$, counting set partitions of $\{1,2,\ldots,n\}$ into $k$ blocks, is closely related to a well-known combinatorial encoding of set partitions via {\it restricted growth functions}, as we explain here;
see also Stanton and White \cite[Sec. 1.5]{StantonWhite}. 

Given any $k$-block set partition $\pi=\{B_1,\ldots,B_k\}$
of $\{1,2,\ldots,n\}$, re-index the blocks so
that $\min B_1 < \min B_2 < \cdots < \min B_k$.
Then the {\it restricted growth function (rgf)}
encoding $\pi$ is the sequence $(i_1,i_2,\ldots,i_n)$
defined by $i_j=\ell$ if $j \in B_\ell$ for $j=1,2,\ldots,n$.  By definition, $i_1:=1$ and $i_j \leq 1+\max\{ i_0,i_1,\ldots,i_{j-1}\}$; 
 it is not hard to check that these two properties characterize the rgf's.  For example, with $n=15$ and $k=3$, this set partition
$$
\pi=
\{
\underbrace{\{\mathbf{1},2,3,5,8,10,15\}}_{B_1},
\underbrace{\{\mathbf{4},6,7,12\}}_{B_2},
\underbrace{\{\mathbf{9},11,13,14\}}_{B_3}
\}
$$
corresponds to this rgf $(i_1,i_2,\ldots,i_{15})$:
$$
\begin{array}{rcccccccccccccccl}
i_1&i_2&i_3&i_4&i_5&i_6&i_7&i_8&i_9&i_{10}&i_{11}&i_{12}&i_{13}&i_{14}&i_{15} \\
1&1&1&2&1&2&2&1&3&1&3&2&3&3&1
\end{array}
$$
We claim that these rgf's correspond bijectively to the standard monomial $\kk$-basis for $A(k+1)^!_{n-k}$ given
in Corollary~\ref{thm: shriek-presentations}.  To explain
this bijection, underline the first (leftmost) occurrence of each value $p=1,2,\ldots,k$ among the $i_j$, and  append an extra (underlined) $i_{n+1}:=k+1$ on the right, as a convention.
One then associates to $(i_1,\ldots,i_n)$ the product 
$m_2 \cdot m_3 \cdots m_k \cdot m_{k+1}$ where $m_p$ is the noncommutative monomial in variables $\{ y_{ip} \}_{i=1}^{p-1}$ obtained by replacing each 
non-underlined value $i_j$ above
with the variable $x_{i_j,p}$ if $p$ is the next underlined value to the right of $i_j$:
$$
\begin{array}{rcccccccccccccccl}
\underline{1}&1&1&\underline{2}&1&2&2&1&\underline{3}&1&3&2&3&3&1&\underline{4}\\
&y_{12}&y_{12}& \cdot &y_{13}&y_{23}&y_{23}&y_{13}&  \cdot &y_{14}&y_{34}&y_{24}&y_{34}&y_{34}&y_{14}&
\end{array}
$$
Since the number of non-underlined values $i_j$ is $n-k$,
this is a standard monomial in $A(k+1)^!_{n-k}$.
\end{remark}

\begin{remark}
The presentations of $\OS(\braid_n)^!, \VG(\braid_n)^!$ in Theorem~\ref{thm: shriek-presentations} are equivalent to what
Cohen and Gitler \cite{CohenGitler}
called {\it graded infinitesimal braid relations} in their presentation of the loop space homology algebra $H_*(\Omega X,\kk)$ where $X=\Conf_n(\RR^d)$; see also
Berglund \cite[Example 5.5]{Berglund}.
For the case of $\OS(\Br_n)^!$, considered as a universal enveloping algebra $\OS(\Br_n)^!=\cU(\cL)$, see also the discussion by Fresse \cite[Ch.~10]{Fresse} referring to $\cL$ as the {\it Drinfeld-Kohno Lie algebra} and $\cU(\cL)$ as the {\it algebra of chord diagrams}.
\end{remark}

\subsection{Stirling representations of the first kind: generating functions}\label{sec:Stirlingrep-gfs}

The $\kk \symm_n$-module structure for either $A(n)=\OS(\braid_n)$ or $\VG(\braid_n)$ are well-studied.  Explicit irreducible decompositions are not known, but can be computed reasonably efficiently through symmetric function formulas involving plethysm and generating functions, given in work of
Sundaram and Welker \cite[Thm. 4.4(iii)]{SundaramWelker} and reviewed here;  see also the summary in Hersh and Reiner \cite[Sec. 2]{HershReiner}.

Let $\kk$ be a field of characteristic zero.
The {\it Frobenius characteristic isomorphism} $R_\kk(\symm_n) \cong \Lambda_n$, where $\Lambda_n$ are the degree $n$ homogeneous symmetric functions $\Lambda(z_1,z_2,\ldots)_n$, mentioned in Section~\ref{sec: rank-two-revisited} above, can be compiled for all $n$ to give a ring isomorphism
$$
\bigoplus_{n=0}^\infty R_\kk(\symm_n) \overset{\mathrm{ch}}{\longrightarrow} \bigoplus_{n=0}^\infty \Lambda_n = \Lambda.
$$
Here the product on the left is the {\it external} or {\it induction product}
$$
([U],[V]) \longmapsto \left[ (U \otimes_\kk V)\uparrow_{\symm_a \times \symm_b}^{\symm_{a+b}} \right],
$$
while the product on the right simply multiplies symmetric functions.  If one defines the power sum symmetric function $p_r:=z_1^r+z_2^r+\cdots$, and
the $\kk$-basis $\{ p_\lambda:=p_{\lambda_1} p_{\lambda_2} \cdots\}$ indexed by partitions $\lambda$ of $n$ for $\Lambda_n$, then for each $\kk\symm_n$-module $U$ with character $\chi_U$, the Frobenius isomorphism maps
$
[U] \overset{\mathrm{ch}}{\longmapsto} \frac{1}{n!} \sum_{\sigma \in \symm_n} \chi_U(\sigma) \, p_{\lambda(\sigma)}
$
where $\lambda(\sigma)$ is the cycle type partition of $\sigma$.

Let $\Lie_n$ denote the $n^{th}$ {\it Lie representation}: the $\symm_n$-representation
on the multilinear component of the free Lie algebra on $n$ variables. It has a formula due to Klyachko \cite{Klyachko} as
$
\Lie_n=\zeta \uparrow_{C_n}^{\symm_n},
$
where $\zeta$ is the one-dimensional representation of the cyclic group $C_n$
inside $\symm_n$ generated by an $n$-cycle $c$, 
that sends $c \mapsto e^{\frac{2\pi i}{n}}$.
Defining symmetric functions
$$
\begin{aligned}
    \ell_n &:= \ch(\Lie_n),\\
    \pi_n &:=  \ch(\mathrm{sgn}_n \otimes \Lie_n),
\end{aligned}
$$
and letting $(f,g) \mapsto f[g]$ denote {\it plethystic composition} of symmetric functions \cite[\S I.8]{Macdonald}, one has the following plethystic expressions and product generating functions
(see Sundaram \cite[Thm. 1.8, and p.249]{Sundaram}, Sundaram and Welker \cite[Thm. 4.4(iii)]{SundaramWelker} and Hersh and Reiner \cite[\S 2, Thm. 2.17]{HershReiner}):

\begin{align}
\label{braid-VG-plethysm-formula}
1+\sum_{n=1}^\infty u^n \sum_{k=1}^n     \ch([\VG(\braid_n)_{n-k}]) t^k 
&= \sum_{\lambda=1^{m_1} 2^{m_2} \cdots }
u^{|\lambda|}  t^{\ell(\lambda)} 
\prod_{j\geq 1} h_{m_j}[\ell_j]  \\
\label{braid-VG-gf}
 &= \prod_{m=1}^\infty (1-u^m p_m)^{-a_m(t)}, \\
\label{braid-OS-plethysm-formula}    
1+\sum_{n=1}^\infty u^n \sum_{k=1}^n     \ch([\OS(\braid_n)_{n-k}]) t^k 
&= \sum_{\lambda=1^{m_1} 2^{m_2} \cdots }
u^{|\lambda|}  t^{\ell(\lambda)} \prod_{j\text{odd }} h_{m_j}[\pi_j] \prod_{\substack{j\text{ even}\\j\geq 2}} e_{m_j} [\pi_j],\\
\label{braid-OS-gf}
   &= \prod_{m=1}^\infty (1+(-u)^m p_m)^{a_m(-t)},
\end{align}
where here 
$
a_m(t) \coloneqq \frac{1}{m} \sum_{d|m} \mu(d) t^{\frac{m}{d}},
$
with $\mu(d)$ the number-theoretic M\"obius function.
Equivalently, define for a partition $\lambda=1^{m_1} 2^{m_2} \cdots $ of $n$ (written $\lambda \vdash n$) with $m_i$ parts equal to $i$, the $\symm_n$-representations $\OS_\lambda, \VG_\lambda$ whose Frobenius characteristics are
the products appearing above.  Then
\begin{equation}
\begin{array}{rclcccl}
\label{OS-and-VG-lambda} 
\ch\, \VG_\lambda&:=& \displaystyle\prod_{i} h_{m_i}[\ell_i], & \text{ so that }&\ch(\VG(\braid_n)_{n-k})
&=&\displaystyle\sum_{\substack{\lambda\vdash n\\\ell(\lambda)=k}} \VG_\lambda,\\
\ch\, \OS_\lambda&:=& \displaystyle\prod_{i\ \mathrm{odd}} h_{m_i}[\pi_i]\prod_{i\ \mathrm{even}} e_{m_i}[\pi_i], 
& \text{ so that } &\ch(\OS(\braid_n)_{n-k})
&=&\displaystyle\sum_{\substack{\lambda\vdash n\\\ell(\lambda)=k}} \OS_\lambda.\\
\end{array}
\end{equation}
Thus $\VG_{(n)}$ is the {\it Lie representation} with $\ch \VG_{(n)}=\ell_n$ mentioned above.
Similarly, $\{ \VG_\lambda \}$ are called {\it higher Lie characters}; see  
Schocker \cite{Schocker}. 
Also, the last equality in \eqref{OS-and-VG-lambda} implies that $\OS(\braid_n)_{n-k}$ coincides with the $\symm_n$-representation on the $(n-k)$th {\it Whitney homology} of the partition lattice, $1\le k\le n$; see Lehrer and Solomon \cite[Thm. 4.5]{LehrerSolomon}, Sundaram \cite[Thm. 1.8]{Sundaram}.

\subsection{Data on Stirling representations of the second kind}

In contrast to the above $\kk\symm_n$-descriptions of
$A(n)_i$ when $A(n)=\OS(\braid_n), \VG(\braid_n)$, for the Koszul duals $A(n)^!_i$, we currently lack formulas of this nature, although we can
tabulate $A(n)^!_i$ recursively from the $A(n)_i$ using \eqref{equivariant-recurrence-for-shrieks}. 

\begin{question}
Are there formulas like \eqref{braid-VG-plethysm-formula}, \eqref{braid-VG-gf}, \eqref{braid-OS-plethysm-formula}, \eqref{braid-OS-gf}
for the duals $\VG(\braid_n)^!, \OS(\braid_n)^!$?
\end{question}

\subsection{Branching rules for both kinds of Stirling representations}\label{sec:branching-Stirlingreps}

Stirling numbers of both kinds satisfy well-known recurrences, mentioned in the Introduction:
\begin{align}
\label{cnk-recursion}
c(n,k) &= (n-1)\cdot c(n-1,k) + c(n-1,k-1) \\
\label{Snk-recursion}
S(n,k) &= k\cdot S(n-1,k) + S(n-1,k-1)
\end{align}
Theorem~\ref{thm: supersolvable-branching-exact-sequences} will allow us to lift these to branching rules for the Stirling representations of both kinds.  We consider here the action of $G=\symm_n$ on the matroid and oriented matroid $\braid_n$.
In this case, the setwise $\symm_n$-stabilizer for the modular coatom 
$F=E_{n-2}$ in \eqref{braid-matroid-hands}
is the subgroup $H=\symm_{n-1}$.
Furthermore, the permutation action $\mathcal{X}$ of $\symm_{n-1}$ on the set
$$
E_{n-1}=E \setminus F=\{\{1,n\},\{2,n\},\ldots,\{n-1,n\}\}
$$
and its signed permtuation action on the real vectors representing $E_{n-1}$
in the oriented matroid $\braid_n$
$$
\{ v_{1n},v_{2n},\ldots,v_{n-1,n}\}=\{e_1-e_n, e_2-e_n,\cdots,e_{n-1}-e_n\},
$$
are both equivalent to the defining
$\symm_{n-1}$-permutation representation $\chi^{(n-1)}_{\mathrm{def}}$
via $(n-1) \times (n-1)$ permutation matrices.
Translating Theorem~\ref{thm: supersolvable-branching-exact-sequences} then immediately
implies the following.

\begin{cor}
\label{cor: braid-matroid-branching}
For any field $\kk$, the recurrences \eqref{cnk-recursion}, \eqref{Snk-recursion} lift equivariantly as follows.
\begin{itemize}
    \item[(i)] Letting $\cnkrep(n,k)$ denote either $\OS(\braid_n)_{n-k}$ or $\VG(\braid_n)_{n-k}$ as $\kk\symm_n$-module, the recurrence \eqref{cnk-recursion} lifts to two recurrences in $R_\kk(\symm_{n-1})$
\begin{equation}
\label{braid-primal-branching-recursion}
[\cnkrep(n,k)\downarrow^{\symm_n}_{\symm_{n-1}} ]
= [\chi^{(n-1)}_{\mathrm{def}}] \cdot  [\cnkrep(n-1,k)] + [\cnkrep(n-1,k-1)],
\end{equation}
reflecting two $\kk\symm_{n-1}$-module exact sequences
\begin{equation}
\label{braid-primal-branching-sequence}
0 \longrightarrow \cnkrep(n-1,k-1) \longrightarrow \cnkrep(n,k)\downarrow^{\symm_n}_{\symm_{n-1}} \longrightarrow \chi^{(n-1)}_{\mathrm{def}} \otimes \cnkrep(n-1,k) \longrightarrow 0.
\end{equation}

 \item[(ii)] Letting $\Snkrep(n,k)$ denote either $\OS(\braid_{k+1})^!_{n-k}$ or $\VG(\braid_{k+1})^!_{n-k}$ as $\kk\symm_{k+1}$-module, the recurrence \eqref{Snk-recursion} lifts to two relations in $R_\kk(\symm_{k})$ 
\begin{equation}
\label{braid-shriek-branching-recursion}
[\Snkrep(n,k)\downarrow^{\symm_{k+1}}_{\symm_{k}} ]
= [\chi^{(k)}_{\mathrm{def}}]  \cdot[\Snkrep(n-1,k)\downarrow^{\symm_{k+1}}_{\symm_{k}}] + [\Snkrep(n-1,k-1)],
\end{equation}

reflecting two $\kk\symm_k$-module exact sequences
\begin{equation}
\label{braid-shriek-branching-sequence}
0 \longrightarrow 
\chi^{(k)}_{\mathrm{def}} \otimes \Snkrep(n-1,k)\downarrow^{\symm_{k+1}}_{\symm_{k}}
 \longrightarrow \Snkrep(n,k)\downarrow^{\symm_{k+1}}_{\symm_{k}} \longrightarrow 
\Snkrep(n-1,k-1) 
 \longrightarrow 0.
\end{equation}

\end{itemize}
\end{cor}
\begin{remark}
Proposition~\ref{prop: general-branching} implies the two versions of \eqref{braid-primal-branching-recursion} are equivalent to those of \eqref{braid-shriek-branching-recursion}.
\end{remark}

\begin{remark}
All of the assertions Corollary~\ref{cor: braid-matroid-branching} are new, as far as we know, when working over
fields $\kk$ of positive characteristic, and \eqref{braid-primal-branching-sequence},\eqref{braid-shriek-branching-recursion},\eqref{braid-shriek-branching-sequence} are new even when $\kk$ has characteristic zero.  
However, when $\kk$ has characteristic zero, it turns out that \eqref{braid-primal-branching-recursion} also follows from work of Sundaram in \cites{Sundaram, SuLiesup2}.
For example, the relation \eqref{braid-primal-branching-recursion} for $\cnkrep(n,k)=\OS(\braid_n)_{n-k}$ can be deduced by combining \cite[Thm~2.2, Part (2) and Prop. 1.9]{Sundaram}; we omit the details here.
Also it turns out that both cases of \eqref{braid-primal-branching-recursion}, when either $\cnkrep(n,k)=\OS(\braid_n)_{n-k}$ or $\VG(\braid_n)_{n-k}$, follow from the symmetric function branching result \cite[Thm. 4.10]{SuLiesup2}.  In the notation there,  choosing $F=\sum_{n\ge 1} \ell_n$, one takes $G^j_n=h_j[F]|_{\deg n}$ to deduce 
\eqref{braid-primal-branching-recursion} for
$\cnkrep(n,k)=\VG(\braid_n)_{n-k}$,
and one takes $G^j_n=e_j[F]|_{\deg n}$
to deduce \eqref{braid-primal-branching-recursion} for
$\cnkrep(n,k)=\OS(\braid_n)_{n-k}$. We
again omit the details here.
\end{remark}

\subsection{Braid matroids and representation stability}
\label{sec: braid-matroid-rep-stability}

Here we wish to apply the representation stability 
results of Section~\ref{sec: rep-stability} to the
braid matroids $\braid_n$.  A special case of the main result of Church and Farb \cite{ChurchFarb} shows, in our language, that for each fixed $i=0,1,2\ldots$, both sequences $\{A(n)_i\}$ where 
$A(n)=\OS(\braid_n), \VG(\braid_n)$ are representation stable.   Hersh and Reiner \cite[Thm. 1.1]{HershReiner} pinned down the onset of this representation stability.

\begin{theorem}
\label{HershR-theorem}
For each fixed $i \geq 1$,
both sequences $\{A(n)_i\}$ where 
$A(n)=\OS(\braid_n)$ and $\VG(\braid_n)$ are representation stable, past $3i$ for $\VG(\braid_n)$ and past $3i+1$ for $\OS(\braid_n)$.
\end{theorem}

One then deduces the following representation stability for their
Koszul duals.
\begin{cor}
\label{braid-matroid-shrieks-have-rep-stability-cor}
For each fixed $i \geq 1$,
both sequences $\{A(n)^!_i\}$ where 
$A(n)=\OS(\braid_n)$ and $\VG(\braid_n)$ are representation stable, past $3i$ for $\{\VG(\braid_n)^!_i\}$ and past $4i$ for $\{\OS(\braid_n)^!_i\}$.
\end{cor}
\begin{proof}
Theorem \ref{HershR-theorem} gives the necessary hypotheses to apply  Corollary~\ref{cor: shrieks-have-rep-stability-cor}, using the constant $c=3$ for $\{\VG(\braid_n)_i\}$ and using the constant $c=4$ (since $3i+1\leq 4i$)
for $\{\OS(\braid_n)_i\}$.
\end{proof}

\begin{remark} 
The bounds in Corollary~\ref{braid-matroid-shrieks-have-rep-stability-cor} happen to be tight for $\OS(\braid_n)^!_i, \VG(\braid_n)^!_i$ when $i=0,1,2$.  To see this, one can apply Lemma~\ref{lem: Hemmer-like} to 
Propositions~\ref{prop:A!n1-S(n,n-1)}, \ref{prop:A!n2OS-S_OS(n+1,n-1)} and Remark~\ref{rem:A!n2VG-S_VG(n+1,n-1)} below (specifically, see equations \eqref{eqn:S-OS(n+1,n-1)-h-rep-stab-bound-is-8},
\eqref{eqn:S-VG(n+1,n-1)-h-rep-stab-bound-is-6}) to deduce that 
for $i=0,1,2$, the sequences $\{ \OS(\braid_n)_i^!\}$ start to stabilize exactly when $n \geq 4i$, and the sequences $\{ \VG(\braid_n)_i^!\}$ start to stabilize exactly when $n \geq 3i$. This suggests
the following conjecture, confirmed by {\tt Sage/Cocalc} for $\OS(\braid_n)_i^!$ up to $i=5$, and for $\VG(\braid_n)_i^!$ up to  $i=7$. 

\begin{conj}\label{conj: onset-stability-shrieks}
    The bounds for onset of stability in Corollary~\ref{braid-matroid-shrieks-have-rep-stability-cor} are tight:
    for $i \geq 0$, the sequences $\{\OS(\braid_n)_i^!\}$ and $\{ \VG(\braid_n)_i^!\}$
    start to stabilize exactly when $n = 4i$ and $n = 3i$, respectively.  
\end{conj}

\end{remark}

Since $A(n)=\OS(\braid_n), \VG(\braid_n)$
are anti-commutative and commutative, respectively,
Corollary~\ref{cor: primitives-inherit-rep-stability} immediately implies the following.

\begin{cor}
\label{cor: braid-matroid-primitives-are-rep-stable}
Letting $A(n)=\OS(\braid_n),\VG(\braid_n)$, and defining
$\cL(n)_i$ by $A(n)^!=\cU(\cL(n))$, for each fixed $i=1,2,\ldots$, the sequence $\{ \cL(n)_i \}$ is  
representation stable.
\end{cor}

\begin{remark}
    The case of Corollary~\ref{cor: braid-matroid-primitives-are-rep-stable} for $A(n)=\OS(\braid_n)$ also follows from work of Church, Ellenberg and Farb \cite[Thm.~7.3.4]{CEF}.  They consider instead
    of $\cL(n)$ the {\it Malcev Lie algebra} $\mathfrak{p}_n$ associated to the fundamental group $\pi_1(X)$ for the configuration space $X=\Conf_n(\RR^2)=\Conf_n(\CC)$ considered in Section~\ref{subsec: configuration-spaces};  alternatively, $X$ is the complement of the complex braid arrangement $\cA$ as in Theorem~\ref{OS-theorem}.  These two Lie algebras
    $\mathfrak{p}_n$ and $\cL(n)$
    coincide due to the $1$-formality of complements of complex algebraic hypersurfaces; see, e.g., Suciu and Wang \cite[\S 6,7]{SuciuWang}.
\end{remark}

Computations in {\tt Sage/Cocalc} through $i=8$ suggest the following conjecture.

\begin{conj}\label{conj: onset-of-stability-primitives} Defining $\{\cL(n)_i\}$ by $A(n)^! = \cU(\cL(n))$, its onset of representation stability is:
\begin{itemize}
    \item $n = 2i$ for a fixed $i\geq 1$ when $A(n) = \OS(\braid_n)$,
    \item $n = 2i$ for a fixed $i\geq 3$ when $A(n) = \VG(\braid_n)$.
\end{itemize}
\end{conj}

\subsection{Near-boundary cases for
Stirling representations of the second kind}
\label{sec: braid-OS-permutation-reps}
Stirling numbers $S(n,k)$ of the second
kind have more explicit formulas
when either $k$ or $n-k$ is small.
We similarly present here more explicit formulas,
in the language of symmetric functions 
for the Stirling representations
\begin{align*}
    \OS(\braid_n)_i&=\Snkrep_\OS((n-1)+i,n-1),\\
   \VG(\braid_n)_i&=\Snkrep_\VG((n-1)+i,n-1),
\end{align*}
as $\kk\symm_n$-modules, when either $i$ or $n$ is small.  

Part of our motivation comes from the following
observations about when $\OS(M)_i, \VG(\OM)_i$
and their Koszul duals $\OS(M)^!_i, \VG(\OM)^!_i$ turn out to be {\it permutation representations}
of their automorphism groups $G=\Aut(M)$ or $\Aut(\OM)$.  The discussion of Boolean matroids Section~\ref{sec: Boolean-matroids-revisited} and low rank matroids
in Section~\ref{sec: rank-one-matroids} and Proposition~\ref{prop: rank-two-OS-shrieks-are-perm-reps} 
showed that 
\begin{itemize}
    \item $\OS(M)_i, \VG(\OM)_i$ are {\it rarely} permutation representations, 
    \item $\VG(\OM)^!_i$ is {\it not always} a permutation representation, 
    \item but $\OS(M)^!_i$ was {\it always} a $G$-permutation representation in
these previous examples.
\end{itemize}
For the braid matroids $\braid_n$, it  is 
{\it not always} true that $\OS(\braid_n)_i$ is
a permutation representation, but the next result shows that it happens
in many cases where $i$ or $n$ is small.

 \begin{theorem}
     \label{thm: OS-perm-reps}
     For $\kk$ of characteristic zero,
    the 
    $\kk \symm_n$-modules 
    $
    \OS(\braid_n)^!_i=\Snkrep_\OS((n-1)+i,n-1)
    $
    \begin{itemize}
       \item[(i)]  are permutation modules for $i=0,1$,
       \item[(ii)]
       are {\bf half}-permutation modules for $i=2$, meaning that
       $2 \cdot [\OS(\braid_n)^!_2] $ is the class of a permutation module in $R_\kk(\symm_n)$,
        \item[(iii)] are permutation modules\footnote{And
        $[\OS(\braid_n)_i]$ are even {\it $h$-positive} permutation modules when $n=1,2,3$, overlapping with the discussion 
       in Section~\ref{sec: rank-two-revisited} on rank two matroids, as $U_{2,3} = \braid_3$.} for $n=1,2,3,4,5$.
\end{itemize}
However, both
      \begin{align*}
             \Snkrep_\OS(10,5)&=\OS(\braid_6)^!_5,\\
           \Snkrep_\OS(11,6)&=\OS(\braid_7)^!_5
        \end{align*}
         fail to be  permutation modules, even after scaling them by positive integers, since they can be shown to have  negative character values.  
\footnote{Trevor Karn's Burnside Solver further shows that  $\Snkrep_\OS(5+i,5)$ is a permutation module for $i\le 4$ and $i=6, 7, 10$; it fails to be one at  $i=8,9$.}
 \end{theorem}


\begin{table}[H]
\begin{tikzpicture}[scale = 0.56]
\draw[gray, very thin] (0, 0.1) grid (10,9);
\draw[->] (-0.5,9) -- (10.5,9);
\draw[->] (0, 9) -- (0, 0); 
\node[right] (c) at (10.5, 9) {$i$};
\node[below] (r) at (0, 0) {$n$};

\foreach \c in {0,1,2,3,4,5,6,7,8} {
	\node[above] (\c) at (\c+1, 9) {\tiny \c};};
\foreach \c in {2,3,4,5,6,7,8,9} {
	\node[left] (\c +10) at (0, 10-\c) {\tiny {\c}};
	};
	
\foreach \c/\r in {
0/2,0/3,0/4,0/5,0/6,0/7,0/8,0/9,
1/2,1/3,1/4,1/5,1/6,1/7,1/8,1/9,
2/2,2/3,2/4,2/5,2/6,
3/2,3/3,3/4,3/5,3/6,3/7,3/8,
4/2,4/3,4/4,4/5,4/6,4/7,
5/2,5/3,5/4,5/5,
6/2,6/3,6/4,6/5,6/6,
7/2,7/3,7/4,7/5,7/6,
8/2,8/3,8/4,8/5
} {
	\draw[thick, fill = black!75] (\c + 1, 10 -\r) circle (0.25);
	};
	
\foreach \c/\r in {
2/7,2/8,2/9, 3/9
} {
	\fill[black!75] (\c+1, 9.75 - \r) -- (\c+1, 10.25 - \r) arc (90:270:0.25) -- cycle;
	\draw[thick] (\c + 1, 10 - \r) circle (0.25);
	};
	
\foreach \c/\r in {
5/6,5/7
} {
	\draw[thick] (\c + 1, 10-\r) circle (0.25);
	};
	
\draw[thick, fill = black!75] (13, 6) circle (0.25);
\node[right] (y) at (13.5, 6) {\small Permutation module};
\draw[thick] (13, 5) circle (0.25);
\node[right] (n) at (13.5, 5) {\small \textbf{Not} a ``fraction'' of a permutation module};
\fill[black!75] (13, 3.75) -- (13, 4.25) arc (90:270:0.25) -- cycle;
\draw[thick] (13, 4) circle (0.25);
\node[right] (m) at (13.5, 4) {\small Half-permutation module};
\end{tikzpicture}
\caption{When are $[\OS(\braid_n)^!_i] =\Snkrep_\OS((n-1)+i,n-1)$ permutation modules or ``fractions'' thereof?}
\label{table:of:when:Stirling:second:reps:are:perm:mods}
\end{table}

Table \ref{table:of:when:Stirling:second:reps:are:perm:mods} summarizes the results of \Cref{thm: OS-perm-reps}; an outline of the proof  appears in \Cref{appendix: OS-VG-perm-reps}.


\section{Further remarks and questions}
\label{sec: remarks-and-questions}
We remark here on some further directions which
could merit exploration.

\subsection{Projective geometries over finite fields}
\label{sec: projective-geometries}

The Boolean matroids $U_{n,n}$ discussed in Example~\ref{ex: Boolean-matroid-1} and Section~\ref{sec: Boolean-matroids-revisited}
have a well-studied ``$q$-analogue": the {\it projective geometries} $PG(n,q)$, associated with the finite vector spaces $\FF_q^n$.  These $PG(n,q)$ are non-orientable simple matroids whose ground set $E=\PP(\FF_q^n)=\PP^{n-1}_{\FF_q}$ is the set of points in a finite projective space, so $|E|=[n]_q:=1+q+q^2+\cdots+q^{n-1}$, with poset of flats $\cF$ given by the lattice of all subspaces in $\FF_q^n$; see Oxley \cite[\S 6.1]{Oxley} and Orlik and Terao \cite[Example 4.33]{OrlikTerao}.  The lattices $\cF$ are {\it modular}, meaning that every flat is a modular element, so that every complete flag $\underline{F}$ is a modular complete flag.
Hence the matroids $PG(n,q)$ are  supersolvable, with exponents $(e_1,\ldots,e_n)=(1,q,q^2,\ldots,q^{n-1})$.
Consequently, the family of $\kk$-algebras $A(n):=\OS(PG(n,q))$ is  Koszul, 
satisfying
\begin{align*}
\Hilb(A(n),t)&=
(1+t)(1+qt)(1+q^2t)\cdots(1+q^{n-1}t) 
\quad \text{ with }
\dim_\kk A(n)^!_i=q^{\binom{i}{2}} \qbinom{n}{i}\\
\Hilb(A(n)^!,t)&=
\frac{1}{(1-t)(1-qt)(1-q^2t)\cdots(1-q^{n-1}t)}
\quad \text{ with }
\dim_\kk A(n)^!_i=\qbinom{n+i-1}{i},
\end{align*}
where $\qbinom{n}{k}:=\frac{[n]!_q}{[k]!_q [n-k]!_q}$
with $[n]!_q:=[n]_q[n-1]_q \cdots [2]_q [1]_q$; see Macdonald \cite[Example I.2.2]{Macdonald}.

\begin{problem}
Study $A(n)=\OS(PG(n,q))$ and $A(n)^!=\OS(PG(n,q))^!$ as $GL_n(\FF_q)$-representations.
\end{problem}

For example, the {\it $q$-Pascal recurrences} for $A(n)_i=\qbinom{n}{i}$ and $A(n)^!_i=\qbinom{n+i-1}{i}$ will have lifts to branching rules via Proposition~\ref{prop: general-branching} and Theorem~\ref{thm: supersolvable-branching-exact-sequences}. There is also an appropriate
analogue here of {\it representation stability for $GL_n(\FF_q)$-representations} developed by Putman and Sam \cite{PutmanSam}.

\subsection{Type $B$, wreath products, and Dowling geometries}
\label{sec: dowling}
As mentioned in Section~\ref{sec: stirling-reps}, the braid matroids $\braid_n$ are represented by the 
root systems of type $A_{n-1}$, accounting for the action of the reflection group $\symm_n$ on them as symmetries. 

There are other real and complex reflection groups giving rise to matroids with large symmetry, but relatively few of these matroids are supersolvable; see Hoge and R\"ohrle \cite{HogeRohrle} for their classification.  They include the dihedral reflection groups giving rise to the 
rank two matroids already discussed in Example~\ref{ex: rank-two-matroid-1} and Section~\ref{sec: rank-two-revisited}.  
They also include the {\it reflection groups of type $B_n$ or $C_n$}, isomorphic to the {\it hyperoctahedral group} or {\it signed permutation group} $\symm^\pm_n$ that appeared in Section~\ref{sec:symmetry-of-OS-and-VG}. Their root systems can be realized over $\RR$, giving rise to an oriented matroid
from the positive roots
\begin{equation}
    \label{type-B-positive-roots}
\Phi_{B_n}^+:=\{+e_i \pm  e_j\}_{1 \leq i < j \leq n} 
\,\, \sqcup \,\,
\{e_i\}_{1 \leq i \leq n}.
\end{equation}
More generally, one has the complex reflection groups $\symm_n[\ZZ/m\ZZ] = (\ZZ/m\ZZ) \wr \symm_n$ for $m \geq 2$, also known as the groups $G(m,1,n)$ within Shephard and Todd's classification \cite{ShephardTodd}
of irreducible complex reflection groups.  Letting $\zeta_m:=e^{\frac{2 \pi i}{m}}$, their associated matroids can be represented by this list of vectors in $\CC^n$:
\begin{equation}\label{eq: matroid-representable-over-C-not-R}
\{e_i - \zeta^k e_j\}_{\substack{1 \leq i < j \leq n\\0 \leq k \leq m-1}} 
\,\, \sqcup \,\,
\{e_i\}_{1 \leq i \leq n}.
\end{equation}
These matroids are not realizable over $\RR$ (and not orientable) unless $m=2$ where they recover the type $B_n/C_n$ reflection groups.

Motivated by these examples,
Dowling \cite{Dowling} introduced a more general class of matroids, now known as 
the {\it Dowling geometries} $\DowlingGeometry{n}{G}$; see Oxley \cite[\S 6.10]{Oxley} for definitions and discussion.
Here $G$ is {\it any} finite group, and the matroid automorphisms of $\DowlingGeometry{n}{G}$
contain the wreath product
$\symm_n[G] = G \wr \symm_n$.
Interestingly, Dowling proved that the matroid $\DowlingGeometry{n}{G}$ is representable over a field $\FF$ {\it if and only if} the finite group $G$ is a subgroup of $\FF^\times$;  in particular, this forces $G$ to be cyclic,
as in the complex reflection
groups $\symm_n[\ZZ/m\ZZ]$ mentioned above. 

Dowling also showed that $\DowlingGeometry{n}{G}$
is supersolvable for any finite group $G$.  Consequently,
their Orlik-Solomon algebras $\OS(\DowlingGeometry{n}{G})$ are always Koszul, and when $|G|=2$, the same holds for the Varchenko-Gel'fand ring $\VG(\OM(B_n))$, e.g., if $\OM(B_n)$ is realized by the vectors in \eqref{type-B-positive-roots} above.

\begin{problem}
Study these families of Koszul algebras $A(n)=\OS(\DowlingGeometry{n}{G})$ and $\VG(\OM(B_n))$, along with their Koszul duals $A(n)^!$, 
as $\symm_n[G]$-representations.
\end{problem}

If $m:=|G|$, then the exponents for the supersolvable matroids $\DowlingGeometry{n}{G}$ turn out to be
$$
(e_1,e_2,\ldots,e_n)=(1,m+1,2m+1,\ldots,(n-1)m+1).
$$
Combining this with Dowling's formulas \cite[\S 4]{Dowling}, for the rank sizes\footnote{Also called the {\it Whitney numbers of the second kind for the poset.}} in the poset of flats of
$\DowlingGeometry{n}{G}$, 
one encounters a similar
coincidence to the equality 
$\dim_\kk \OS(\braid_{n})_i=S((n-1)+i,n-1)$ discussed in Remark~\ref{rem: Stirling-number-coincidence}:  the dimension of $\OS(\DowlingGeometry{n}{G})^!_i$ is
the size of the $(n-1)^{st}$ rank in the flat poset of $\DowlingGeometry{(n-1)+i}{G}$. 
This again reflects a bijection
between the standard monomial $\kk$-basis for 
$\OS(\DowlingGeometry{n}{G})^!_i$ from Theorem~\ref{thm: shriek-presentations}
and an encoding of flats in $\DowlingGeometry{n}{G}$  generalizing restricted growth 
functions, similar to work of
Komatsu, Bagno, and Garber \cite[\S 2.3]{BagnoGarberKomatsu}.
We omit the details here.

\subsection{Equivariant degree one injections}\label{sec:VG-injectivity}

Recall the following consequences of  Theorem~\ref{thm: shrieks-have-right-NZD}: By Part (ii) of Corollary~\ref{cor: equiv-hilb-factorizations}, 
for the {\it matroid} automorphism group $G=\Aut(M)$, there are $G$-equivariant degree one injections 
\begin{equation}
\label{OS-shriek-degree-one-injections}
[\OS(M)^!_i]\hookrightarrow [\OS(M)^!_{i+1}],\text{ for all }   i\ge 0 
\end{equation}
\noindent
while Part (iii)  of Corollary~\ref{cor: equiv-hilb-factorizations} asserts that 
for the full {\it oriented} matroid automorphism group $G=\Aut(\OM)$, there are $G$-equivariant degree \emph{two} injections 
\[
[\VG(\OM)^!_i]\hookrightarrow [\VG(\OM)_{i+2}^!], \text{ for all }  \ i\ge 0. 
\]

The latter injections arise from right-multiplication by a degree two $G$-invariant $\underline{E}$-generic power sum $p_2(\yy)$, such as $p_2(\yy)=\sum_i y_i^2$.  
Unfortunately, for some oriented matroids $\OM$, there are no degree one $\underline{E}$-generic element power sums $p_1(\yy)$ in $A^!_1$ that are also $G$-invariant. For example, $p_1(\yy)=\sum_i y_i$ is {\it not} always $G$-invariant.
In fact, the calculation for rank one oriented matroids $\OM=U_{1,1}$ in \eqref{rank-one-illustration-of-VG-factorization}
shows that in that case,
there are no
$G$-equivariant injections $\VG(\OM)^!_i \hookrightarrow \VG(\OM)_{i+1}^!$ for any $i$. 

Nonetheless, for the braid matroids $\OM=\braid_n$, {\tt Sage} calculations for $n \leq 10$ and $1\le i\le 9$ support the following conjecture.

\begin{conj}
\label{conj: braid-matroid-equivariant-injectivity}
For the braid oriented matroid $\OM=\Br_n$,
there exist equivariant injections
\begin{align*}
&[\Snkrep_{\VG}((n-1)+i,n-1)]=[{\VG}(\Br_n)^!_i] \\
&\quad \hookrightarrow [\Snkrep_{\VG}((n-1)+i+1,n-1)]=[\VG(\Br_n)^!_{i+1}],\text{ for all } i\ge 1.
\end{align*}   
\end{conj}
\noindent
Propositions~\ref{n-at-most-3-braid-shrieks} and~\ref{prop:VGBr4shriek-perm} establish Conjecture~\ref{conj: braid-matroid-equivariant-injectivity}  for $n\le 4$ in characteristic zero.

We close with some observations on
a consequence of the $G$-equivariant injections 
in \eqref{OS-shriek-degree-one-injections}:
they imply that the following alternating sum in $R_\kk(G)$ is always the class of a genuine $\kk G$-module:
\begin{equation}
\label{eqn:alt-sum-OS-shriek}[\OS^!(M)_i]- [\OS^!(M)_{i-1}]+\ldots +(-1)^{i-1}[\OS^!(M)_0 
]\end{equation}

\begin{problem}
\label{prob: alternating-sum-rep}
 Investigate the genuine $\kk G$-modules \eqref{eqn:alt-sum-OS-shriek}. 
 Do they have interesting descriptions?
 \end{problem}
 
\noindent 
For example, for the braid matroid $M=\Br_n$, 
 the dimension of the genuine  module~\eqref{eqn:alt-sum-OS-shriek} is 
\begin{equation}
\label{Mansour-Munagi-number}
S(n\!-\!1\!+\!i, n\!-\!1)-S(n\!-\!2\!+\!i, n\!-\!1)+\cdots+(-1)^i S(n\!-\!1\!, n\!-\!1).
\end{equation}
This has an interpretation via a result of Mansour and Munagi \cite[Corollary 11]{MansourMunagi2014}: it is the number of set partitions of $\{1,2,\ldots,n+i\}$ into $n$ blocks, where no block contains a pair $j, j+1$ modulo $n+i$ for $1\le j\le n+i$. We know of
no accompanying $\kk \symm_n$-module 
built from these objects.

 We remark that for any matroid $M$, the alternating sum analogous to \eqref{eqn:alt-sum-OS-shriek} for $\OS(M)$, namely
\begin{equation}\label{eqn:alt-sum-OS}[\OS(M)_i]- [\OS(M)_{i-1}]+\ldots +(-1)^{i-1}[\OS(M)_0] \end{equation}
is always a genuine $\kk G$-module for $G=\Aut(M)$, isomorphic to the top homology of a rank-selected subposet of the lattice of flats.  We quickly sketch how this follows from combining these two results:
\begin{itemize}
\item \cite{OrlikSolomon} exhibits a $\kk G$-module isomorphism $\OS(M) \cong \Whit(\mathcal{L}_M)$, where $\Whit(\mathcal{L}_M)$ is the Whitney homology of the lattice of flats $\mathcal{L}_M$ of $M$, and 
\item \cite{Sundaram} If $G=\Aut(P)$ for a Cohen-Macaulay poset $P$, then the alternating sum 
in $\RR_\kk(G)$ 
\[[\Whit_i(P)]-[\Whit_{i-1}(P)]+\ldots+(-1)^{i-1}[\Whit_0(P)]\]
is $\kk G$-isomorphic to the top homology of the rank-selected subposet of $P$ consisting of the bottom $i$ nonzero ranks.  The Hopf trace argument in \cite[Lemma 1.1]{Sundaram}, written for characteristic zero, can be replaced by applying, for  an arbitrary field $\kk$,
Proposition~\ref{easy-facts-on-Grothendieck-ring-prop}(ii) to Baclawski's complex. Similarly the arguments of Baclawski and Bj\"orner as cited in \cite[Theorem~1.2]{Sundaram} can be adapted for any field $\kk$, by appealing to the isomorphism in 
\cite[p. 262, Theorem 7.9.6]{bjorner1992homology}.  Finally, the  equivariant isomorphism with $\OS(M)$ follows from 
\cite[Theorem 7.10.2]{bjorner1992homology}, extending the argument of \cite[Theorem 4.3]{OrlikSolomon} to the whole Orlik-Solomon algebra.
\end{itemize}

\section*{Acknowledgements}
The authors thank J\"{o}rgen Backelin, Vladimir Dotsenko, Ralf Fr\"{o}berg, Darij Grinberg, Shiyue Li, Ivan Marin, Anne Shepler, Keller VandeBogert, Peter Webb, Craig Westerland, and Sarah Witherspoon for helpful conversations and references. 
They are grateful to 
Trevor Karn for his wonderful {\tt Sage/Cocalc} code that checks whether a symmetric group representation is isomorphic to a permutation representation, which helped us create \Cref{table:of:when:Stirling:second:reps:are:perm:mods}. In particular, they are grateful to him for experimentally discovering the half-integers in the decomposition of  Proposition~\ref{prop:A!n2OS-S_OS(n+1,n-1)}.
They also thank two anonymous referees for their helpful suggestions.
First and second authors received partial support from NSF grants DMS-2053288 and DMS-1745638.

\newpage
\appendix
\section{Tables of irreducibles for Stirling representations}\label{appendix: tables}

\noindent This section consists of several tables for the decomposition into irreducible modules for $A(n)^!_i$ and its primitives $\cL(A(n))_i$ when $A(n) = \OS(\braid_n)$ or $\VG(\braid_n)$.
For each table, the observed onset of representation stability in each column is shaded in blue. The data is presented in terms of the Frobenius characteristics of the modules, expanded in the Schur basis.
All data was generated using SAGE code which is publicly available at \cite{Almousa_StirlingRepresentations_2024}.
\begin{table}[ht]
\centering
\resizebox{\columnwidth}{!}{
\begin{tabular}{|c||Hc|c|c|c|c|}
\hline
\diagbox[width = 0.75cm]{$n$}{$i$} &&1&2&3\\ \hline\hline
3 & $s_3$ & $s_{2,1} + s_{3}$ & $s_{1,1,1} + 2 s_{2,1} + 2 s_{3}$ & $2 s_{1,1,1} + 5 s_{2,1} + 3 s_{3}$
\\ \hline
4 & $s_4$ & \cellcolor{blue!15} $s_{2,2} + s_{3,1} + s_{4}$ &
$\begin{gathered}
    s_{1,1,1,1} + 2 s_{2,1,1} \\ + 3 s_{2,2} + 3 s_{3,1} + 3 s_{4}
\end{gathered}$ &
$\begin{gathered}
    4 s_{1,1,1,1} + 9 s_{2,1,1} \\ + 10 s_{2,2} + 11 s_{3,1} + 6 s_{4}
\end{gathered}$
\\ \hline
5 & $s_5$ & $s_{3,2} + s_{4,1} + s_{5}$ &
$\begin{gathered}
    s_{2,1,1,1} + 2 s_{2,2,1} + 2 s_{3,1,1} \\ + 4 s_{3,2} + 4 s_{4,1} + 3 s_{5}
\end{gathered}$
&
$\begin{gathered}
    2 s_{1,1,1,1,1} + 8 s_{2,1,1,1} + 13 s_{2,2,1} \\ + 15 s_{3,1,1} + 18 s_{3,2} + 16 s_{4,1} + 7 s_{5}
\end{gathered}$
\\ \hline
6 & $s_6$ & $s_{4,2} + s_{5,1} + s_{6}$ & $\begin{gathered}
    s_{2,2,2} + s_{3,1,1,1} \\ + 2 s_{3,2,1} + s_{3,3}  + 2 s_{4,1,1} \\ + 5 s_{4,2}  + 4 s_{5,1} + 3 s_{6}
\end{gathered}$ 
&
$\begin{gathered}
    3 s_{2,1,1,1,1} + 5 s_{2,2,1,1} + 7 s_{2,2,2} \\ + 10 s_{3,1,1,1} + 21 s_{3,2,1} + 8 s_{3,3} \\ + 17 s_{4,1,1} + 24 s_{4,2} + 17 s_{5,1} + 8 s_{6}
\end{gathered}$
\\ \hline
7 & $s_7$ & $s_{5,2} + s_{6,1} + s_{7}$ &
$\begin{gathered}
s_{3,2,2} + s_{4,1,1,1}  \\ + 2 s_{4,2,1} + 2 s_{4,3}  + 2 s_{5,1,1} \\ + 5 s_{5,2} + 4 s_{6,1} + 3 s_{7}
\end{gathered}$
& 
$\begin{gathered}
    s_{2,2,1,1,1} + 2 s_{2,2,2,1} \\ + 3 s_{3,1,1,1,1} + 7 s_{3,2,1,1} + 9 s_{3,2,2} \\ + 8 s_{3,3,1} + 10 s_{4,1,1,1} + 24 s_{4,2,1} + 14 s_{4,3} \\ + 17 s_{5,1,1} + 25 s_{5,2} + 18 s_{6,1} + 8 s_{7}
\end{gathered}$
\\ \hline
8 & $s_8$ & $s_{6,2} + s_{7,1} + s_{8}$ & \cellcolor{blue!15}
$\begin{gathered}
    s_{4,2,2} + s_{4,4} + s_{5,1,1,1} \\ + 2 s_{5,2,1} + 2 s_{5,3} + 2 s_{6,1,1} \\ + 5 s_{6,2} + 4 s_{7,1} + 3 s_{8}
\end{gathered}$
&
$\begin{gathered}
    s_{2,2,2,2} + s_{3,2,1,1,1} + 2 s_{3,2,2,1} + 2 s_{3,3,1,1} \\ + 2 s_{3,3,2}  + 3 s_{4,1,1,1,1} + 7 s_{4,2,1,1} + 10 s_{4,2,2}  \\ + 11 s_{4,3,1} + 6 s_{4,4} + 10 s_{5,1,1,1} + 24 s_{5,2,1} \\ + 15 s_{5,3} + 17 s_{6,1,1} + 26 s_{6,2} + 18 s_{7,1} + 8 s_{8}
\end{gathered}$
\\ \hline
9 & $s_9$ & $s_{7,2} + s_{8,1} + s_{9}$ &
$\begin{gathered}
s_{5,2,2} + s_{5,4} + s_{6,1,1,1} \\ + 2 s_{6,2,1} + 2 s_{6,3} + 2 s_{7,1,1} \\ + 5 s_{7,2} + 4 s_{8,1} + 3 s_{9}
\end{gathered}$
& 
$\begin{gathered}
    s_{3,2,2,2} + s_{4,2,1,1,1} + 2 s_{4,2,2,1} \\ + 2 s_{4,3,1,1} + 3 s_{4,3,2} + 3 s_{4,4,1} \\ + 3 s_{5,1,1,1,1} + 7 s_{5,2,1,1} + 10 s_{5,2,2} \\ + 11 s_{5,3,1} + 7 s_{5,4} + 10 s_{6,1,1,1} + 24 s_{6,2,1} \\ + 16 s_{6,3} + 17 s_{7,1,1} + 26 s_{7,2} + 18 s_{8,1} + 8 s_{9}
\end{gathered}$
\\ \hline
10 & $s_{10}$ & $s_{8,2} + s_{9,1} + s_{10}$ & 
$\begin{gathered}
    s_{6,2,2} + s_{6,4} + s_{7,1,1,1} \\ + 2 s_{7,2,1} + 2 s_{7,3} + 2 s_{8,1,1} \\ + 5 s_{8,2} + 4 s_{9,1} + 3 s_{10}
\end{gathered}$
&$\begin{gathered}
s_{4,2,2,2} + s_{4,4,2} + s_{5,2,1,1,1} + 2 s_{5,2,2,1} \\ + 2 s_{5,3,1,1} + 3 s_{5,3,2} + 3 s_{5,4,1} + s_{5,5} \\ + 3 s_{6,1,1,1,1} + 7 s_{6,2,1,1} + 10 s_{6,2,2} + 11 s_{6,3,1} \\ + 8 s_{6,4}  + 10 s_{7,1,1,1} + 24 s_{7,2,1} + 16 s_{7,3} \\ + 17 s_{8,1,1} + 26 s_{8,2} + 18 s_{9,1} + 8 s_{10}
\end{gathered}$
\\ \hline
11 & $s_{11}$ & $s_{9,2} + s_{10,1} + s_{11}$ &
$\begin{gathered}
    s_{7,2,2} + s_{7,4} + s_{8,1,1,1} \\ + 2 s_{8,2,1} + 2 s_{8,3} + 2 s_{9,1,1} \\ + 5 s_{9,2} + 4 s_{10,1} + 3 s_{11}
\end{gathered}$
&
$\begin{gathered}
    s_{5,2,2,2} + s_{5,4,2} + s_{6,2,1,1,1} + 2 s_{6,2,2,1} \\ + 2 s_{6,3,1,1} + 3 s_{6,3,2} + 3 s_{6,4,1} + 2 s_{6,5} \\ + 3 s_{7,1,1,1,1} + 7 s_{7,2,1,1} + 10 s_{7,2,2} + 11 s_{7,3,1} \\ + 8 s_{7,4} + 10 s_{8,1,1,1} + 24 s_{8,2,1} + 16 s_{8,3} \\ + 17 s_{9,1,1} + 26 s_{9,2} + 18 s_{10,1} + 8 s_{11}
\end{gathered}$
\\ \hline
12 &
$s_{12}$ & $s_{10,2} + s_{11,1} + s_{12}$ & 
$\begin{gathered}s_{8,2,2} + s_{8,4} + s_{9,1,1,1} \\ + 2 s_{9,2,1} + 2 s_{9,3} + 2 s_{10,1,1} \\ + 5 s_{10,2} + 4 s_{11,1} + 3 s_{12}\end{gathered}$ & 
\cellcolor{blue!15}$\begin{gathered}
s_{6,2,2,2} + s_{6,4,2} + s_{6,6}+ s_{7,2,1,1,1} \\ + 2 s_{7,2,2,1} + 2 s_{7,3,1,1} + 3 s_{7,3,2} \\ + 3 s_{7,4,1} + 2 s_{7,5} + 3 s_{8,1,1,1,1} \\ + 7 s_{8,2,1,1} + 10 s_{8,2,2} + 11 s_{8,3,1} \\ + 8 s_{8,4} + 10 s_{9,1,1,1} + 24 s_{9,2,1} + 16 s_{9,3} \\+ 17 s_{10,1,1} + 26 s_{10,2} + 18 s_{11,1} + 8 s_{12}\end{gathered}$
\\ \hline
\end{tabular}
}
\caption{Irreducibles for $[{\OS}(\Br_n)^!_i] = [\Snkrep_{\OS}((n-1)+i,n-1)]$. Note that $\ch[\OS(\braid_1)^!_i] = s_2$ for $i\geq 0$ and $\ch [\OS(\braid_n)^!_0] = s_n$ for $n\geq 2$.
}
\label{table: OS-shriek}
\end{table}
\begin{table}[ht]
\centering
\resizebox{\columnwidth}{!}{
\begin{tabular}{|c||c|c|c|c|}
\hline
\diagbox[width=0.75cm]{$n$}{$i$} &0&1&2&3\\ \hline\hline
2 & $s_2$ & $s_{1,1}$ & $s_2$ & $s_{1,1}$
\\ \hline
3 & $s_3$ & \cellcolor{blue!15} $s_{1,1,1} + s_{2,1}$ & $s_{1,1,1} + 2 s_{2,1} + 2 s_{3}$ & $3 s_{1,1,1} + 5 s_{2,1} + 2 s_{3}$
\\ \hline
4& $s_4$ & $s_{2,1,1} + s_{3,1}$ & 
$\begin{gathered}
    s_{1,1,1,1} + 3 s_{2,1,1} \\+ 2 s_{2,2} + 3 s_{3,1} + 2 s_{4}
\end{gathered}$
& 
$\begin{gathered}
    4 s_{1,1,1,1} + 11 s_{2,1,1} + 8 s_{2,2} \\ + 11 s_{3,1} + 4 s_{4}
\end{gathered}$
\\ \hline
5 & $s_5$ & $s_{3,1,1} + s_{4,1}$ & 
$\begin{gathered}
2 s_{2,1,1,1} + 2 s_{2,2,1} + 3 s_{3,1,1} \\ + 3 s_{3,2} + 3 s_{4,1} + 2 s_{5}
\end{gathered}$
&
$\begin{gathered}
    3 s_{1,1,1,1,1} + 10 s_{2,1,1,1} + 14 s_{2,2,1} \\ + 17 s_{3,1,1}  + 15 s_{3,2}  + 14 s_{4,1} + 4 s_{5}
\end{gathered}$
\\ \hline
6 & $s_6$ & $s_{4,1,1} + s_{5,1}$ & 
\cellcolor{blue!15}$ \begin{gathered}
    s_{2,2,1,1} + 2 s_{3,1,1,1} \\ + 2 s_{3,2,1} + s_{3,3} + 3 s_{4,1,1} \\ + 3 s_{4,2}  + 3 s_{5,1} + 2 s_{6}
\end{gathered}$
&
$\begin{gathered}
    s_{1,1,1,1,1,1} + 5 s_{2,1,1,1,1} + 8 s_{2,2,1,1} \\ + 7 s_{2,2,2} + 13 s_{3,1,1,1} + 21 s_{3,2,1} + 7 s_{3,3}  \\ + 18 s_{4,1,1}  + 18 s_{4,2} + 14 s_{5,1} + 4 s_{6}
\end{gathered}$
\\ \hline
7 & $s_7$ & $s_{5,1,1} + s_{6,1}$ & 
$\begin{gathered}
    s_{3,2,1,1} + 2 s_{4,1,1,1} \\ + 2 s_{4,2,1}  + s_{4,3} + 3 s_{5,1,1} \\ + 3 s_{5,2}  + 3 s_{6,1} + 2 s_{7}
\end{gathered}$
&
$\begin{gathered}
    s_{2,1,1,1,1,1} + 2 s_{2,2,1,1,1} + 3 s_{2,2,2,1} \\ + 6 s_{3,1,1,1,1} + 11 s_{3,2,1,1} + 8 s_{3,2,2} \\ + 7 s_{3,3,1} + 13 s_{4,1,1,1} + 22 s_{4,2,1} + 10 s_{4,3} \\ + 18 s_{5,1,1} + 18 s_{5,2} + 14 s_{6,1} + 4 s_{7}
\end{gathered}$
\\ \hline
8 & $s_8$ & $s_{6,1,1} + s_{7,1}$ &
$\begin{gathered}
    s_{4,2,1,1} + 2 s_{5,1,1,1} \\ + 2 s_{5,2,1} + s_{5,3} \\ + 3 s_{6,1,1} + 3 s_{6,2} \\ + 3 s_{7,1} + 2 s_{8}
\end{gathered}$
&
$\begin{gathered}
    s_{2,2,2,2} + s_{3,1,1,1,1,1} + 3 s_{3,2,1,1,1} \\ + 3 s_{3,2,2,1} + 3 s_{3,3,1,1} + s_{3,3,2} \\ + 6 s_{4,1,1,1,1} + 11 s_{4,2,1,1} + 8 s_{4,2,2}  \\ + 8 s_{4,3,1} + 3 s_{4,4}+ 13 s_{5,1,1,1} \\ + 22 s_{5,2,1} + 10 s_{5,3} + 18 s_{6,1,1} \\ + 18 s_{6,2} + 14 s_{7,1} + 4 s_{8}
\end{gathered}
$
\\ \hline
9 & $s_9$ & $s_{7,1,1} + s_{8,1}$ &
$\begin{gathered}
    s_{5,2,1,1} + 2 s_{6,1,1,1} \\ + 2 s_{6,2,1} + s_{6,3} \\ + 3 s_{7,1,1} + 3 s_{7,2} \\ + 3 s_{8,1} + 2 s_{9}
\end{gathered}$
&
\cellcolor{blue!15}$\begin{gathered}
s_{3,2,2,2} + s_{3,3,1,1,1} + s_{4,1,1,1,1,1} \\ + 3 s_{4,2,1,1,1} + 3 s_{4,2,2,1} + 3 s_{4,3,1,1} \\  + s_{4,3,2} + s_{4,4,1}  + 6 s_{5,1,1,1,1} \\ + 11 s_{5,2,1,1} + 8 s_{5,2,2} + 8 s_{5,3,1} \\ + 3 s_{5,4} + 13 s_{6,1,1,1} + 22 s_{6,2,1} + 10 s_{6,3} \\+ 18 s_{7,1,1} + 18 s_{7,2} + 14 s_{8,1} + 4 s_{9}
\end{gathered}$
\\ \hline
10 & $s_{10}$ & $s_{8,1,1} + s_{9,1}$
&
$\begin{gathered}
    s_{6,2,1,1} + 2 s_{7,1,1,1} \\ + 2 s_{7,2,1} + s_{7,3} \\ + 3 s_{8,1,1} + 3 s_{8,2} \\ + 3 s_{9,1} + 2 s_{10}
\end{gathered}$
&
$\begin{gathered}
 s_{4,2,2,2} + s_{4,3,1,1,1} + s_{5,1,1,1,1,1} \\ + 3 s_{5,2,1,1,1} + 3 s_{5,2,2,1} + 3 s_{5,3,1,1} \\ + s_{5,3,2} + s_{5,4,1} + 6 s_{6,1,1,1,1} \\ + 11 s_{6,2,1,1} + 8 s_{6,2,2} + 8 s_{6,3,1} \\ + 3 s_{6,4} + 13 s_{7,1,1,1} + 22 s_{7,2,1} + 10 s_{7,3} \\ + 18 s_{8,1,1} + 18 s_{8,2} + 14 s_{9,1} + 4 s_{10}
\end{gathered}$
\\ \hline
\end{tabular}
}
\caption{Irreducibles for $[{\VG}(\Br_n)^!_i] = [\Snkrep_{\VG}((n-1)+i,n-1)]$. 
}\label{table: VG-shriek}
\end{table}
\newpage
\centering
\begin{sidewaystable}
\resizebox{\textwidth}{!}{
\begin{tabular}{|c||c|c|c|c|c|c|}
\hline
\diagbox[width=0.75cm]{$n$}{$i$} &1&2&3&4&5\\ \hline\hline
2 & $s_{2}$ & 0 & 0 & 0 & 0
\\ \hline
3 & $s_{2,1} + s_{3}$ & $s_{1,1,1}$ & $s_{2,1}$ & $s_{1,1,1} + s_{2,1}$ & $s_{1,1,1} + 2 s_{2,1} + s_{3}$ \\ \hline 
4 & \cellcolor{blue!15} $s_{2,2} + s_{3,1} + s_{4}$ & \cellcolor{blue!15} $s_{1,1,1,1} + s_{2,1,1}$ &  $s_{2,1,1} + 2 s_{2,2} + s_{3,1}$ & $2 s_{1,1,1,1} + 3 s_{2,1,1} + 2 s_{2,2} + 2 s_{3,1}$ & $3 s_{1,1,1,1} + 6 s_{2,1,1} + 6 s_{2,2} + 6 s_{3,1} + 3 s_{4}$
\\ \hline
5 & $s_{3,2} + s_{4,1} + s_{5}$ & $s_{2,1,1,1} + s_{3,1,1}$ & $
 \begin{gathered} 2 s_{2,2,1} + s_{3,1,1} \\ + 2 s_{3,2} + s_{4,1}\end{gathered}$ & $\begin{gathered} s_{1,1,1,1,1} + 3 s_{2,1,1,1} + 3 s_{2,2,1} \\ + 5 s_{3,1,1} + 3 s_{3,2} + 2 s_{4,1} \end{gathered}$ &
$\begin{gathered}
    2 s_{1,1,1,1,1} + 8 s_{2,1,1,1} + 11 s_{2,2,1} \\ + 11 s_{3,1,1} + 12 s_{3,2} + 10 s_{4,1} + 3 s_{5}
\end{gathered}$
\\ \hline
6 & $s_{4,2} + s_{5,1} + s_{6}$ & $s_{3,1,1,1} + s_{4,1,1}$ & \cellcolor{blue!15} $\begin{gathered} s_{2,2,2} + 2 s_{3,2,1} \\ + s_{4,1,1} + 2 s_{4,2} + s_{5,1}\end{gathered}$ & 
$\begin{gathered}
    s_{2,1,1,1,1} + s_{2,2,1,1} + s_{2,2,2} \\ + 4 s_{3,1,1,1} + 5 s_{3,2,1}  + s_{3,3} \\ + 5 s_{4,1,1} + 3 s_{4,2} + 2 s_{5,1}
\end{gathered}$
&
$\begin{gathered}
    4 s_{2,1,1,1,1} + 7 s_{2,2,1,1} + 7 s_{2,2,2} \\ + 11 s_{3,1,1,1} + 18 s_{3,2,1} + 6 s_{3,3} \\ + 13 s_{4,1,1} + 17 s_{4,2} + 10 s_{5,1} + 3 s_{6}
\end{gathered}$
\\ \hline
7 & $s_{5,2} + s_{6,1} + s_{7}$ & $s_{4,1,1,1} + s_{5,1,1}$ & 
$\begin{gathered}
    s_{3,2,2} + 2 s_{4,2,1} \\ + s_{5,1,1} + 2 s_{5,2} + s_{6,1}
\end{gathered}$
&
$\begin{gathered}
    s_{3,1,1,1,1} + 2 s_{3,2,1,1} + s_{3,2,2} \\ + 2 s_{3,3,1} + 4 s_{4,1,1,1} + 5 s_{4,2,1} \\ + s_{4,3} + 5 s_{5,1,1} + 3 s_{5,2} + 2 s_{6,1}
\end{gathered}$
& 
$\begin{gathered}
    2 s_{2,2,1,1,1} + 3 s_{2,2,2,1} + 5 s_{3,1,1,1,1} \\ + 10 s_{3,2,1,1} + 9 s_{3,2,2} + 7 s_{3,3,1} \\ + 11 s_{4,1,1,1} + 21 s_{4,2,1} + 11 s_{4,3} \\ + 13 s_{5,1,1} + 17 s_{5,2} + 10 s_{6,1} + 3 s_{7}
\end{gathered}$
\\ \hline
8 & $s_{6,2} + s_{7,1} + s_{8}$ &
$\begin{gathered}
    s_{5,1,1,1} + s_{6,1,1}
\end{gathered}$
&
$\begin{gathered}
    s_{4,2,2} + 2 s_{5,2,1} \\ + s_{6,1,1} + 2 s_{6,2} + s_{7,1}
\end{gathered}$
& \cellcolor{blue!15}
$\begin{gathered}
    s_{3,3,1,1} + s_{4,1,1,1,1} \\ + 2 s_{4,2,1,1} + s_{4,2,2} + 2 s_{4,3,1} \\ + 4 s_{5,1,1,1} + 5 s_{5,2,1} + s_{5,3} \\ + 5 s_{6,1,1} + 3 s_{6,2} + 2 s_{7,1}
\end{gathered}$
&
$\begin{gathered}
    s_{2,2,2,2} + 3 s_{3,2,1,1,1} + 3 s_{3,2,2,1} \\ + 3 s_{3,3,1,1} + 2 s_{3,3,2} + 5 s_{4,1,1,1,1} \\ + 10 s_{4,2,1,1} + 10 s_{4,2,2} + 10 s_{4,3,1} \\ + 5 s_{4,4} + 11 s_{5,1,1,1} + 21 s_{5,2,1} + 11 s_{5,3} \\ + 13 s_{6,1,1} + 17 s_{6,2} + 10 s_{7,1} + 3 s_{8}
\end{gathered}$
\\ \hline
9 & $s_{7,2} + s_{8,1} + s_{9}$ & $s_{6,1,1,1} + s_{7,1,1}$ & $\begin{gathered} s_{5,2,2} + 2 s_{6,2,1} \\ + s_{7,1,1} + 2 s_{7,2} + s_{8,1}\end{gathered}$ 
&
$\begin{gathered}
    s_{4,3,1,1} + s_{5,1,1,1,1} \\ + 2 s_{5,2,1,1} + s_{5,2,2} + 2 s_{5,3,1} \\ + 4 s_{6,1,1,1} + 5 s_{6,2,1} + s_{6,3} \\ + 5 s_{7,1,1} + 3 s_{7,2} + 2 s_{8,1}
\end{gathered}$
& 
$\begin{gathered}
    s_{3,2,2,2} + s_{3,3,1,1,1} + 3 s_{4,2,1,1,1} \\ + 3 s_{4,2,2,1} + 3 s_{4,3,1,1} + 3 s_{4,3,2} \\ + 3 s_{4,4,1} + 5 s_{5,1,1,1,1} + 10 s_{5,2,1,1} \\ + 10 s_{5,2,2} + 10 s_{5,3,1} + 5 s_{5,4} \\ + 11 s_{6,1,1,1} + 21 s_{6,2,1} + 11 s_{6,3} \\ + 13 s_{7,1,1} + 17 s_{7,2} + 10 s_{8,1} + 3 s_{9}
\end{gathered}$
\\ \hline
10 & $s_{8,2} + s_{9,1} + s_{10}$ & $s_{7,1,1,1} + s_{8,1,1}$ & $\begin{gathered}s_{6,2,2} + 2 s_{7,2,1} \\ + s_{8,1,1} + 2 s_{8,2} + s_{9,1}\end{gathered}$
&
$\begin{gathered}
    s_{5,3,1,1} + s_{6,1,1,1,1} \\ + 2 s_{6,2,1,1} + s_{6,2,2} + 2 s_{6,3,1} \\ + 4 s_{7,1,1,1} + 5 s_{7,2,1} + s_{7,3} \\ + 5 s_{8,1,1} + 3 s_{8,2} + 2 s_{9,1}
\end{gathered}$
& \cellcolor{blue!15}
$\begin{gathered}
    s_{4,2,2,2} + s_{4,3,1,1,1} + s_{4,4,2} \\ + 3 s_{5,2,1,1,1} + 3 s_{5,2,2,1} + 3 s_{5,3,1,1} \\ + 3 s_{5,3,2} + 3 s_{5,4,1} + 5 s_{6,1,1,1,1} \\ + 10 s_{6,2,1,1} + 10 s_{6,2,2} + 10 s_{6,3,1} \\ + 5 s_{6,4} + 11 s_{7,1,1,1} + 21 s_{7,2,1} + 11 s_{7,3} \\ + 13 s_{8,1,1} + 17 s_{8,2} + 10 s_{9,1} + 3 s_{10}
\end{gathered}$
\\ \hline
\end{tabular}
}
\caption{Irreducibles for $[\cL(n)_i]$ when $A = \OS(\braid_n)$. 
}\label{table: deviations-OS}
\end{sidewaystable} \newpage
\centering
\begin{sidewaystable}[ht]
\resizebox{\textwidth}{!}{
\begin{tabular}{|c||c|c|c|c|c|c|}
\hline
\diagbox[width=0.75cm]{$n$}{$i$} &1&2&3&4&5\\ \hline\hline
2 & $s_{1,1}$ & $s_2$ & 0 & 0 & 0
\\ \hline
3 & \cellcolor{blue!15} $s_{1,1,1} + s_{2,1}$ & $s_{2,1} + 2 s_{3}$ & $s_{2,1}$ & $s_{1,1,1} + s_{2,1}$ & $s_{1,1,1} + 2 s_{2,1} + s_{3}$
\\ \hline
4 & $s_{2,1,1} + s_{3,1}$ & $s_{2,2} + 2 s_{3,1} + 2 s_{4}$ & $s_{2,1,1} + 2 s_{2,2} + s_{3,1}$ & $\begin{gathered} s_{1,1,1,1} + 4 s_{2,1,1} + s_{2,2} + 2 s_{3,1}
\end{gathered}$ & $\begin{gathered} 2 s_{1,1,1,1} + 7 s_{2,1,1}  + 4 s_{2,2} + 7 s_{3,1} + 2 s_{4}\end{gathered}$
\\ \hline
5 & $s_{3,1,1} + s_{4,1}$ & $2 s_{3,2} + 2 s_{4,1} + 2 s_{5}$ & $\begin{gathered} 2 s_{2,2,1} + s_{3,1,1} \\ + 2 s_{3,2} + s_{4,1}
\end{gathered}$
& $\begin{gathered}
3 s_{2,1,1,1} + 3 s_{2,2,1} \\ + 6 s_{3,1,1} + 2 s_{3,2} + 2 s_{4,1}
\end{gathered}$
& $\begin{gathered}
    s_{1,1,1,1,1} + 8 s_{2,1,1,1} 
    + 10 s_{2,2,1} \\ + 13 s_{3,1,1}
    + 11 s_{3,2} + 10 s_{4,1} + 2 s_{5}
\end{gathered}$
\\ \hline
6 & $s_{4,1,1} + s_{5,1}$ & \cellcolor{blue!15} $\begin{gathered}
    s_{3,3} + 2 s_{4,2} \\ + 2 s_{5,1} + 2 s_{6}
\end{gathered}$
& \cellcolor{blue!15} $\begin{gathered}
    s_{2,2,2} + 2 s_{3,2,1}  \\ + s_{4,1,1} + 2 s_{4,2} + s_{5,1}
\end{gathered}$
& $\begin{gathered}
    2 s_{2,2,1,1} + 4 s_{3,1,1,1} \\ + 5 s_{3,2,1} + s_{3,3} \\ + 6 s_{4,1,1} + 2 s_{4,2} + 2 s_{5,1}
\end{gathered}$
& $\begin{gathered}
    3 s_{2,1,1,1,1} + 9 s_{2,2,1,1} + 4 s_{2,2,2} \\ + 10 s_{3,1,1,1} + 18 s_{3,2,1} + 8 s_{3,3} \\ + 16 s_{4,1,1} + 14 s_{4,2} + 10 s_{5,1} + 2 s_{6}
\end{gathered}$
\\ \hline
7 & $s_{5,1,1} + s_{6,1}$ & $\begin{gathered}
    s_{4,3} + 2 s_{5,2} \\ + 2 s_{6,1} + 2 s_{7}
\end{gathered}$
& $\begin{gathered}
    s_{3,2,2} + 2 s_{4,2,1} \\ + s_{5,1,1} + 2 s_{5,2} + s_{6,1}
\end{gathered}$
& $\begin{gathered}
    3 s_{3,2,1,1} + 2 s_{3,3,1} \\ + 4 s_{4,1,1,1} + 5 s_{4,2,1} + s_{4,3} \\ + 6 s_{5,1,1}  + 2 s_{5,2} + 2 s_{6,1}
\end{gathered}$
& $\begin{gathered}
    3 s_{2,2,1,1,1} + 3 s_{2,2,2,1}  + 3 s_{3,1,1,1,1} \\ + 11 s_{3,2,1,1} + 6 s_{3,2,2} + 9 s_{3,3,1} \\ + 11 s_{4,1,1,1} + 21 s_{4,2,1} + 11 s_{4,3}  \\+ 16 s_{5,1,1} + 14 s_{5,2} + 10 s_{6,1} + 2 s_{7}
\end{gathered}$
\\ \hline
8 & $s_{6,1,1} + s_{7,1}$ & $\begin{gathered} s_{5,3} + 2 s_{6,2} \\ + 2 s_{7,1} + 2 s_{8}\end{gathered}$ 
& $\begin{gathered} s_{4,2,2} + 2 s_{5,2,1} \\ + s_{6,1,1} + 2 s_{6,2} + s_{7,1} \end{gathered}$
& \cellcolor{blue!15} $\begin{gathered} s_{3,3,1,1} + 3 s_{4,2,1,1} \\ + 2 s_{4,3,1} + 4 s_{5,1,1,1} \\ + 5 s_{5,2,1}+ s_{5,3} \\ + 6 s_{6,1,1} + 2 s_{6,2} + 2 s_{7,1} \end{gathered} $
& $\begin{gathered} s_{2,2,2,1,1} + 3 s_{3,2,1,1,1} + 3 s_{3,2,2,1} \\ + 2 s_{3,3,1,1} + 3 s_{3,3,2}  + 3 s_{4,1,1,1,1} \\ + 12 s_{4,2,1,1} + 6 s_{4,2,2}  + 12 s_{4,3,1} \\ + 3 s_{4,4} + 11 s_{5,1,1,1} + 21 s_{5,2,1} + 11 s_{5,3} \\ + 16 s_{6,1,1} + 14 s_{6,2} + 10 s_{7,1} + 2 s_{8} \end{gathered} $
\\ \hline
9 & $s_{7,1,1} + s_{8,1}$ & $\begin{gathered} s_{6,3} + 2 s_{7,2} \\ + 2 s_{8,1} + 2 s_{9} \end{gathered}$
& $\begin{gathered} s_{5,2,2} + 2 s_{6,2,1} \\ + s_{7,1,1} + 2 s_{7,2} + s_{8,1} \end{gathered}$
& $\begin{gathered} s_{4,3,1,1} + 3 s_{5,2,1,1} \\ + 2 s_{5,3,1} + 4 s_{6,1,1,1} \\ + 5 s_{6,2,1} + s_{6,3} \\ + 6 s_{7,1,1} + 2 s_{7,2} + 2 s_{8,1} \end{gathered}$
& $\begin{gathered}
    s_{3,2,2,1,1} + s_{3,3,3} + 3 s_{4,2,1,1,1} \\ + 3 s_{4,2,2,1} + 3 s_{4,3,1,1} + 3 s_{4,3,2} \\ + 3 s_{4,4,1} + 3 s_{5,1,1,1,1} + 12 s_{5,2,1,1} \\ + 6 s_{5,2,2} + 12 s_{5,3,1} + 3 s_{5,4} \\ + 11 s_{6,1,1,1} + 21 s_{6,2,1} + 11 s_{6,3} \\ + 16 s_{7,1,1} + 14 s_{7,2} + 10 s_{8,1} + 2 s_{9}
\end{gathered}$
\\ \hline
10 & $s_{8,1,1} + s_{9,1}$ & $\begin{gathered} s_{7,3} + 2 s_{8,2} \\ + 2 s_{9,1} + 2 s_{10} \end{gathered}$ & 
$\begin{gathered}
    s_{6,2,2} + 2 s_{7,2,1} \\ + s_{8,1,1} + 2 s_{8,2} + s_{9,1}
    \end{gathered}
    $
& 
$ \begin{gathered} s_{5,3,1,1} + 3 s_{6,2,1,1} \\ + 2 s_{6,3,1} + 4 s_{7,1,1,1} \\ + 5 s_{7,2,1} + s_{7,3} \\ + 6 s_{8,1,1} + 2 s_{8,2} + 2 s_{9,1}\end{gathered}$
& \cellcolor{blue!15} $\begin{gathered}
    s_{4,2,2,1,1} + s_{4,3,3} + s_{4,4,1,1} \\ + 3 s_{5,2,1,1,1} + 3 s_{5,2,2,1} + 3 s_{5,3,1,1} \\ + 3 s_{5,3,2} + 3 s_{5,4,1} + 3 s_{6,1,1,1,1} \\ + 12 s_{6,2,1,1} + 6 s_{6,2,2} + 12 s_{6,3,1} \\ + 3 s_{6,4} + 11 s_{7,1,1,1} + 21 s_{7,2,1} + 11 s_{7,3} \\ + 16 s_{8,1,1} + 14 s_{8,2} + 10 s_{9,1} + 2 s_{10}
\end{gathered}$
\\ \hline
\end{tabular}
}
\caption{Irreducibles for $[\cL(n)_i]$ when $A = \VG(\braid_n)$.
}\label{table: deviations-VG}
\end{sidewaystable}
\clearpage

\section{Proof of Theorem~\ref{thm: OS-perm-reps}}\label{appendix: OS-VG-perm-reps}

The proof of Theorem~\ref{thm: OS-perm-reps}
and calculation of near-boundary cases for
$\OS(\braid_n)^!_i, \VG(\braid_n)^!_i$
employ a brute-force strategy, which we outline here, giving only
brief sketches of the arguments.

Note that since all $\kk\symm_n$-modules $U$
are self-contragredient, one has
$[U^*]=[U]$ in $R_\kk(\symm_n)$, and
so the defining recurrence \eqref{equivariant-recurrence-for-shrieks} for Koszul modules simplifies to this:
\begin{equation}
   \label{eqn:shriek-rep-recurrence}
   [A^!_d]=\sum_{i=1}^{d} (-1)^{i-1} [A_i] \cdot [(A^!_{d-i})]
   \end{equation}
This means that if one defines
\begin{align*}
    g_i&=\ch\,\OS(n)_i \text{ for }0\le i\le n-1,\\
    f_i&:=\ch\,\OS(n)^!_i=\Snkrep_{\OS}(n-1+i,n-1) \text{ for }i=0,1,2,\ldots,
\end{align*}
then, with $*$ below  denoting the internal (Kronecker) product of symmetric functions, \eqref{OS-and-VG-lambda} lets one sometimes compute explicit formulas for the $g_i$ in terms of the homogeneous symmetric functions $\{h_\lambda\}$, and \eqref{eqn:shriek-rep-recurrence}  gives a recurrence for $f_i$ in terms of $f_0,f_1,\ldots,f_{i-1}$:
\begin{equation}
\label{f-g-srhiek-recurrence}
f_i=\sum_{i=1}^d (-1)^{i-1} g_i * f_{d-i}
\end{equation}
In each of the cases below, we identify a small subset $T$ of partitions  of $n$ such that the linear span of $\{ h_\lambda:\lambda\in T \}$ contains the $f_i$.  Further manipulation then gives the results described in \Cref{thm: OS-perm-reps}, and the precise formulas below.

\subsection{Proof of Theorem~\ref{thm: OS-perm-reps} part (i).}
Corresponding to the Stirling number formulas 
\begin{align*}
    S(n-1,n-1)&=1,\\
    S(n,n-1)&=\binom{n}{2},
\end{align*}
one has the following result, implying Theorem~\ref{thm: OS-perm-reps} part (i).

\begin{prop} 
\label{prop:A!n1-S(n,n-1)}
For the cases $i=0,1$, one has
\begin{align}
\label{braid-shriek-degree-zero}
\ch\, \mathcal{S}_{\OS}(n-1,n-1) = \ch\, \mathcal{S}_{\VG}(n-1,n-1)&=h_n  \quad\text{ for }n \geq 1,\\
\label{braid-OS-shriek-degree-one}
\ch\, \mathcal{S}_{\OS}(n,n-1)&= h_2 h_{n-2} \quad \text{ for }n \geq 2,\\
\label{braid-VG-shriek-degree-one}
\ch\, \mathcal{S}_{\VG}(n,n-1) &= e_2 h_{n-2}
\quad \text{ for }n \geq 2.
\end{align}
\end{prop}
\begin{proof}
Equation~\eqref{braid-shriek-degree-zero} follows since $\OS(\braid_n)_0=\VG(\braid_n)_0=\kk$, carrying the trivial $\symm_n$ representation in either case.
For \eqref{braid-OS-shriek-degree-one}, \eqref{braid-VG-shriek-degree-one}, note that
\eqref{eqn:shriek-rep-recurrence} and \eqref{OS-and-VG-lambda} imply
\begin{align*}
[\OS(\braid_n)^!_1]
 &=[\OS(\braid_n)_1]
 =[\OS(\braid_n)_{(21^{n-2})}]
 = h_1[\pi_2]\cdot h_{n-2}[\pi_1] = h_2 h_{n-2},,\\
[\VG(\braid_n)^!_1]
 &=[\VG(\braid_n)_1]
 =[\VG(\braid_n)_{(21^{n-2})}]
 = h_1[\ell_2] \cdot h_{n-2}[\ell_1] = e_2 h_{n-2}. \qedhere
\end{align*}
\end{proof}

\subsection{Proof of Theorem~\ref{thm: OS-perm-reps} part (ii).}

Here we prove the curious fact that for $n\ge 7$, $\OS(\braid_n)^!_2=\mathcal{S}_{\OS}(n+1,n-1)$ is in fact \emph{half} of a permutation module.

For $n=7,8,9,10$, {\tt Sage} computation with the Burnside ring shows that $\mathcal{S}_{\OS}(n+1,n-1)$ is NOT a permutation module.  By running his Burnside solver on the first formula in Proposition~\ref{prop:A!n2OS-S_OS(n+1,n-1)} below with rational coefficients, Trevor Karn noticed positive half-integers in the data and conjectured that  two copies of 
$\mathcal{S}_{\OS}(n+1,n-1)$ together constitute  a permutation module.  

\begin{prop}\label{prop:A!n2OS-S_OS(n+1,n-1)} 
One has the following decompositions as permutation modules for $n\le 6$:
\begin{align*}
\ch\,\mathcal{S}_{\OS}(3,1)&=h_2,\\
\ch\,\mathcal{S}_{\OS}(4,2)&=h_3+h_1^3,\\
\ch\,\mathcal{S}_{\OS}(5,3)
&=h_2[h_1^2]+h_1^2 h_2+h_4,\\ 
\ch\,\mathcal{S}_{\OS}(6,4)&=h_1 \cdot h_2[h_2] + h_2(h_3+e_3) +h_2^2 h_1,\\  
\ch\,\mathcal{S}_{\OS}(7,5)&=h_2[h_2 h_1] + h_3^2 + h_4 h_1^2,
\end{align*}
and then for $n\ge 4$ one has
 \begin{align} \ch\,\OS(\braid_n)^!_2&=\ch\,\mathcal{S}_{\OS}(n+1,n-1)\notag\\
 &= h_{n-2}h_2+h_{n-3}h_1^3
+h_{n-3}h_3 +h_{n-4}h_2^2 +h_{n-4} h_4 \textcolor{red}{\mathbf{- h_{n-3} h_2 h_1 - h_{n-4} h_3 h_1}} \label{eqn:S-OS(n+1,n-1)-h}\\
& =h_{n-2}\, s_{(2)}+ h_{n-3}\, (s_{(1^3)} + s_{(2,1)}+ s_{(3)}) + h_{n-4} \,(s_{(2,2)}+s_{(4)}) \label{eqn:S-OS(n+1,n-1)-h-rep-stab-bound-is-8}\\
&=h_{n-2} h_2+\textcolor{blue}{\frac{1}{2}} h_{n-3}h_1^3+ \textcolor{blue}{\frac{1}{2}} h_{n-3}(h_3+e_3) +  h_{n-4} \cdot h_2[h_2]. \label{eqn:S-OS(n+1,n-1)-perm}
 \end{align} 
So two copies of 
$\OS(\braid_n)^!_2=\mathcal{S}_{\OS}(n+1,n-1)$ together form a permutation module, with orbit stabilisers 
$$
\{ \symm_{n-2}\times S_2 ,  \,\, \symm_{n-3} , \,\,  \symm_{n-3}\times C_3, \,\, \symm_{n-4}\times I_2(4)
\}
$$
where $C_n$ is the cyclic group of order $n$ (generated by the $n$-cycle $(1,2,\ldots,n)$ in $\symm_n$) and  $I_2(n)$ is the dihedral group of order $2n$ inside $\symm_{n}$ containing that same $n$-cycle.
\end{prop}

\begin{proof}[Sketch of proof.] 
The expansion for $n=2$ is clear.  For $n \geq 3$, writing $f_i,g_i$ as in \eqref{f-g-srhiek-recurrence}, one finds that 
\begin{align*}
f_1&=g_1,\\
f_2&=f_1*g_1-g_2.
\end{align*}
Using \eqref{OS-and-VG-lambda} and writing $\delta(S) \in \{0,1\}$ depending on whether statement $S$ is false or true, one has
\begin{align*}
f_1=g_1&=\ch\,A(n)_1=\ch\,\OS_{(2,1^{n-2})} =h_{n-2}\pi_2=h_{n-2}h_2,\\
g_2&=h_{n-3}(h_2 h_1-h_3)\cdot \delta_{n\ge 3} 
+h_{n-4} (h_3 h_1-h_4) \cdot \delta_{n\ge 4}.
\end{align*}
Using the standard  fact that $U\otimes (V\big\uparrow_H^G)\, \cong (U\big\downarrow_H\otimes V)\big\uparrow_H^G, $ for the Young subgroup $H=\symm_2\times \symm_{n-2},$ and the {\it skewing operators} $s_{(2)}^\perp$ and $s_{(1^2)}^\perp$ as defined in Macdonald \cite[Ex. I.5.3]{Macdonald}, the expression~\eqref{eqn:S-OS(n+1,n-1)-h}  follows by routine manipulation. Then to establish~\eqref{eqn:S-OS(n+1,n-1)-perm}, we 
use these facts:
\begin{align*}
    h_1^3+2 h_3-2 h_2 h_1=h_3+e_3&=\ch\, 1\big\uparrow_{C_3}^{\symm_3},\\
    h_2[h_2]=h_4+s_{(2,2)}=h_4+h_2^2-h_3h_1&=\ch\, 1\big\uparrow_{I_2(4)}^{\symm_4}.
    \qedhere
    \end{align*}
\end{proof}

\begin{remark}
\label{rem:A!n2VG-S_VG(n+1,n-1)} 
A similar analysis gives  the following for $\VG(\Br_n^!)_2$:
\begin{align}
\ch\, \mathcal{S}_\VG(n+1,n-1)&=h_{n-2} h_2+ h_{n-3}(h_3-h_2h_1+h_1^3)+h_{n-4}(h_4+h_2 h_1^2-h_3 h_1-h_2^2),\, n\ge 4\notag\\
&=h_{n-2}\, s_{(2)}+ h_{n-3}\, (s_{(1^3)} + s_{(2,1)}+ s_{(3)}) + h_{n-4} \,s_{(2,1,1)}.\label{eqn:S-VG(n+1,n-1)-h-rep-stab-bound-is-6}
\end{align}
Also $\ch\, \mathcal{S}_\VG(n+1,n-1)=h_1^n+h_n$ for $n=3,4, $ and hence 
 $\mathcal{S}_\VG(n+1,n-1)$ is a permutation module for $n\le 4$; it is half a permutation module for $n=5$.  However for $n=6,7,$ the Burnside solver shows that it is not a permutation module (even after  scaling), even though all character values are nonnegative.
At $n=8$ there are negative character values, so even scaling will not result in a permutation module.
\end{remark}

\subsection{Proof of Theorem~\ref{thm: OS-perm-reps} part (iii).}

For fixed small $k$, the general Stirling number formula 
\begin{equation}
\label{inclusion-exclusion-formula}
S(n,k) ={\frac {1}{k!}}
\sum _{i=0}^{k}(-1)^{i}{\binom {k}{i}}(k-i)^{n}
\end{equation}
gives fairly simple explicit formulas for $S(n,k)$  as a function of $n$, e.g., for $k=1,2, 3,4,5$:
\begin{align}
\label{explicit-S(n,1)-formula}
S(n,1) &= 1 \\
        S(n,2) &= \frac{1}{2}(2^n - 2 \cdot 1^n) = 2^{n-1}-1 = 1+ 2+ 2^2 + \cdots + 2^{n-2},
        \label{explicit-S(n,2)-formula}\\ 
        S(n,3) &= \frac{1}{6}(3^n - 3 \cdot 2^n + 3 )\label{explicit-S(n,3)-formula}\\ 
        S(n,4) &= \frac{1}{24}(4^n - 4 \cdot 3^n + 6 \cdot 2^n - 4 )
        \label{explicit-S(n,4)-formula}\\
        S(n,5) &= \frac{1}{120}(5^n - 5 \cdot 4^n + 10 \cdot 3^n - 10 \cdot 2^n + 5 )
        \label{explicit-S(n,5)-formula}
\end{align}
We give here analogous descriptions of the $\kk\symm_n$-modules $\OS(\braid_n)^!_i, \VG(\braid_n)^!_i$, having dimension $S(n-1+i,n-1)$,
starting\footnote{There is little to say for $n=1$, as $\symm_1$ is the trivial group, and
$
\OS(\braid_1)=\VG(\braid_1)=\kk=\OS(\braid_1)^!=\VG(\braid_1)^!,
$
and 
$
\ch\, \Snkrep_\OS(0,0)=\ch\, \Snkrep_\VG(0,0)=h_1.
$} with $n=2,3$. 

\begin{prop}
    \label{n-at-most-3-braid-shrieks}
    The Frobenius characteristics of the $\kk\symm_n$-modules \begin{align*}
    \OS(\braid_n)^!_i&=\Snkrep_\OS(n-1+i,n-1),\\ \VG(\braid_n)^!_i&=\Snkrep_\VG(n-1+i,n-1)
    \end{align*}
    for $n=2,3$ have these formulas:
\begin{enumerate}
\item[$n=2:$] 
\begin{itemize}
\item[] $\ch\, \Snkrep_\OS(i+1,1) =h_2,$
\item[] $\ch\, \Snkrep_\VG(i+1,1) =
 \begin{cases} h_2, &i \text{ even} ,\\
      e_2, &i \text{ odd}.
    \end{cases}$
\end{itemize}
\item[$n=3:$]
\begin{itemize}
    \item[] 
    $\ch\, \mathcal{S}_{OS}(i+2,2)= \ch\, \mathcal{S}_{VG}(i+2,2) 
    =  \frac{2^i-1}{3} h_1^3+h_3$, if $i$ even,
    \item[] 
    $\ch\, \mathcal{S}_{OS}(i+2,2) =  \frac{2(2^{i-1}-1)}{3} h_1^3+h_1h_2$, if $i$ odd,
     \item[]
     $\ch\, \mathcal{S}_{VG}(i+2,2)=\frac{2(2^{i-1}-1)}{3} h_1^3+h_1e_2,$ if $i$ odd.
\end{itemize}
\end{enumerate}
In particular,
\begin{itemize}
    \item  $\OS(\braid_2)^!_i,\OS(\braid_3)^!_i,$  are permutation modules, while 
    \item $\VG(\braid_2)^!_i, \VG(\braid_3)^!_i$ are  permutation modules for all \emph{even} $i$, and 
    \item When $n=2,3$ one has
    $$
    \OS(\braid_n)^!_i \cong 
    \begin{cases}
        \VG(\braid_n)^!_i &\text{ if }i\text{ is even},\\
              \mathrm{sgn}_n \otimes \VG(\braid_n)^!_i &\text{ if }i\text{ is odd}.\\
    \end{cases}
    $$
\end{itemize}
\end{prop}
\begin{proof}[Sketch of proof.]
These all follow by induction on $i$ via the recurrences \eqref{eqn:shriek-rep-recurrence} and \eqref{OS-and-VG-lambda}.
\end{proof}

\begin{remark}
\label{S(n,2)-as-almost-multiples-reg-rep}

Note the expressions for $n=2$ are
consistent with the formula $S(i+1,1)=1$ coming from \eqref{explicit-S(n,1)-formula}.
We claim that the expressions for $n=3$ are also
consistent with the formulas 
\begin{align}
\label{S(n,2)-formula-rewritten-1}
S(i+2,2)&= 2^{i+1}-1 \\
\label{S(n,2)-formula-rewritten-2}
&= 1 + 2+ 2^2+ \cdots+2^i
\end{align}
coming from \eqref{explicit-S(n,2)-formula}, which we illustrate here
for $\OS(\braid_3)_i=\Snkrep(i+2,2)$.
One can rewrite it as
\begin{align*}
\ch\,\,\OS(\braid_3)_i=\ch\, \mathcal{S}_{\OS}(i+2,2)
&=
\underbrace{\left(
\frac{2^{i+1}+(-1)^i}{6} - \frac{1}{2}
\right)}_{\text{call this }c_i}
\cdot h_1^3 
+ \begin{cases}
h_3 &\text{ if }i\text{ is even},\\
h_2 h_1 &\text{ if }i\text{ is odd}
\end{cases}\\
&=
\frac{2^{i+1}}{6} 
\cdot h_1^3 
+ \begin{cases}
h_3 -\frac{1}{3} h_1^3 &\text{ if }i\text{ is even},\\
h_2 h_1 -\frac{2}{3} h_1^3&\text{ if }i\text{ is odd}.
\end{cases}
\end{align*}
Since $h_1^3, h_2 h_1, h_3$
correspond to $\kk\symm_3$-modules of dimensions $6,3,1$, one can check
that this last formula
lifts \eqref{S(n,2)-formula-rewritten-1}. 
Interestingly, the number $c_i$ of copies of the regular representation here (that is,
the coefficient of $h_1^3$) gives
a sequence $0,0,1,2,5,10,21,42,85,\ldots$
for which every other term $0,1,5,21,85,341,\ldots$ appears in the Online Encyclopedia of Integer Sequences as {\tt OEIS A002450}.

Expressions lifting \eqref{S(n,2)-formula-rewritten-2} arise when one uses the
recurrences \eqref{eqn:shriek-rep-recurrence} and \eqref{OS-and-VG-lambda},
without trying to rewrite things in terms of $h_\lambda$.  Recalling that
$\specht^\lambda$ is the irreducible
$\kk\symm_n$-module indexed by $\lambda$,
with $\ch \,\, \specht^\lambda=s_\lambda$, a {\it Schur function}.  One can check that
these recurrences give
\begin{align}  
\label{OS-geometric-series-lift}
\ch\,\mathcal{S}_\OS(i+2,2) 
&=h_3+s_{(2,1)}+s_{(2,1)}^{*2}+\cdots+s_{(2,1)}^{*i}.\\
\label{VG-geometric-series-lift}
\ch\,\mathcal{S}_\VG(i+2,2) 
&=\omega^i(h_3)+s_{(2,1)}+s_{(2,1)}^{*2}+\cdots +s_{(2,1)}^{*i},
\end{align}
where $\omega: \Lambda \rightarrow \Lambda$
is the involution on symmetric functions
swapping $h_n \leftrightarrow e_n$ for $n\ge 1$, corresponding to tensoring
$\kk \symm_n$-modules by the sign character
$\mathrm{sgn}_n$.  Bearing in mind that $h_3, e_3$ correspond to $1$-dimensional modules, while $s_{(2,1)}$ corresponds to the $2$-dimensional reflection representation $\specht^{(2,1)}$ of $\symm_3$, one sees that \eqref{OS-geometric-series-lift}, \eqref{VG-geometric-series-lift} lift \eqref{S(n,2)-formula-rewritten-2}.
Note also that \eqref{OS-geometric-series-lift} is consistent with the $n=3$ case of \eqref{rank-two-OS-shriek-as-geometric-series},
since one has a matroid isomorphism $\braid_3 \cong U_{2,3}$. 
\end{remark}

\subsection{The cases $\OS(\braid_4)^!$ and $\VG(\braid_4)^!$}

\phantom{blank line}

Here we show that the $\symm_4$-modules $\Snkrep_{\OS}(n+3,3)$ are permutation modules.
One observes a  periodicity in the initial expressions for $f_n=\ch\,\cS_{\OS}(n+3,3)$ below.

\begin{equation}\label{eqn:initial-A!4}
\begin{split}
&f_0=h_{4}\\
&f_1=  h_2^2\\
& f_2=h_1^4 \textcolor{red}{\mathbf{-h_2 h_1^2}} +2 h_2^2 +h_4 \quad\quad = h_2[h_1^2]+h_1^2 h_2+h_4\\
&f_3= 4 h_1^4 \textcolor{red}{\mathbf{-3 h_2 h_1^2}} +5 h_2^2 \quad\qquad\ = 2h_1^4+2 h_2[h_1^2] +h_1^2 h_2 +h_2^2\\
& f_4=14 h_1^4 \textcolor{red}{\mathbf{- 8 h_2 h_1^2}} +10 h_2^2 +h_4= 10 h_1^4 + 4 h_2[h_1^2] +2 h_2^2 +h_4\\
& f_5= 44 h_1^4  \textcolor{red}{\mathbf{-18 h_2 h_1^2}} +21 h_2^2 \qquad=35 h_1^4 +9 h_2[h_1^2] +3 h_2^2\\
&f_6= 135 h_1^4  \textcolor{red}{\mathbf{- 39 h_2 h_1^2}} +42 h_2^2 +h_4=115 h_1^4+20 h_2 [h_1^2]+2 h_2^2+h_2 h_1^2+h_4\\
&f_7=408 h_1^4  \textcolor{red}{\mathbf{- 81 h_2 h_1^2}} +85 h_2^2 
\qquad=367 h_1^4+41 h_2[h_1^2]+3 h_2^2 +h_2 h_1^2
\end{split}
\end{equation}

\begin{prop}\label{prop:A!4} The Frobenius characteristic $\ch\,\mathcal{S}_{OS}(n+3,3)=\ch\,\OS^!(\Br_4)_n$ is 
 an integer combination of $\{h_1^4, h_1^2h_2, h_2^2, h_4\}$.  

Let  $\mathbf{\ch\,\mathcal{S}_{OS}(n+3,3)=a_n h_1^4+b_n h_1^2 h_2 +c'_n h_2^2 +d_n h_4}, \, n\ge 0$.  Let $G_2$ be the subgroup of order 2 generated by $(12)(34)$.Then 
$\mathcal{S}_{OS}(n+3,3)$ is a permutation module with orbit stabilisers consisting of the wreath product $\symm_2[G_2]$, as well as a subset of the Young subgroups $\symm_\lambda$,  $\lambda\in \{(1^4), (2,1^2), (2^2), (4) \}$.  %
We have, for 
$a_n, c'_n, d_n\ge 0$ and $b_n<0$, 
\begin{align}
&\ch\, \mathcal{S}_{OS}(n+3,3) \notag\\
&=(a_n+\frac{b_n}{2}) h_1^4 -\frac{b_n}{2} h_2[h_1^2] +(c'_n+b_n) h_2^2 +d_n h_4, \ n\equiv 0, 1\bmod 4, 
\label{eqn:A!(4)-perm1}\\
&=(a_n+\frac{b_n-1}{2}) h_1^4 -\frac{b_n-1}{2} h_2[h_1^2] +h_1^2h_2+(c'_n+b_n-1) h_2^2 +d_n h_4,\ 
n\equiv 2, 3\bmod 4.\label{eqn:A!(4)-perm2}
\end{align}
The coefficients $a_n, b_n, c'_n, d_n$ are determined below.  

\begin{enumerate}
\item The coefficients of $h_4$ are the sequence $d_n=\frac{1+(-1)^n}{2}, \, n\ge 0.$
\item The coefficients of $h_2^2$ are $\{0,0,1,2,5,10,21,\ldots\}$, i.e. the numbers $c_n$ from Remark~\ref{S(n,2)-as-almost-multiples-reg-rep}.  More precisely, $c'_n=c_{n+3}=\frac{2^{n+2}-3+(-1)^{n+3}}{6},\, n\ge 0.$

\item $b_n-b_{n-1}=-c_{n+2}=-\frac{2^{n+1}-3+(-1)^n}{6}, n\ge 1,$ with $b_0=0=b_1$. In particular, $b_n<0$ for $n\ge 2$.  We have  
\[ b_n= -\frac{2}{3}(2^n-1) +\frac{n}{2}+\frac{1}{12}(1-(-1)^n)=-\frac{2^{n+1}}{3}+\frac{n}{2}+\frac{3}{4} -\frac{1}{12}(-1)^n, \, n\ge 0.\]
Thus $\mathcal{S}_{OS}(n+3,3)$ is NOT $h$-positive for $n\ge 2$.

\item The coefficients $a_n$ are all nonnegative, and strictly positive if $n\ge 2$. 
We have $a_0=a_1=0$ and for $n\ge 2$, 
\[a_n=\frac{3^{n+1}}{16}-\frac{n}{4}-\frac{4-(-1)^n}{16}\]
In particular, $a_n+\frac{b_n}{2}$, and $a_n+\frac{b_n-1}{2}$ are positive integers for $n\ge 3$.

Also, $a_n$ is the multiplicity of the sign representation.

\end{enumerate} 
\end{prop}

\begin{remark}\label{rem:h-coeffs-oeis-OS4shriek}
Computing dimensions shows that the $h$-expansion of $\ch\,\Snkrep(n+3,3)$ lifts the formula~\eqref{explicit-S(n,3)-formula} for the Stirling number $S(n+3,3)$. 

    The coefficient $a_n$ of $h_1^4$ gives the sequence $0,0,1,4,14,44,135,408, 1228,\ldots$, appearing as {\tt OEIS A097137}.
    (One checks that $a_n-a_{n-2}= (3^{n-1}-1)/2$.) Also  the negative of the coefficient $b_n$ of $h_1^2h_2$ gives $0,0,1,3,8,18, 39, 81, 166, 336, 677, \ldots $, which is {\tt OEIS A011377} or {\tt OEIS A178420}.
\end{remark}
A similar analysis for $\VG(\braid_4)_i^!$ shows that its  Frobenius characteristic 
$\ch \, \Snkrep_{\VG}(3+i,3)$ is also 
 an integer combination of $\{h_1^4, h_1^2h_2, h_2^2, h_4\}$, in fact of $\{h_1^4, h_2 e_2, h_4\}$. 

Here is the data for $f_i=\ch\,\VG^!(4)_i$ with $0\le i\le 11$:
\begin{align*}
f_0&=h_4 &   f_1&=h_2 e_2\\
f_2&= h_1^4+h_4   & f_3&=4 \, h_1^4- \, h_2 e_2 \\
f_4&=12 \, h_1^4 +2 \, h_2 e_2 +h_4  & f_5&=40 \, h_1^4 +\, h_2 e_2\\
f_6&=127 \, h_1^4 -4 \, h_2 e_2 +h_4   &  f_7&= 388 \, h_1^4 +3 \, h_2 e_2\\
f_8&= 1186 \, h_1^4 + 6 \, h_2e_2  + h_4 &  f_9&=3608 \, h_1^4 - 11 \, h_2 e_2\\
f_{10}&=10901 \, h_1^4 + h_4  &  f_{11}&=32868 \, h_1^4+ 23 \, h_2 e_2\\
%
%
\end{align*}
Observe that the set $\{h_1^4, h_2e_2 ,  h_4\}$ is linearly independent.  
 One then has the following more precise statement:
 \begin{prop}\label{prop:VGBr4shriek-perm}
     Write $f_n$ for $\ch\,\mathcal{S}_{VG}(n+3,3)=\ch\,\VG^!(\Br_4)_n$.  Then 
$f_{2n-1}, f_{2n}-h_4\in \ZZ[h_1^4, h_2 e_2]$ 
and hence for $n \geq 0$, both the representation $\mathcal{S}_{VG}(2n+2,3)$ and the quotient representation $\mathcal{S}_{VG}(2n+3,3) / \one_{\symm_4}$ are fixed under tensoring
with the sign representation $\mathrm{sgn}$ of $\symm_4$.

Let $f_n=a_n h_1^4 +b_n h_2 e_2 +d_n h_4.$ Then, with initial values $a_0=a_1=0, a_2=1, a_3=4$, $b_0=0, b_1=1, b_2=0, b_3=-1$,   one has that $d_n=\frac{1+(-1)^n}{2}$, $a_n\ge 0$ for all $n\ge 0$, and  for $n\ge 3$:
\begin{align*} 
a_n&= 6 a_{n-1}-11 a_{n-2} +6 a_{n-3} +2(b_{n-1} -b_{n-2} +b_{n-3}) -d_{n-2},\\
b_n 
&=-b_{n-2}+2 b_{n-3}.
\end{align*}

The sequence $\{b_n\}_{n\ge 0}$ appears in {\tt OEIS A077912}, with generating function $\frac{x}{1+x^2-2x^3}$.

Moreover  $\mathcal{S}_{VG}(n+3,3)$ is a permutation module if and only if $b_n=0$ or $b_n\le -2$.  Write $\ -b_n=2\alpha_n+3\beta_n$ for nonnegative integers $\alpha_n, \beta_n$. Then $a_n-(\alpha_n+\beta_n)$ is nonnegative and 
\[f_n=(a_n-(\alpha_n+\beta_n) )h_1^4 + \alpha_n\, \ch\, (1\big\uparrow_{G_2}^{\symm_4}) +\beta_n\,  \ch\, (1\big \uparrow_{V_4}^{\symm_4}) + d_n h_4\]
is the Frobenius characteristic of a permutation module, 
where the orbit stabilisers are $\symm_1$, $\symm_4$ and the subgroups $G_2=\langle (12)(34)\rangle$ and $V_4=\{(1), (12)(34), (13)(24), (14) (23)\}$ of $\symm_4$.
\end{prop}

\subsection{The case $\OS(\braid_5)^!$}

\phantom{blank line}

In this section we show that the $\symm_5$-modules $\cS_{\OS}(n+4,4)$ are also permutation modules.  We also show that the $h$-expansions exhibit a curious  periodicity modulo 4. 

The initial expressions for $f_n=\ch\,\cS_{\OS}(n+4,4)$ are as follows.
\begin{equation}\label{eqn:initial-A!5}
    \begin{split}
f_0&=h_5,\ f_1=h_3h_2,\\
f_2&=h_4h_1+h_2h_1^3+2h_3h_2\textcolor{red}{\mathbf{-h_3h_1^2}}=h_1 h_2[h_2]+h_2(h_3+e_3) +h_2^2 h_1,\\
f_3&=2h_1^5+3h_2^2h_1+2h_3h_2,\\ 
f_4&=12 h_1^5+8h_2^2h_1+2 h_3 h_2+h_5\\
f_5&=60 h_1^5 +18 h_2^2h_1+3 h_3 h_2\\
f_6&=274 h_1^5 + 38 h_2^2 h_1 + h_2 h_1^3  \textcolor{red}{\mathbf{ - h_3 h_1^2}} + 4 h_3 h_2+ h_4 h_1\\
&=(274 h_1^5  + 38 h_2 h_1^2 +2  h_3 h_2) + f_2 \\
f_7&= 1194 h_1^5+ 81 h_2^2 h_1 + 4 h_3 h_2.
    \end{split}
\end{equation}

\begin{prop}\label{prop:A!5} The $\symm_5$-module $\cS_{\OS}(n+4,4)=\ch\,\OS^!(\Br_5)_n$
is a permutation module for all $n\ge 0$, with orbit stabilisers given by 
\begin{itemize}
\item
the Young subgroups $\symm_\lambda$ for $\lambda\in\{(1^5), (2^2,1), (3,2)\}$ if $\mathbf{n\equiv 1, 3 \bmod 4}.$
\item 
the Young subgroup $\symm_{ (2^2,1)}$, as well as  the subgroups 
$\symm_1\times I_2(4), A_3\times \symm_2$ if $\mathbf{n\equiv 2 \bmod 4}.$

Here $A_3\times \symm_2$ is the subgroup of the Young subgroup $\symm_3\times \symm_2$, for  the alternating subgroup $A_3$ of $\symm_3$.
\item
the Young subgroups $\symm_\lambda$ for 
$\lambda\in\{(1^5), (2^2,1), (3,2), (5)\}$ if $\mathbf{n\equiv 0 \bmod 4}.$
\end{itemize} 

Let $J=\{h_1^5, h_2^2h_1, h_3h_2\}$.  Let  
$f_n=\ch\,\mathcal{S}_{OS}(n+4,4)=\ch\, A^!(5)_n$. 
Then 
\begin{enumerate}
\item
$f_n$ is a nonnegative integer combination of the set $J$ if $n\equiv 1, 3 \bmod 4.$
\item 
$f_n-f_2$ is a nonnegative integer combination of $J$ if $n\equiv 2 \bmod 4.$
\item
$f_n-f_0$ is a nonnegative integer combination of $J$ if $n\equiv 0 \bmod 4.$
\end{enumerate}

The following explicit decomposition holds for $f_{n+4}-f_n$:
\begin{equation}\label{eqn:h-expansion-A!5}
f_{n+4}-f_n=a_n h_1^5 +b_n h_1 h_2^2 + 2 h_2 h_3,\ n\ge 0,
\end{equation} 
where $b_0=8$,  $b_n=10(2^n)-2,\, n\ge 1,$ and 
\begin{equation}\label{eqn:coeff-formula-an}
a_n=\frac{1}{3} (1+17\cdot 4^{n+1}-3\cdot 2^{n+1}-3^{n+3}).
\end{equation}
\end{prop}

Let $0\le i\le 3$ and $k\ge 0$. Then 
\begin{equation}\label{eqn:Ashriek5OS-complete} 
\begin{split}
&f_{4k+4+i}-f_i=
\alpha_{k,i} h_1^5 + \beta_{k,i} h_2^2h_1+ 2(k+1) h_2 h_3\\
&\text{where}\\
&\alpha_{k,i}= \frac{k+1}{3}+4^{i+1} \frac{256^{k+1}-1}{45} -3^{i+2} \frac{81^{k+1}-1}{80} -2^{i+1}\frac{16^{k+1}-1}{15},\\
&\beta_{k,i}=2^{i+1}\frac{16^{k+1}-1}{3}-2(k+1).
\end{split}    
\end{equation}

\noindent
The multiplicity of the sign representation in $f_n$ is 
$\begin{cases} \alpha_{k,i}, & n=4(k+1)+i \text{ and } k\ge 0,\\
2, & n=3,\\
0, & n<3. 
\end{cases}$

\begin{remark}\label{rem:restrictions} ({\it The restriction of $\Snkrep_{\OS}(n+1,n-1)$ and $\Snkrep_{\VG}(n+1,n-1)$ to $\symm_{n-1}$})

Observe that in each of the cases $\OS(\braid_n)$, $1\le n\le 5$, the restriction of the $\symm_n$-module to $\symm_{n-1}$ is always an $h$-positive permutation module. 
The restriction is not $h$-positive for $\Snkrep(n+1,n-1)$ when $n\ge 5$, although the following formula  shows that it is a permutation module.  
\[\ch \, \Snkrep(n+1,n-1)\downarrow^{\symm_n}_{\symm_{n-1}}=(h_{n-2}h_1+h_{n-3}h_2 +2 h_{n-3} h_1^2 + h_{n-4} h_1^3 +h_{n-4} h_3) \delta_{n\ge 4}  +h_{n-5}\,\ch \,\one\uparrow_{I_2(4)}^{\symm_4}\delta_{n\ge 5}.\]
Here $I_2(4)$ is the dihedral group of order 8.
\end{remark}

\begin{remark} ({\it The restriction of $\Snkrep_{\OS}(n+3, 3)$ and $\Snkrep_{\VG}(n+3, 3)$}) 
With the coefficients defined in Proposition~\ref{prop:A!4}, the restriction of $\mathcal{S}_{OS}(n+3,3)$ to $\symm_3$ has Frobenius characteristic \[(4a_n+b_n) h_1^3+2(b_n+c'_n) h_1h_2+ d_n h_3,\] 
 and is thus $h$-positive.  In particular 
 $\Snkrep_{\OS}(n+3, 3)\big\downarrow_{\symm_{3}}$ is a permutation module whose point stabilisers are Young subgroups. 

 Proposition~\ref{prop:VGBr4shriek-perm} shows that a similar statement holds for $\Snkrep_{\VG}(n+3, 3)\big\downarrow_{\symm_{3}}$ ; here the orbit stabilisers are $\symm_1$ and $\symm_3$.
\end{remark}
\begin{remark} ({\it The restriction of $\Snkrep_{\OS}(n+4, 4)$})
The restriction of $f_n=\ch\,  \cS_{\OS}(n+4,4)$ to $\symm_4$ is $h$-positive, 
 supported on the set $\{h_1^4, h_1^2 h_2, h_2^2\}$ if $n\equiv 1,3 \bmod 4$,  the set  $\{h_1^4, h_1^2 h_2, h_2^2, h_4\}$ if $n\equiv 0 \bmod 4$, and finally the set 
$\{h_1^4, h_1^2 h_2, h_2^2, h_3 h_1, h_4\}$ if $n\equiv 2\bmod 4$.  In particular 
 $\Snkrep_{\OS}(n+4, 4)\big\downarrow_{\symm_{4}}$ is a permutation module whose point stabilisers are Young subgroups. 
\end{remark}
\begin{remark}\label{rem:triv-rep} ({\it The multiplicity of the trivial representation})  Here we collect formulas for the multiplicity of the trivial representation:

For $\Snkrep_{\OS}(n+1,n-1)$, the multiplicity of the trivial representation is 3 for $n\ge 4$, and the multiplicity of the sign representation is 0 for $n\ne 3,4$, and 1 otherwise.

For $\mathcal{S}_{OS}(n+2,2)$, the multiplicity is \[\frac{ 2^{n+1}}{6} +\frac{3+(-1)^n}{6}.\]

For $\Snkrep_{\OS}(n+3, 3)$, the multiplicity of the trivial representation is  
\[\frac{3^{n+1}}{16}+\frac{n}{4}+\frac{8+5(-1)^n}{16},\]
giving $\{1,1,3,6,17,47,139,412,\ldots  \}$.

For $\Snkrep_{\OS}(n+4, 4)$, (with  definitions as in Proposition~\ref{prop:A!5}), the multiplicity of the trivial representation is  
\begin{equation*}\label{eqn:Ashriek5OS-complete-trivial} 
\frac{k+1}{3}+4^{i+1} \frac{256^{k+1}-1}{45} -3^{i+2} \frac{81^{k+1}-1}{80} +2^{i+3}\frac{16^{k+1}-1}{5}
+\langle f_i, h_5\rangle.    
\end{equation*}
\end{remark}

\bibliographystyle{amsplain}
\bibliography{biblio}
\addcontentsline{toc}{section}{Bibliography}

\end{document}